%% file: main.tex
\RequirePackage{fix-cm}
\documentclass[smallextended]{svjour3}       
\smartqed  
\renewcommand{\smartqed}{\tag*{$\square$}}

\usepackage{graphicx}
\usepackage{diagbox}
\usepackage{multirow}
\usepackage{multicol}
\usepackage{subfiles}
\usepackage{amsmath, amssymb, mathtools} 
\usepackage{stmaryrd, accents}
\usepackage{xcolor}

\definecolor{ao(english)}{rgb}{0.0, 0.5, 0.0}

\usepackage{hyperref}
\hypersetup{
    colorlinks=true,
    linkcolor=blue,
    filecolor=magenta,
    urlcolor= cyan,
    citecolor=ao(english)
}
\usepackage[inline]{enumitem}
\usepackage{tabularx}
\usepackage{empheq} 
\usepackage{tikz}
\usetikzlibrary{automata, positioning}

\usepackage[caption=false]{subfig}

\usepackage[toc,page]{appendix}

\usepackage{pifont}
\usepackage[acronym]{glossaries}

\allowdisplaybreaks

\spdefaulttheorem{counterex}{Counter-example}{\it}{\rm}
\spdefaulttheorem{assumption}{Assumption}{\bf}{\rm}

\newcommand{\rt}[1]{\textcolor{red}{#1}}

\makeatletter
\newcommand{\setword}[2]{%
  \phantomsection
  #1\def\@currentlabel{\unexpanded{#1}}\label{#2}%
}
\makeatother

\newcommand{\prox}{\mathrm{Prox}}
\newcommand{\R}{\mathbb{R}}
\newcommand{\primal}{\mathcal{P}}
\newcommand{\spp}{\mathcal{SPP}}
\newcommand{\X}{\mathcal{X}}
\newcommand{\Y}{\mathcal{Y}}

\newcommand{\Z}{\mathcal{Z}}
\newcommand{\C}{\mathcal{C}}
\newcommand{\K}{\mathcal{K}}
\newcommand{\W}{\mathcal{W}}
\newcommand{\D}{\mathcal{D}}
\newcommand{\F}{\mathcal{F}}
\newcommand{\U}{\mathcal{U}}
\newcommand{\Op}{\mathcal{O}}
\newcommand{\Lag}{\mathcal{L}}
\newcommand{\Q}{\mathcal{Q}}
\newcommand{\Qb}{\mathbf{Q}}
\newcommand{\0}{\mathbf{0}}
\newcommand{\cb}{\mathbf{c}}
\newcommand{\nor}{\mathcal{N}}
\newcommand{\dom}{\mathrm{dom}}
\newcommand{\dist}{\mathrm{dist}}
\newcommand{\proj}{\mathrm{Proj}}

\newcommand{\Id}{\mathrm{Id}}
\newcommand{\Ran}{\mathrm{Ran}}
\newcommand{\G}{\mathcal{G}_{\beta}}

\newcommand{\sgn}{\mathrm{sgn}}

\newignoredglossary{ignored}

\newacronym{cs}{CSi}{Cauchy–Schwarz inequality} 
\newacronym{yi}{Yi}{Young's inequality} 
\newacronym[type=ignored]{kkt}{\textcolor{red}{KKT error}}{\textcolor{red}{Karush–Kuhn–Tucker error}}
\newacronym[type=ignored]{dg}{\textcolor{red}{DG}}{\textcolor{red}{Duality Gap}}
\newacronym[type=ignored]{OG}{\textcolor{red}{OG}}{\textcolor{red}{Optimality Gap}}
\newacronym[type=ignored]{og}{\textcolor{red}{OG}}{\textcolor{red}{optimality gap}}
\newacronym[type=ignored]{FE}{\textcolor{red}{FE}}{\textcolor{red}{Feasibility Error}}
\newacronym[type=ignored]{fe}{\textcolor{red}{FE}}{\textcolor{red}{feasibility error}}
\newacronym[type=ignored]{OGFE}{\textcolor{red}{OGFE}}{\textcolor{red}{Optimality Gap and Feasibility error}}
\newacronym[type=ignored]{ogfe}{\textcolor{red}{OGFE}}{\textcolor{red}{optimality gap and feasibility error}}
\newacronym[type=ignored]{SDG}{\textcolor{red}{SDG}}{\textcolor{red}{Smoothed Duality Gap}}
\newacronym[type=ignored]{sdg}{\textcolor{red}{SDG}}{\textcolor{red}{smoothed duality gap}}
\newacronym[type=ignored]{PDG}{\textcolor{red}{PDG}}{\textcolor{red}{Projected Duality Gap}}
\newacronym[type=ignored]{pdg}{\textcolor{red}{PDG}}{\textcolor{red}{projected duality gap}}
\newacronym[type=ignored]{doss}{\textcolor{red}{DOSS}}{\textcolor{red}{Distance to the Optimal Solution Set}}
\newacronym[type=ignored]{ids}{\textcolor{red}{IDS}}{\textcolor{red}{Infimal Sub-differential Size}}
\newacronym[type=ignored]{MSR}{\textcolor{red}{MSR}}{\textcolor{red}{Metric Sub-Regularity}}
\newacronym[type=ignored]{msr}{\textcolor{red}{MSR}}{\textcolor{red}{metric sub-regularity}}
\newacronym[type=ignored]{QEB}{\textcolor{red}{QEB}}{\textcolor{red}{Quadratic Error Bound}}
\newacronym[type=ignored]{qeb}{\textcolor{red}{QEB}}{\textcolor{red}{quadratic error bound}}
\newacronym[type=ignored]{qebsg}{\textcolor{red}{QEBSG}}{\textcolor{red}{quadratic error bound of the smoothed gap}}
\newacronym[type=ignored]{QEBSG}{\textcolor{red}{QEBSG}}{\textcolor{red}{Quadratic Error Bound of the Smoothed Gap}}
\newacronym[type=ignored]{msrsdl}{\textcolor{red}{MSRSDL}}{\textcolor{red}{metric sub-regularity of the sub-differential of the Lagrangian}}
\newacronym[type=ignored]{MSRSDL}{\textcolor{red}{MSRSDL}}{\textcolor{red}{Metric Sub-Regularity of the Sub-Differential of the Lagrangian}}

\begin{document}

\title{The Smoothed Duality Gap as a Stopping Criterion 
}


\author{Iyad Walwil         \and
        Olivier Fercoq 
}


\institute{Iyad Walwil \at
              LTCI, T\'el\'ecom Paris, Institut Polytechnique de Paris, France \\ Department of Mathematics, Faculty of Sciences, An-Najah National University, Nablus, Palestine \\
              \email{iyad.walwil@telecom-paris.fr}           
           \and
           Olivier Fercoq \at
            LTCI, T\'el\'ecom Paris, Institut Polytechnique de Paris, France \\
            \email{olivier.fercoq@telecom-paris.fr}
}

\date{Received: date / Accepted: date}

\maketitle


\input{Abstract}

\section{Introduction} \label{sec:intro}
\input{Introduction}
\subsection{Notations}
\input{Notations}

\section{Background} \label{sec:background}
\input{Background}

\section{Optimality measures} \label{sec:optimality measures}

\input{Optimality_measures}

\section{Smoothed Duality Gap} \label{sec:SDG}
\input{SDG}

\section{Optimality Gap Bounds} \label{sec:OG bounds}
\input{OG_Bounds}

\section{Comparability Bounds} \label{sec:comparability bounds}

\input{CB_intro}

\input{SDG_KKT_bounds}

\input{SDG_PDG_bounds}

\section{Numerical experiments} \label{sec:numerical exps}

\input{EXP_intro}

\input{LCLS}

\input{DO}

\input{QP}

\input{BP}

\input{Benchmark}

\section{Conclusion and Perspectives}
\input{Conclusion_and_Perspectives}

\appendix
\section{MSR and QEB for LC-LS} \label{appendix:MSR et QEB, LC-LS}
\input{Proofs}
\section{PDG in terms of SDG}
\label{appendix:PDG in terms of SDG}
\input{Last_results}

\begin{acknowledgements}
This work was supported by the Agence National de la Recherche grant ANR-20-CE40-0027, Optimal Primal-Dual Algorithms (APDO), and T\'el\'ecom Paris's research and teaching chair Data Science and Artificial Intelligence for Digitalized Industry and Services DSAIDIS.
\end{acknowledgements}

%
\paragraph{\emph{\textbf{Conflict of interest}}}

The authors declare that they have no conflict of interest.

\bibliographystyle{spmpsci}      
\bibliography{ref}   




\end{document}

%% file: Abstract.tex


\begin{abstract}
We optimize the running time of the primal-dual algorithms by optimizing their stopping criteria for solving convex optimization problems under affine equality constraints, which means terminating the algorithm earlier with fewer iterations. We study the relations between four stopping criteria and show under which conditions they are accurate to detect optimal solutions. The uncomputable one: \textit{"Optimality gap and Feasibility error"}, and the computable ones: the \textit{"Karush-Kuhn-Tucker error"}, the \textit{"Projected Duality Gap"}, and the \textit{"Smoothed Duality Gap"}. Assuming metric sub-regularity or quadratic error bound, we establish that all of the computable criteria provide practical upper bounds for the optimality gap, and approximate it effectively. Furthermore, we establish comparability between some of the computable criteria under certain conditions. Numerical experiments on basis pursuit, and quadratic programs with(out) non-negative weights corroborate these findings and show the superior stability of the smoothed duality gap over the rest.
\keywords{Convex optimization \and Stopping criteria \and Optimality gap and Feasibility error \and Karush-Kuhn-Tucker\and Projected Duality Gap \and Smoothed Duality Gap}
\subclass{65K05 \and 90C25 \and 90C46}
\end{abstract}


%% file: Introduction.tex
How do we decide when to stop doing something when we don’t really know when to stop? This is a non-trivial question in general that we can face in several events during our life. For instance, as researchers, deciding when to stop conducting more experiments to practically validate a new theoretical result is one of the hardest decisions to make. In this paper, we would like to answer a quite similar question: How do we decide when to stop an iterative algorithm seeking an $\varepsilon-$solution for an optimization problem? 

When solving an optimization problem with an iterative method, one is looking for an $\varepsilon-$solution (i.e. a solution that's $\varepsilon$ close to the optimal one). However, testing if a point is actually an $\varepsilon-$solution is not always easy, so we do not really know when to stop the algorithm. This can be detrimental to the running time of the method because of unnecessary iterations. In this work, we aim to address this problem by studying several stopping criteria and determine under which conditions they are accurate to detect $\varepsilon-$solutions.

\begin{figure}[htbp]
\subfloat[Fixed number of iterations, $\varepsilon_1 \approx 10^{-28}$]{\includegraphics[width=0.5\linewidth]{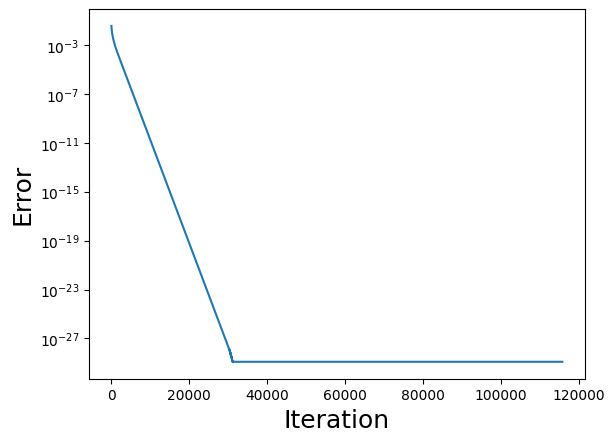}\label{subfig:mot:FNoI}}
\subfloat[Karush-Kuhn-Tucker error, $\varepsilon_2 \approx 10^{-6}$]{\includegraphics[width=0.4825\linewidth]{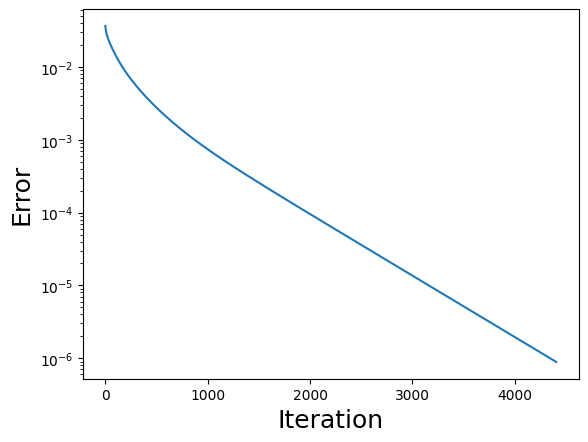}\label{subfig:mot:KKT}}
\caption{Gradient descent is employed to address an unconstrained Least-Squares problem, utilizing two distinct stopping criteria aimed at achieving an $\varepsilon = 10^{-5}$ solution}
\label{fig:MOT}
\end{figure}

Figure \ref{fig:MOT} illustrates the significance of our study by demonstrating a substantial disparity observed when solving an unconstrained Least-Squares problem with the same algorithm but employing two distinct stopping criteria to achieve an $\varepsilon = 10^{-5}$ solution. At first glance, one might infer that sub-figure \ref{subfig:mot:FNoI} outperforms sub-figure \ref{subfig:mot:KKT}. Indeed, this inference holds when solely considering the $y-$axis, where a smaller error indicates closer proximity to optimality. However, the $x-$axis conveys an alternate narrative regarding the requisite number of iterations and thereby the computational run-time. While our objective is to detect an $\varepsilon = 10^{-5}$ solution, sub-figure \ref{subfig:mot:FNoI} has gone too far for, approximately, a $10^{-28}$ solution but at the expense of time. Although sub-figure \ref{subfig:mot:FNoI} yields almost an optimal solution, it might become unfeasible to track in other scenarios involving high-dimensional problems for example.

We are interested in convex optimization problems under affine equality constraints. That is:
\begin{equation} \tag{$\primal$} \label{Primal problem}
    \min_{x \in \X} ~ f(x) \hspace{1cm} \text{subject to} \hspace{1cm} Ax = b
\end{equation}
where $f \colon \X \rightarrow \R \cup \{+\infty\}$ is a proper, lower semi-continuous, and convex function with a computable proximal operator, $A \colon \X \rightarrow \Y$ is a linear operator, and $b \in \Y$. We will, simultaneously, solve the primal and dual problems by solving their associated saddle point problem: 
\begin{equation} \tag{$\spp$} \label{Saddle pt. prob.}
    \min_{x \in \X} \max_{y \in \Y} ~ \Lag(x, y) := f(x) + \left\langle Ax - b, y \right\rangle 
\end{equation}
where $\Lag(x, y)$ is the associated Lagrangian with (\ref{Primal problem}) and $y \in \Y$ is the so-called Lagrange multiplier or dual variable. Throughout the paper, we assume the existence of a solution to (\ref{Saddle pt. prob.}). 

Problems of form (\ref{Saddle pt. prob.}) are ubiquitous in operational research, signal processing, shape optimization, statistical learning, etc... For example,
\begin{itemize}
    \item Method of framers \cite{daubechies1992ten} and Basis Pursuit \cite{Chen1994BasisP} are techniques for decomposing a signal into an “optimal” superposition of dictionary elements by picking a solution whose coefficients have minimum $\ell_2$, and $\ell_1$ norms, respectively.
    \item Statistical learning covers wide range of (un)constrained optimization problems like LASSO \cite{LASSO,LASSO1}, constrained least-squares \cite{LCLS1,LCLS}, etc...
    \item In operational research: linear programs \cite{LP,Taha1998} stand out as the most renowned problems within a broader framework that encompasses the presence of inequality constraints, $\displaystyle \min_{x \geq 0} \max_y \langle c, x \rangle + \left\langle Ax - b, y \right \rangle$

\end{itemize}

Several measures of optimality have been considered in the literature. The first and most natural one is \textit{"\acrfull{OGFE}"}, which directly fits the definition of the optimization problem at stake. Indeed, the \textit{\acrlong{OG}} represents the difference between the objective function value at the current solution $(x_k)$ and the optimal solution $(x^{\star})$, while the \textit{\acrlong{FE}} measures the constraints violation. The second traditional one is the \textit{"\acrfull{doss}"} that measures how far or close the solution is to be an optimal one. However, as both of those measures depend on the unknown point: $x^{\star}$, one cannot compute them before the problem is actually solved! Hence in algorithms, the \textit{"\acrlong{kkt}"} \cite{convex_analysis,convex_opt} is widely used \cite{AKKT,AKKT_proximity,KKT_Knitro}, it is a computable quantity and serves as a first-order optimality measure for achieving optimality in non-linear programming problems. It accomplishes this by quantifying the error in feasibility and identifying, within the sub-differential of the associated Lagrangian, the element with the smallest norm. Moreover, if the Lagrangian’s gradient is metrically sub-regular \cite{variational_analysis}, then a small \acrshort{kkt} implies that the current point is close to the set of saddle points. When the primal and dual domains are bounded, the difference between the primal and dual optimal values that define the so-called \textit{"\acrfull{dg}"} \cite{convex_opt} is a good way to measure optimality: it is often easily computable, and it is an upper bound to the \textit{\acrlong{og}}. However, for unbounded domains, it's no longer informative, as it will consistently be infinity, except for the final iterations when the algorithm begins identifying feasible solutions. A first generalization to unbounded domains has been proposed in \cite{SDG_intro}: the \textit{"\acrfull{SDG}"}, a new measure of optimality that is widely applicable but less well-studied than the other ones. It's based on the smoothing of non-smooth functions \cite{Nesterov2005SmoothMO}, and takes finite values for constrained problems, unlike the duality gap. Moreover, if the smoothness parameter is small and the smoothed duality gap is small, this means that the \textit{\acrlong{og}}, and the \textit{\acrlong{fe}} are both small. A second one has been proposed for linear programs in \cite{PDG}, we have extended its definition within our framework (\ref{Saddle pt. prob.}) and termed it the "\acrfull{PDG}." This concept involves calculating the duality gap at each iteration while simultaneously projecting the primal-dual solution onto their respective feasibility spaces. Another recent measure has been proposed for analyzing the primal-dual hybrid gradient algorithm in \cite{lu2023infimal}, the authors dub it: the \textit{"\acrfull{ids}"}. It always has a finite value, easy to compute, and more importantly, it monotonically decays. \acrshort{ids} essentially measures the distance between 0 and the sub-differential of the objective function. Throughout the remainder of this paper, our attention will be directed towards four optimality measures: the \acrlong{ogfe}, the \acrlong{kkt}, the \acrlong{pdg}, and the \acrlong{sdg}. On the one hand, we aim to demonstrate the conditions under which the computable measures (\acrshort{kkt}, \acrshort{pdg}, and \acrshort{sdg}) serve as upper bounds or approximations for the uncomputable measure, \acrlong{ogfe}. On the other hand, our objective is to assess and compare the performance of \acrlong{sdg} against both the \acrlong{kkt} and the \acrlong{pdg} to determine its efficacy and stability.

\subsection{Our contributions and Paper organization}
 The paper's contributions can be outlined as follows:
\begin{enumerate}
    \item In section \ref{sec:background}, we start by providing a comprehensive background. 
    \item Section \ref{sec:optimality measures} is dedicated to establishing a common understanding by precisely defining the different optimality measures, accompanied by a detailed analysis when necessary. Moreover, this section introduces our generalization of the existing measure: the \rt{projected duality gap}, originally defined only for linear programs, to our framework (\ref{Saddle pt. prob.}).
    \item Section \ref{sec:SDG} is allocated to present and deeply study the novel measure, the \acrlong{sdg}, while also deriving some new properties.
    \item Section \ref{sec:OG bounds} is devoted to the establishment of computable approximations for the uncomputable measure, the \acrlong{og}, in terms of the other computable measures (the \acrshort{kkt} error, \acrshort{pdg}, and \acrshort{sdg}). The necessary regularity assumptions for this purpose are also introduced.
    \item A more in-depth examination of the \acrlong{sdg} and its interconnections with both: the \acrshort{kkt} and \acrshort{pdg} is presented in section \ref{sec:comparability bounds}. More precisely, we elucidate the conditions under which these measures can function as approximations to the \acrlong{sdg}, and vice versa.
    \item Section \ref{sec:numerical exps} presents several numerical experiments illustrating our findings, while the technical proofs are relegated to the appendix. 
\end{enumerate}

%% file: Notations.tex
We shall denote $\X$ the primal space, $\Y$ the dual space, and $\Z = \X \times \Y$ the primal-dual space. We assume that those vector spaces are Euclidean spaces. Similarly, for a primal vector $x$, and a dual vector $y$, we shall denote $z = (x, y)$. The set of saddle points will be denoted as $\Z^{\star}$. Let $\Gamma_0(\X)$ denote the set of all proper, lower semi-continuous, and convex functions $f \colon \X \rightarrow \R\cup\{+\infty\}$. The proximal operator of a function $f$ and a step size $s > 0$ is given by: 
\begin{equation*}
    \prox_{sf}(x) = \arg\min_{x' \in \X} f(x') + \frac{1}{2s} \|x' - x\|^2
\end{equation*}
Let $\displaystyle \dist(z, \Z) = \min_{z' \in \Z} \|z - z'\|$ denote the distance between point $z$ and set $\Z$. We will make use of the convex indicator function associated with the convex subset $\mathcal{C} \subset \mathcal{X}$: 
\begin{equation*}
    \imath_{\mathcal{C}}(x) = \begin{cases} 0 & \text{if} ~ x \in \mathcal{C} \\ + \infty & \text{Otherwise} \end{cases}
\end{equation*}

%% file: Background.tex
\begin{proposition}[Young's inequality]  \label{props:YI} 

 For all vectors $\mathbf{u}$ and $\mathbf{v}$ of an inner product space, and for any scalar $\lambda$. The following inequality holds: 
\begin{equation} \label{eqn:YI}
    |\langle \mathbf{u}, \mathbf{v} \rangle| \leq \frac{\lambda^2}{2} \|\mathbf{u}\|^2 + \frac{1}{2\lambda^2} \|\mathbf{v}\|^2
\end{equation}
\end{proposition}
\begin{lemma} \label{lem:charac for str-cvx}
    A function $f \colon \X \rightarrow  (-\infty, +\infty]$ is $\mu-$strongly convex if for any $x, y \in \dom f$ and for any $q \in \partial f(x)$, the following inequality holds: 
    \begin{equation*} \label{eqn:charac for str-cvx}
        f(y) \geq f(x) + \langle q, y - x \rangle + \frac{\mu}{2} \|y - x\|^2 
    \end{equation*}
\end{lemma}
\begin{lemma} \label{lem:Sub-diff property}
    Given $f \in \Gamma_0(\X)$, if $q_1 \in \partial f(x_1)$ and $q_2 \in \partial f(x_2)$, then: 
    \begin{equation*} \label{eqn:Sub-diff property}
        \langle q_1 - q_2, x_1 - x_2 \rangle \geq 0 
    \end{equation*}
\end{lemma}
\begin{lemma} \label{lem:prox property}
    Let $g \in \Gamma_0(\X)$, and denoting $p = \prox_{s g}(x)$, then we have for all $v \in \X$: 
    \begin{equation*} \label{eqn:prox property}
        g(p) + \frac{1}{2s} \|x - p\|^2 \leq g(v) + \frac{1}{2s} \|x - v\|^2 - \frac{1}{2s} \| p -v \|^2
    \end{equation*}
\end{lemma}
\begin{proposition}[Projection properties]

    Let $\C \subseteq \X$ be a nonempty, closed, and convex subset. Define $a := \proj_{\C}(x)$ for $x \in \X$. Then for any $u \in \C$, the following holds: 
    \begin{align}
        &\|a - x\|^2 \leq \|u - x\|^2 \label{eqn:proj - norms} \\
        &~ \langle u - a, a - x \rangle \geq 0  \label{eqn:proj - inner}
    \end{align}
\end{proposition}
\begin{definition}[Separable function] \label{def:separable fun} 

We say that a function $\varphi \colon \R^n \rightarrow \R \cup \{+\infty\}$ is separable if there exists $n$ functions $\varphi_i \colon \R \rightarrow \R \cup \{+\infty\}$ such that $\displaystyle \forall x \in \R^n, \varphi(x) = \sum_{i = 1}^n \varphi_i(x_i)$.
\end{definition}
\begin{proposition}[Properties of Separable function] \label{props:Prop of separable fun} 

Let $\varphi \colon \R^n \rightarrow \R \cup \{+\infty\}$ be a separable function, then for any $x \in \R^n$, 
\begin{align}
   \partial \varphi (x) &= \partial \varphi_1(x_i) \times \dots \times \partial \varphi_n(x_n)  \label{eqn:separable, sub-diff} \\ 
   \prox_{\gamma \varphi}(x) &= \left(\prox_{\gamma \varphi_1}(x_1), \dots, \prox_{\gamma \varphi_n}(x_n) \right) \label{eqn:separable, prox}\\ 
   \sup_{x \in \R^n} \varphi(x) &= \sup_{x_1 \in \R} \varphi_1(x_1) + \dots + \sup_{x_n \in \R} \varphi_n(x_n) \label{eqn:separable, sup}
\end{align}
\end{proposition}
\begin{definition}[Fenchel-Legendre Conjugate]

Let $f \colon \X \rightarrow [- \infty, +\infty]$. The \textit{Fenchel-Legendre conjugate} of $f$ is the function $f^* \colon \X \rightarrow [-\infty, +\infty]$ defined by: 
\begin{equation*} \label{eqn:Fenchel-conjugate}
    f^*(\phi) = \sup_{x \in \X} \langle \phi, x \rangle - f(x), \hspace{2cm}  \phi \in \X
\end{equation*}
\end{definition}
\begin{proposition}[Fenchel-Young's inequality] \label{prop:Fenchel-Young}

Let $f \colon \X \rightarrow [-\infty, +\infty]$. For all $(x, \phi) \in \X\times \X$, the following inequality holds: 
\begin{equation*} \label{eqn:Fenchel-Young}
    f(x) + f^*(\phi) \geq \langle \phi, x \rangle 
\end{equation*}
with \textbf{equality} if, and only if, $\phi \in \partial f(x)$.
\end{proposition}
\begin{proposition} \label{prop: f and f* sub-diff}
    Let $f \colon \X \rightarrow (-\infty, +\infty]$ be proper, convex, and l.s.c. at some point $x \in \X$. Then, 
    \begin{equation*}  \label{eqn: f and f* sub-diff}
        \phi \in \partial f(x) \iff x \in \partial f^*(\phi)
    \end{equation*}
\end{proposition}
\begin{proposition}[Moreau's identity] \label{prop: Moreau's identity }

    Consider $f \in \Gamma_0(\X)$ and $s > 0$. Then, for any $x \in \X$, 
    \begin{equation*} \label{eqn: Moreau's identity }
        \prox_{sf}(x) + s \prox_{s^{-1}f^*}\left(\frac{x}{s}\right) = x
    \end{equation*}
\end{proposition}
\begin{definition} \label{def:MSR}
    A set-valued function $F \colon \Z \rightrightarrows \Z$ is \textbf{metrically sub-regular} at $z$ for $v$ if there exists $\gamma > 0$ and a neighborhood $N(z)$ of $z$ such that $\forall z' \in N(z)$, 
    \begin{equation*}
        \dist\left(F(z'), v\right) \geq \gamma \dist\left(z', F^{-1}(v)\right) 
    \end{equation*}
\end{definition}
\begin{definition} \label{def:QEB}
 We say that a function $f \colon \X \rightarrow \R \cup \{+\infty\}$ has a \textit{quadratic error bound} if there exists $\eta$ and an open region $\mathcal{R} \subseteq \X$ that contains $\arg\min f$ such that for all $x \in \mathcal{R}$: 
    \begin{equation*}
        f(x) - \min f \geq \frac{\eta}{2} \dist(x, \arg\min f)^2
    \end{equation*}
    We shall use the acronym $f$ has an $\eta$-QEB.
\end{definition}
\begin{proposition}[\textit{Proposition 2 in \cite{QEB_fercoq}}] \label{lem:MSR-->QEB}

   Let $f$ be a convex function such that $f(x) \leq f_0$ implies $\left \| \partial f(x) \right \|_0 \geq \eta \dist\left(x, \X^{\star}\right)$. Then, $f(x) \geq f(x^{\star}) + \frac{\eta}{2} \dist\left(x, \X^{\star}\right)$ as soon as $f(x) \leq f_0$. 
\end{proposition}
\begin{lemma}
    Let $M \in \R^{m \times n}$ be a symmetric matrix, then for any vector $x \in \R^n$: 
    \begin{equation*} \label{eqn:Rayleigh}
        \|Mx\| \geq |\lambda(M)|_{\min} \|x\| 
    \end{equation*}
    where $|\lambda(M)|_{\min}$ is the smallest absolute value of the non-zero eigenvalues of $M$. 
\end{lemma}

%% file: Optimality_measures.tex
Within this section, our objective is to establish a common understanding by defining the aforementioned optimality measures and the concept of $\varepsilon-$solution. Furthermore, where necessary, we provide a comprehensive analysis supporting these definitions.
\begin{definition}[$\varepsilon-$Solution] \label{def:eps - solution}

Given a target accuracy $\varepsilon > 0$.  A point $\hat{x} \in \mathcal{X}$ is said to be an $\varepsilon-$Solution of (\ref{Primal problem}) if:
    \begin{align} \label{eqn:eps - solution}
    |f(\hat{x}) - f(x^{\star})| \leq \varepsilon && \text{and} && \|A\hat{x} - b\| \leq \varepsilon 
\end{align}
\end{definition}

\begin{definition}[\acrfull{OGFE}] \label{def:OGFG}

    The \textit{Optimality gap}, and the \textit{Feasibility error} for (\ref{Saddle pt. prob.}) at a point $\hat{x} \in \X$ are defined, respectively, as follows:
    \begin{align} \label{eqn:OGFG}
        \Op(\hat{x}) &= \max\left(f(\hat{x}) - f^{\star}, 0\right) & \F(\hat{x}) &= \left\|A\hat{x} - b\right\| 
    \end{align}
\end{definition}
One can observe that the definition of the \acrlong{ogfe} aligns precisely with the definition of $\varepsilon-$solutions. Consequently, any combination of the \textit{\acrlong{og}} and the \textit{\acrlong{fe}} can be utilized to assess whether a provided solution constitutes an $\varepsilon-$solution. However, due to the uncomputability of the \textit{\acrlong{og}}, determining whether a solution truly qualifies as an $\varepsilon$ solution becomes a challenging task.

Next, we introduce the \textit{\acrlong{kkt}}, which is quite similar to the outlined definition of the \textit{\acrlong{ids}} as presented in \cite{lu2023infimal}. They employ the same definition as we do but with weighted norms. This definition draws inspiration from the Karush-Kuhn-Tucker conditions applied to (\ref{Saddle pt. prob.}), where the saddle points are identified by having the sub-differential of the associated Lagrangian equal to 0. Said otherwise, $(x^{\star}, y^{\star})$ qualifies as a saddle point for (\ref{Saddle pt. prob.}) if, and only if:
\begin{align*}
    &\partial_x \Lag(x^{\star}, y^{\star}) = \partial f(x^{\star}) + A^Ty^{\star} = 0  &&\text{(Stationarity)} \\
    &\partial_y \Lag(x^{\star}, y^{\star}) = Ax^{\star} \hspace{0.5cm}- b \hspace{0.725cm} = 0  &&\text{(Primal-feasibility)}
\end{align*}
Therefore, employing any combination of these two conditions would provide insight into whether the current point functions as a saddle point or not. We utilize the squared norm of the vector that combines these two conditions.
\begin{definition}[\acrlong{kkt} (\rt{KKT})] \label{def:KKT error}
    
    The Karush-Kuhn-Tucker error for (\ref{Saddle pt. prob.}) is defined as follows: 
    \begin{equation} \label{eqn:KKT error}
        \K(z) := \left\|\partial f(x) + A^Ty \right\|^2_0 + \left\|Ax - b\right\|^2 
    \end{equation}
    where we define the "Infimal size" of a set $\Q$ as: 
    \begin{equation}
        \left\|\Q\right\|_0 := \min \{\|q\| ~|~ q \in \Q\}
    \end{equation}
\end{definition}
\subsection{Projected Duality Gap}
The optimality measure, which we termed as the \textit{\acrlong{pdg}}, was initially introduced in \cite{PDG} only for linear programs. This metric serves as the stopping criterion utilized in the \texttt{\href{https://docs.scipy.org/doc/scipy/reference/optimize.linprog-interior-point.html}{linprog}} solver within SciPy for Python. In this work, we have extended its application to integrate with our (\ref{Saddle pt. prob.}) framework. Essentially, our generalization operates quite similarly to the conventional duality gap. However, it differs in that it computes the duality gap at each iteration while simultaneously projecting the primal-dual solution onto their respective feasibility spaces. Consequently, this method always yields a finite value, unlike the conventional duality gap.
\begin{definition}[\acrfull{PDG}] \label{def:PDG}
    
    The Projected Duality Gap for (\ref{Saddle pt. prob.}) is defined as follows: 
    \begin{align} 
        \D(z) &:= |f(x) + f^*(a) + \langle b, y\rangle|^2 +  \left\| a +A^Ty \right\|^2 + \left\|Ax - b\right\|^2 \label{eqn:PDG}\\ 
        a &:= \proj_{\dom f^*}\left(-A^Ty\right) \label{eqn:PDG - a}
    \end{align}
\end{definition}
This definition is inspired by the aforementioned saddle point problem (\ref{Saddle pt. prob.}):
\begin{align*}
    \min_{x \in \X} \max_{y \in \Y} ~ \Lag(x, y) := f(x) + \left\langle Ax - b, y \right\rangle 
\end{align*}
Thanks to the Fenchel-Legendre transform, we can express the associated primal and dual problems equivalently in terms of it. That is: 
\begin{align*}
    \min_{x \in \X} &\max_{y \in \Y} f(x) + \left\langle Ax, y \right\rangle - \langle b, y \rangle \\ 
    \equiv &\min_{x \in \X} f(x) + \imath_{\{b\}}\left(Ax\right) \\ 
    \equiv &\max_{y \in \Y}  -f^*\left(-A^Ty\right) - \langle b, y \rangle 
\end{align*}
Therefore, $(x, y)$ is a saddle point if the following conditions hold: 
\begin{empheq}[left=\empheqlbrace]{align*}
&|f(x) + \imath_{\{b\}}\left(Ax\right) + f^*\left(-A^Ty\right) + \langle b, y \rangle| = 0  \\
&Ax \in \dom(\imath_{\{b\}}) \equiv Ax = b\\
&-A^Ty \in \dom(f^*)
\end{empheq}
Hence, by considering any combination of these conditions, an optimality measure is obtained. We take the squared norm of the vector combining the three conditions. 

The fourth measure under consideration is entirely novel and less well-studied compared to the aforementioned ones. This is an area where we invest more time and effort in our study. Consequently, we allocate a dedicated section to comprehensively introduce it, along with deriving some new properties.

%% file: SDG.tex
In this section, we present the last optimality measure along with some new properties. The \textit{\acrlong{sdg}}, initially introduced in \cite{SDG_intro}, represents a novel measure of optimality that is widely applicable but remains less studied compared to the previously discussed ones.
\begin{definition}[\textit{Definition 4 in \cite{QEB_fercoq}}]. \label{def:SDG general}

Given $\beta = (\beta_x, \beta_y) \in [0, +\infty]^2, z \in \Z$ and $\Dot{z} \in \Z$, the smoothed gap $\G$ is the function defined by:
\begin{equation} \label{eqn:SDG general}
    \G(z;\Dot{z}) = \sup_{z' \in \mathcal{Z}} \mathcal{L}(x, y') - \mathcal{L}(x', y) - \frac{\beta_x}{2} \|x' - \Dot{x}\|^2 - \frac{\beta_y}{2} \|y' - \Dot{y}\|^2
\end{equation}
\end{definition}
When the smoothness parameter $\beta = 0$, we recover the conventional duality gap. The smoothed duality gap concept involves smoothing the duality gap through a proximity function \cite{Nesterov2005SmoothMO}, thereby ensuring that the smoothed duality gap attains finite values for constrained problems, unlike its conventional counterpart. Additionally, when the smoothness parameter is small and the smoothed duality gap is small, it signifies that both the optimality gap and the feasibility error are also small.

Moreover, the author in \cite{QEB_fercoq} has found that the smoothed duality gap offers a robust outcome. Independently of any unknown or uncomputable variables, it serves as a valid optimality measure. Therefore,  it could be utilized as a stopping criterion.
\begin{definition}[\textit{Definition 5 in \cite{QEB_fercoq}}]

Given $\beta = (\beta_x, \beta_y) \in [0, +\infty[^2$, and $z \in \Z$, the \textit{self-centered} smoothed gap is given by $\G(z, z)$. 
\end{definition}
\begin{theorem}[\textit{Proposition 15 in \cite{QEB_fercoq}}] 

The self-centered smoothed gap is a measure of optimality. Indeed, $\forall z \in \Z, \forall \beta \in [0, +\infty[^2$:
\begin{enumerate*}[label=(\roman*)]
    \item $\G(z, z) \geq 0$
    and 
    \item $\G(z, z) = 0 \iff z \in \Z^{\star}$
\end{enumerate*}
\end{theorem}
An obstacle in the definition of the smoothed duality gap lies in it constituting an optimization problem in itself, thereby adding complexity. However, for the previously mentioned (\ref{Saddle pt. prob.}), we have managed to derive a closed-form expression for the smoothed duality gap.
\begin{proposition} \label{props:SDG}
The self-centered smoothed gap for (\ref{Saddle pt. prob.}) can be computed as follows: 
    \begin{align} 
            \G(z) &:= f(x) - f(p) + \langle A(x - p), y \rangle - \frac{\beta_x}{2} \|p - x\|^2 + \frac{1}{2\beta_y} \left\|Ax - b\right\|^2 \label{eqn:SDG}\\ 
            p&:=  \prox_{\beta_x^{-1}f} \left(x - \frac{1}{\beta_x}A^Ty\right) \label{eqn:p}
        \end{align}
    \end{proposition}
    \begin{proof}
 We start with the definition of the smoothed duality gap: 
    \begin{align*}
        \G(z) &= \begin{multlined}[t] f(x) + \langle b, y \rangle + {\color{red} \max_{x'} -f(x') -\langle Ax', y \rangle - \frac{\beta_x}{2} \|x'-x\|^2} + {\color{ao(english)} \max_{y'} \langle Ax - b, y' \rangle } \\ {\color{ao(english)} - \frac{\beta_y}{2} \|y' - y\|^2} \end{multlined} \\ 
        &= \begin{multlined}[t] f(x) + \langle b, y \rangle {\color{red} -f({\color{blue} p}) -\langle A{\color{blue} p}, y \rangle - \frac{\beta_x}{2} \|{\color{blue} p}-x\|^2} + {\color{ao(english)} \left\langle Ax- b, y \right\rangle + \frac{1}{2\beta_y} \left\|Ax - b\right\|^2 }\end{multlined} \\ 
        &= f(x) - f(p) + \langle A(x - p), y \rangle - \frac{\beta_x}{2} \|p - x\|^2 + \frac{1}{2\beta_y} \|Ax - b\|^2
    \end{align*}
    where 
 \begin{align*}
    p &= \arg\max_{x'} -f(x') - \left\langle Ax', y \right\rangle - \frac{\beta_x}{2} \|x' - x\|^2 \\
      &= \arg\min_{x'} f(x') + \langle x', A^Ty \rangle + \frac{\beta_x}{2} \left\|x' - \left(x {\color{blue} - \frac{1}{\beta_x} A^Ty}\right){\color{blue} - \frac{1}{\beta_x} A^Ty}\right\|^2 \\ 
        &= \begin{multlined}[t] \arg\min_{x'} f(x') {\color{red} + \langle x', A^Ty \rangle} + \frac{\beta_x}{2} \left\|x' - \left(x - \frac{1}{\beta_x} A^Ty\right)\right\|^2 {\color{ao(english)} + \frac{1}{2\beta_x}\left\|A^Ty\right\|^2}  \\ {\color{red} - \left\langle x', A^Ty \right\rangle} {\color{ao(english)} + \langle x - \frac{1}{\beta_x} A^Ty, A^Ty \rangle}\end{multlined} \\
    &= \prox_{\beta_x^{-1}f} \left(x - \frac{1}{\beta_x}A^Ty\right) \smartqed
\end{align*}
Where the {\color{red} red} terms cancel, and the {\color{ao(english)} green} ones are free of $x'.$  
\end{proof}
Throughout the remainder of this section, we outline several properties of the self-centered smoothed gap, which will significantly contribute to justifying our subsequent findings. Additionally, we will call $\G(x^{\star}, y; x, y^{\star})$ and $\G(x, y^{\star}; x^{\star}, y)$ the \textit{outer-saddle} and the \textit{inner-saddle} smoothed gaps, respectively.
\begin{lemma} \label{lem:p and fermat}
    Given $\beta = (\beta_x, \beta_y) \in [0, +\infty[^2$, then the proximal point, $p$, defined in (\ref{eqn:p}) satisfies: 
    \begin{equation*} \label{eqn:p and fermat}
        p = \prox_{\beta_x^{-1}f} \left(x - \frac{1}{\beta_x}A^Ty\right) \iff \beta_x(x - p) \in \partial f(p) + A^Ty 
    \end{equation*}
\end{lemma}
\begin{proof}
 Direct implication of Fermat's rule.
\end{proof}
\begin{lemma} \label{lem:G lower bound}
    Given $\beta = (\beta_x, \beta_y) \in [0, +\infty[^2$, then for all $z \in \Z$ the self-centered smoothed gap satisfies: 
    \begin{equation*} \label{eqn:G lower bound}
        \G(z) \geq \frac{\beta_x}{2} \|x - p \|^2 + \frac{1}{2\beta_y} \|Ax - b\|^2 
    \end{equation*}
\end{lemma}
\begin{proof}
By Lemma \ref{lem:prox property}, we know that: 
\begin{multline*}
    p =  \prox_{\beta_x^{-1}f} \left(x - \frac{1}{\beta_x}A^Ty\right) \iff \forall v \in \X, \\ \frac{1}{\beta_x} f(p) + \frac{1}{2} \left\|p - \left(x - \frac{1}{\beta_x}A^Ty\right)\right\|^2 \leq \frac{1}{\beta_x} f(v) + \frac{1}{2} \left\|v - \left(x - \frac{1}{\beta_x}A^Ty\right)\right\|^2 - \frac{1}{2} \|v - p\|^2
\end{multline*}
Taking $v = x$, we obtain: 
\begin{equation} \label{eqn:prox property, v = x}
    f(x) - f(p) - \frac{\beta_x}{2} \|x - p\|^2 \geq \frac{\beta_x}{2} \left\|p - \left(x - \frac{1}{\beta_x}A^Ty\right)\right\|^2 - \frac{1}{2\beta_x}\left\|A^Ty\right\|^2
\end{equation}
Thus, 
\begin{align*}
    \G(z) &= {\color{red} f(x) - f(p) - \frac{\beta_x}{2} \| p - x\|^2} + \langle A(x-p), y \rangle + \frac{1}{2\beta_y} \|Ax- b\|^2 \\ 
    & \stackrel{(\ref{eqn:prox property, v = x})}{\geq} \begin{multlined}[t] {\color{red} \frac{\beta_x}{2}\left\|p - \left(x- \frac{1}{\beta_x}A^Ty\right)\right\|^2 - \frac{1}{2\beta_x}\left\|A^Ty\right\|^2} + \langle A(x - p), y \rangle \\ + \frac{1}{2\beta_y}\|Ax - b\|^2 \end{multlined} \\ 
    & = \frac{\beta_x}{2} \|x - p\|^2 + \frac{1}{2\beta_y} \|Ax - b\|^2  \smartqed
\end{align*}
\end{proof}

\begin{corollary} \label{coro:FG-SDG}
    Given $\beta = (\beta_x, \beta_y) \in [0, +\infty]^2$. Then, for any $z \in \Z$, the \acrlong{fe} defined in (\ref{eqn:OGFG}) could be approximated in terms of the self-centered \acrlong{sdg} defined in (\ref{eqn:SDG}) for (\ref{Saddle pt. prob.}). More precisely, for all $z \in \Z$, 
    \begin{equation} \label{eqn:FG-SDG}
        \|Ax - b\| \leq \sqrt{2\beta_y \G(z)}
    \end{equation}
\end{corollary}
\begin{proof}
Direct implication of Lemma \ref{lem:G lower bound}.   
\end{proof}
\begin{lemma} \label{lem:bound G}
Given $\beta = (\beta_x, \beta_y) \in [0, +\infty]^2$, then for all $z \in \Z$, the self-centered \acrlong{sdg} satisfies: 
    \begin{align}
        \G(z) &\geq \frac{\beta_x}{2} \|x_{\beta}(y) - x\|^2 + \frac{\beta_y}{2} \|y_{\beta}(x) - y\|^2 \label{eqn:bound G}\\ 
        x_{\beta}(y) &:= \arg\max_{x'} - f(x') - \left\langle Ax' - b, y \right\rangle - \frac{\beta_x}{2} \|x' -x\|^2 \\ 
        y_{\beta}(x) &:= \arg\max_{y'} \left\langle Ax - b, y' \right\rangle - \frac{\beta_y}{2} \|y' -y\|^2 
    \end{align}
\end{lemma}
\begin{proof}
\begin{align*}
    \G(z) &= \begin{multlined}[t] f(x) + {\color{red} \sup_{x'} - f(x') - \left\langle Ax' - b, y \right \rangle - \frac{\beta_x}{2} \|x' - x\|^2 } + {\color{blue} \sup_{y'} \left\langle Ax - b, y' \right\rangle } \\ {\color{blue} - \frac{\beta_y}{2} \|y' - y\|^2} \end{multlined} \\ 
    &= f(x) + {\color{red} F_1(x_{\beta}(y))} + {\color{blue} F_2(y_{\beta}(x))} 
\end{align*}
Where 
\begin{align*}
    F_1(\mu) &=  - f(\mu) - \left\langle A\mu - b, y \right \rangle - \frac{\beta_x}{2} \|\mu - x\|^2 \\
    F_2(\nu) &=  \left\langle Ax - b, \nu \right\rangle - \frac{\beta_y}{2} \|\nu - y\|^2
\end{align*}
One can observe that $F_1$ and $F_2$ are $\frac{\beta_x}{2}$ and $\frac{\beta_y}{2}-$strongly concave, respectively. Hence, by Lemma \ref{lem:charac for str-cvx}: 
\begin{equation} \label{eqn: F is str. concave}
\begin{split}
F_1(x_{\beta}(y)) &\geq F_1(x) +  \frac{\beta_x}{2}\| x_{\beta}(y) - x\|^2 \\
F_2(y_{\beta}(x)) &\geq F_2(y) +  \frac{\beta_y}{2}\| y_{\beta}(x) - y\|^2
\end{split}
\end{equation}
Thus, 
\begin{align*}
    \G(z) &= f(x) + {\color{red} F_1(x_{\beta}(y))} + {\color{blue} F_2(y_{\beta}(x))} \\
    & \stackrel{(\ref{eqn: F is str. concave})}{\geq} \begin{multlined}[t] f(x) {\color{red} - f(x) - \left\langle Ax - b, y \right\rangle + \frac{\beta_x}{2}\|x_{\beta}(y) - x\|^2} {\color{blue} +\left\langle Ax - b, y \right\rangle } {\color{blue} + \frac{\beta_y}{2}\|y_{\beta}(x) - y\|^2}\end{multlined} \\ 
    & = \frac{\beta_x}{2}\|x_{\beta}(y) - x\|^2 + \frac{\beta_y}{2}\|y_{\beta}(x) - y\|^2 \smartqed 
\end{align*}
\end{proof}
\begin{corollary} \label{coro:bound G(*, .;.,*)}
Given $\beta = (\beta_x, \beta_y) \in [0, +\infty]^2, z \in \Z$, and $z^{\star} \in \Z^{\star}$, then the outer-saddle smoothed gap satisfies: 
\begin{align*}
    \G(x^{\star}, y ; x, y^{\star}) &\geq -\frac{\beta_x}{2} \|x - x^{\star}\|^2 + \frac{\beta_x}{2} \|x_{\beta}(y) - x^{\star}\|^2  \\ 
    x_{\beta}(y) &= \arg\max_{x'} -f(x') - \left\langle Ax' - b, y \right\rangle - \frac{\beta_x}{2} \|x' - x\|^2
\end{align*}
\end{corollary}
\begin{proof}
\begin{align*}
    \G(x^{\star}, y ; x, y^{\star}) &=  \begin{multlined}[t] f(x^{\star}) + {\color{red} \sup_{x'} \left(- f(x') - \left\langle Ax' - b, y \right \rangle - \frac{\beta_x}{2} \|x' - x\|^2 \right)} \\ + {\color{blue} \sup_{y'} - \frac{\beta_y}{2} \|y' - y^{\star}\|^2} \end{multlined}\\ 
    &= \begin{multlined}[t] f(x^{\star}) {\color{red} - f(x_{\beta}(y)) - \left\langle Ax_{\beta}(y) - b, y \right \rangle - \frac{\beta_x}{2} \|x_{\beta}(y) - x\|^2 } + {\color{blue} 0 } \end{multlined} \\ 
    &= f(x^{\star}) + {\color{red} F(x_{\beta}(y))}
\end{align*}
Where 
\begin{equation*}
    F(\mu) =  - f(\mu) - \left\langle A\mu - b, y \right \rangle - \frac{\beta_x}{2} \|\mu - x\|^2 \\
\end{equation*}
Since $F$ is $\frac{\beta_x}{2}-$Strongly-concave and its maximum is attained at $x_{\beta}(y)$, then by Lemma \ref{lem:charac for str-cvx}: 
\begin{equation*}
 F(x_{\beta}(y)) \geq F(x^{\star}) +  \frac{\beta_x}{2}\| x_{\beta}(y) - x^{\star}\|^2
\end{equation*}
Therefore, 
\begin{align*}
    \G(x^{\star}, y ; x, y^{\star}) &\geq f(x^{\star}) + F(x^{\star}) + \frac{\beta_x}{2} \|x_{\beta}(y) - x^{\star} \|^2 \\ 
    &= \begin{multlined}[t] f(x^{\star}) {\color{red} - f(x^{\star}) - \left\langle Ax^{\star} - b, y \right\rangle - \frac{\beta_x}{2} \|x^{\star} - x \|^2 + \frac{\beta_x}{2}\|x_{\beta}(y) - x\|^2} \end{multlined} \\ 
    & = - \frac{\beta_x}{2}\|x - x^{\star}\|^2 + \frac{\beta_x}{2}\|x_{\beta}(y) - x\|^2 \smartqed 
\end{align*}
\end{proof}
\begin{lemma} \label{lem:G decomposition} Let $z^{\star} \in \Z^{\star}$, then for any $z \in \Z$, the self-centered \acrlong{sdg} can be decomposed in terms of the outer and inner saddle smoothed gaps as: 
    \begin{equation*} \label{eqn:G decomposition}
        \G(z) = \G(x, y^{\star} ; x^{\star}, y) + \G(x^{\star}, y ; x, y^{\star})
    \end{equation*}
\end{lemma}
\begin{proof}
 By using the definition of SDG twice, first: 
\begin{align*}
    \G(x, y^{\star} ; x^{\star}, y) &= \begin{multlined}[t] f(x) + {\color{red} \sup_{x'} - f(x') - \left\langle Ax' -b, y^{\star} \right\rangle -\frac{\beta_x}{2} \|x' - x^{\star} \|^2} \\  + \sup_{y'} \left\langle Ax - b, y' \right\rangle -\frac{\beta_y}{2} \|y' - y\|^2 \end{multlined} \\
    &= f(x) {\color{red} - f(x^{\star})} + \sup_{y'} \left\langle Ax - b, y' \right\rangle -\frac{\beta_y}{2} \|y' - y\|^2
\end{align*}
Second:
\begin{align*}
    \G(x^{\star}, y ; x, y^{\star}) &=  \begin{multlined}[t] f(x^{\star}) + \sup_{x'} - f(x') - \left\langle Ax' - b, y \right\rangle - \frac{\beta_x}{2} \|x' - x\|^2  \\ + {\color{blue} \sup_{y'} - \frac{\beta_y}{2} \|y' - y^{\star} \|^2} \end{multlined} \\ 
    &=  \begin{multlined}[t] f(x^{\star}) + \sup_{x'} - f(x') - \left\langle Ax' - b, y \right\rangle - \frac{\beta_x}{2} \|x' - x\|^2 + {\color{blue} 0 } \end{multlined}  \smartqed
\end{align*}
Summing the two terms implies the result.  
\end{proof}
\begin{lemma} \label{lem: G(*, .;.,*) - sqrt eta}
Given $\beta = (\beta_x, \beta_y) \in [0, +\infty]^2, z \in \Z$, and $z^{\star} \in \Z^{\star}$, then the outer-saddle smoothed gap satisfies: 
    \begin{equation*} \label{eqn:G(*, .;.,*) - sqrt eta}
        \G(x^{\star}, y ; x, y^{\star}) \geq -2\sqrt{\beta_x \G(z)}\|x - x^{\star}\|
    \end{equation*}
\end{lemma}
\begin{proof}
 By corollary \ref{coro:bound G(*, .;.,*)}, we know that: 
\begin{align*}
    \G(x^{\star}, y ; x, y^{\star}) &\geq  -\frac{\beta_x}{2} \|x - x^{\star}\|^2 + \frac{\beta_x}{2} \|x_{\beta}(y) {\color{blue} - x + x } - x^{\star}\|^2 \\
    &= \frac{\beta_x}{2} \|x_{\beta}(y) - x\|^2 + \beta_x \left\langle x_{\beta}(y) - x, x - x^{\star} \right\rangle \\ 
    &\geq \beta_x \left\langle x_{\beta}(y) - x, x - x^{\star} \right\rangle \\
    &\stackrel{(CS)}{\geq} -\beta_x \|x_{\beta}(y) - x\| \|x - x^{\star}\| \\ 
    &\stackrel{(\ref{eqn:bound G})}{\geq} -\sqrt{2\beta_x \G(z)} \|x - x^{\star}\|   \smartqed 
\end{align*}
\end{proof}
This last lemma will play a crucial role in establishing an upper bound for the \acrlong{og} in terms of the \acrlong{sdg} as we will elucidate in the subsequent section.

%% file: OG_Bounds.tex
In our earlier discussion regarding the definition of an $\varepsilon$-solution, a solution is considered as an $\varepsilon$-solution if both the \acrlong{og} and the \acrlong{fe} are below $\varepsilon$. While the computation of the \acrlong{fe} is straightforward, the same cannot be said for the \acrlong{og}. Consequently, determining whether the \acrlong{og} is indeed less than $\varepsilon$ or not poses a more complicated challenge.

This section aims to establish computable approximations for the uncomputable measure, \textit{\acrlong{og}}, by setting upper bounds in terms of the aforementioned computable ones: the \acrshort{kkt}, \acrshort{pdg}, and \acrshort{sdg}. The initial approach involves attempting to set an upper bound on the optimality gap in terms of the \acrshort{kkt} defined in equation (\ref{eqn:KKT error}). That is: 
\begin{equation*}
    \Op(x) = f(x) - f^{\star} \stackrel{?}{\leq} \W(\K(z))
\end{equation*}
For instance, this $\W$ could be $\K^2$ or $c\K$ for any scalar $c$, and so forth. Consequently, if $\W(\K(z)) \leq \varepsilon$, it follows that $\Op(x) \leq \varepsilon$ as well. Therefore, we will attain an $\varepsilon_{g}-$solution with $\varepsilon \leq \varepsilon_g$, depending on the tightness of our subsequent bounds. 

One possible starting point is by taking a vector that satisfies the stationarity property of the Lagrangian:
\begin{align*}
    q \in \partial f(x) + A^Ty &\iff \forall u \in \X, ~ f(u) \geq f(x) + \langle q - A^Ty, u - x \rangle  \\ 
    &~ \stackrel{u = x^{\star}}{\Longrightarrow} ~ f(x) - f(x^{\star}) \leq \langle q - A^Ty, x - x^{\star} \rangle \\ 
    & \iff f(x) - f(x^{\star}) \leq \langle q, x - x^{\star} \rangle + \langle -y, A(x - x^{\star}) \rangle \\
    & \stackrel{Ax^{\star} = b}{\iff} f(x) - f(x^{\star}) \leq \underbrace{\|q\| {\color{blue}  \|x - x^{\star}\|} + \|y\| \|Ax - b\|}_{\text{Initial bound}} \stepcounter{equation}\tag{\theequation}\label{eqn:initial bound kkt}
\end{align*}
Yet, our efforts have resulted in an expression that remains depending on the unknown quantity, $x^{\star}$. This implies that unless we can eliminate this problematic term, our progress will not surpass the limitations of the optimality gap itself. Furthermore, when we attempted a similar approach with the \textit{\acrlong{pdg}} and the \textit{\acrlong{sdg}}, we encountered an identical issue:
\begin{itemize}
    \item Projected duality gap
        \begin{equation*}
        f(x) - f^{\star} \leq |f(x) + f^*(a) + \langle b, y \rangle | + \|x^{\star}\|\left\|a + A^Ty\right\| \leq \left(1 + {\color{blue} \|x^{\star}\|}\right) \sqrt{\D(z)}
        \end{equation*}
    \item Smoothed duality gap
        \begin{align*}
        f(x) - f^{\star} &= \begin{multlined}[t] f(x) + \left\langle Ax - b, y \right\rangle + \frac{1}{2\beta_y} \left\|Ax - b\right\|^2 - f^{\star}  - \left\langle Ax - b, y \right\rangle \\ - \frac{1}{2\beta_y} \left\|Ax - b\right\|^2 \end{multlined} \\ 
        &= \G(z) {\color{red} - \G(x^{\star}, y ; x, y^{\star})} - \left\langle Ax - b, y \right\rangle - \frac{1}{2\beta_y} \left\|Ax - b \right\|^2 \\
        &\leq \G(z) {\color{red} + \sqrt{2\beta_x\G(z)}} {\color{blue} \|x - x^{\star}\|} - \dots \dots 
        \end{align*} 
\end{itemize}
In each scenario, we continuously encountered the presence of the blue annoying unknown term. Nevertheless, we successfully transformed the initial bounds of the optimality gap into final ones by introducing additional assumptions that effectively eliminate the aforementioned blue annoying term.

Just to remark, more detailed proofs of each step of what we did in our initial bounds will be presented later on in this section.

\subsection{Regularity assumptions} 
In this subsection, we present the regularity assumptions that we employ to eliminate the blue annoying term from our initial bounds of the optimality gap. The first assumption is the \textit{metric sub-regularity} of the sub-differential of the Lagrangian, it was first introduced in the works of \cite{MSR}. The approach involves applying the definition of metric sub-regularity \textit{(Definition \ref{def:MSR})} to the sub-differential of the Lagrangian of (\ref{Saddle pt. prob.}) ($i.e. ~ F = \partial \Lag$ and $v = 0$). Formally, this can be expressed as follows:
\begin{assumption}[\acrshort{MSRSDL}, \cite{MSR}] \label{assump:MSR}

 The \acrfull{msrsdl} assumes that: there exists $\gamma > 0$ such that 
    \begin{equation}  \label{eqn:MSR}
        \left\|\partial_x \Lag(x, y)\right\|_0 + \left\|\nabla_y \Lag(x, y)\right\| \geq \gamma \dist(x, \X^{\star}) + \gamma \dist(y, \Y^{\star})
    \end{equation}
\end{assumption} 
As we can observe, assuming the metric sub-regularity of the sub-differential of the Lagrangian provides an upper bound for the earlier identified blue annoying term in terms of the KKT error.

The second assumption we introduce is referred to as the \textit{quadratic error bound of the smoothed gap}. While the \textit{quadratic error bound} is a widely recognized and commonly used assumption in general contexts, its application to the smoothed duality gap is relatively recent, first introduced in \cite{QEB_fercoq}. It's as broadly applicable as the metric sub-regularity of the Lagrangian's sub-differential, and it also serves as an upper bound for the blue annoying term in terms of the smoothed duality gap.
\begin{assumption}[\acrshort{QEBSG}, \textit{Proposition 15 (iii)} in  \cite{QEB_fercoq}] \label{assump:QEBSG}
    
    The \acrfull{qebsg} assumes that: there exists $\beta = (\beta_x, \beta_y) \in ]0, +\infty]^2, \eta > 0$ and a region $\mathcal{R} \subseteq \Z$ such that $\G$ has a \textit{quadratic error bound} with constant $\eta$ in the region $\mathcal{R}$. Said otherwise, for all $z \in \mathcal{R}$: 
    \begin{equation} \label{eqn:QEBSG}
        \G(z) \geq \frac{\eta}{2} \dist\left(z, \Z^{\star}\right)^2 
    \end{equation}    
\end{assumption}
\subsection{Final bounds} \label{subsec:final bounds}
In the subsection, we present our final bounds of the \acrlong{og} by assuming the aforementioned regularity assumptions. 
\begin{theorem} \label{thm:OG-KKT}
    Let $f \in \Gamma_0(\X), x^{\star} = \proj_{\X^{\star}}(x)$, and $q = \left\|\partial f(x) + A^Ty\right\|_0$. Then, under Assumption \ref{assump:MSR}, the \acrlong{og} defined in (\ref{eqn:OGFG}) could be approximated in terms of the \acrlong{kkt} defined in (\ref{eqn:KKT error}) for (\ref{Saddle pt. prob.}). More precisely, for all $z \in \Z$, 
    \begin{equation} \label{eqn:OG-KKT}
        f(x) - f^{\star} \leq  \frac{2}{\gamma} \K(z) + \|y\| \sqrt{\K(z)} 
    \end{equation}
\end{theorem}
\begin{proof}
 As we have seen in our initial bound (eqn:\ref{eqn:initial bound kkt}), taking $q \in \partial f(x) + A^Ty$ yields: 
\begin{align*}
    f(x) - f^{\star} &\leq \|q\| {\color{blue} \|x - x^{\star}\|} + \|y\| \|Ax - b\| \\
    &\stackrel{(\ref{eqn:MSR})}{\leq} \left\|\partial f(x) + A^Ty\right\|_0 {\color{blue} \frac{1}{\gamma}\left(\|\partial f(x) + A^Ty\|_0 + \|Ax - b\|\right)} + \|y\| \|Ax - b\| \\ 
    &= \frac{1}{\gamma}{\color{red} \left\| \partial f(x) + A^Ty \right\|^2_0} + \frac{1}{\gamma} {\color{red} \left\|\partial f(x) + A^Ty\right\|_0 } {\color{ao(english)} \|Ax - b\|} + \|y\| {\color{ao(english)} \|Ax - b\|} \\
    & \stackrel{(\ref{eqn:KKT error})}{\leq} \frac{1}{\gamma}{\color{red} \K(z)} + \frac{1}{\gamma} {\color{red} \sqrt{\K(z)}}{\color{ao(english)} \sqrt{\K(z)}} + \|y\| {\color{ao(english)} \sqrt{\K(z)}} \\ 
    &= \frac{2}{\gamma} \K(z) + \|y\| \sqrt{\K(z)} \smartqed
\end{align*}
\end{proof}
\begin{remark}
An interesting observation in this first finding pertains to the term $\left(\frac{2}{\gamma} \K(z)\right)$ that depends on $\gamma$. The metric sub-regularity constant, $\gamma$, typically takes very small values, such as $10^{-8}$, $10^{-10}$, or even smaller. So, having $\frac{1}{\gamma}$ multiplied with $\K(z)$ rather than $\sqrt{\K(z)}$ yields a tighter and more efficient bound where the algorithm will require fewer iterations to beat $\frac{1}{\gamma}$ before identifying an $\varepsilon-$solution.
\end{remark}
\begin{counterex}
In Theorem \ref{thm:OG-KKT}, we derived an upper bound for the \acrlong{og} based on the \acrshort{kkt}. It is important to note, however, that the reverse relationship may not always hold. The following counterexample illustrates this point:

Let $\varepsilon > 0$, and consider the unconstrained optimization problem $\displaystyle \min_{x \in \R} f(x)$ where  $f \colon \R \rightarrow \R$ is defined as: 
\begin{equation*}
    f(x) = \begin{cases}
        x &\text{If}~ x > \varepsilon \\
        \frac{x^2}{2\varepsilon} + \frac{\varepsilon}{2} &\text{If}~ x \leq \varepsilon
    \end{cases}
\end{equation*}
Then, 
\begin{equation*}
    f'(x) = \begin{cases}
        1 &\text{If}~ x > \varepsilon \\
        \frac{x}{\varepsilon} &\text{If}~ x \leq \varepsilon
    \end{cases}
\end{equation*}
One can observe that, $\forall \varepsilon > 0$: 
\begin{align*}
        f'(\varepsilon) = 1 && \textbf{but} &&  
        f(\varepsilon) - f(0) = \frac{\varepsilon^2}{2\varepsilon} + \frac{\varepsilon}{2} - \frac{\varepsilon}{2} = \frac{\varepsilon}{2} \xrightarrow[\varepsilon \rightarrow 0]{} 0 \smartqed
\end{align*}
which concludes the example.
\end{counterex}
\begin{theorem} \label{thm:OG-SDG}
    Let $f \in \Gamma_0(\X), x^{\star} = \proj_{\X^{\star}}(x)$, and $\beta = (\beta_x, \beta_y) \in (0, +\infty)^2$. Then, under Assumption \ref{assump:QEBSG}, the \acrlong{og} defined in (\ref{eqn:OGFG}) could be approximated in terms of the self-centered \acrlong{sdg} defined in (\ref{eqn:SDG}) for (\ref{Saddle pt. prob.}). More precisely, for all $z \in \Z$, 
       \begin{equation} \label{eqn:OG-SDG}
            f(x) - f^{\star} \leq \left(1 + \sqrt{\frac{2\beta_x}{\eta}}\right) \G(z) + \sqrt{2\beta_y}\|y\|\sqrt{\G(z)}
        \end{equation}
\end{theorem}
\begin{proof}
We start by rewriting the optimality gap in a decomposed way: 
\begin{align*}
    f(x) - f^{\star} &= \begin{multlined}[t] {\color{red} f(x) + \left\langle Ax - b, y \right\rangle + \frac{1}{2\beta_y} \left\|Ax - b\right\|^2 - f(x^{\star}) } - \left\langle Ax - b, y \right\rangle \\ - \frac{1}{2\beta_y} \left\|Ax - b\right\|^2 \end{multlined}\\ 
    &= {\color{red} \G(x, y^{\star} ; x^{\star}, y)} - \left\langle Ax -b, y \right\rangle - \frac{1}{2\beta_y} \left\|Ax - b\right\|^2  \\ 
    & \stackrel{(\ref{eqn:G decomposition})}{=} {\color{red} \G(z) - \G(x^{\star}, y ; x, y^{\star})} - \left\langle Ax - b, y \right\rangle - \frac{1}{2\beta_y} \left\|Ax - b \right\|^2 \\
    & \stackrel{(\ref{eqn:G(*, .;.,*) - sqrt eta})}{\leq} \G(z) {\color{red} + \sqrt{2\beta_x \G(z)}} {\color{blue} \|x - x^{\star}\|} - \left\langle Ax - b, y \right\rangle {\color{ao(english)}- \frac{1}{2\beta_y} \left\|Ax - b \right\|^2 }\\
    & \stackrel{(\ref{eqn:QEBSG})}{\leq} \G(z) + \sqrt{2\beta_x \G(z)} {\color{blue} \sqrt{\frac{2}{\eta} \G(z)}} - \left\langle Ax - b, y \right\rangle {\color{ao(english)}- 0}\\
    &\leq \left(1 + 2\sqrt{\frac{\beta_x}{\eta}}\right) \G(z) + \|y\| {\color{red} \left\|Ax - b\right\|}  \\
    &\stackrel{(\ref{eqn:FG-SDG})}{\leq} \left(1 + 2\sqrt{\frac{\beta_x}{\eta}}\right) \G(z) + \|y\|\sqrt{2\beta_y \G(z)}    \smartqed 
 \end{align*}
 \end{proof}
\begin{remark}
The quadratic error bound constant, $\eta$, exhibits a similar characteristic of taking very small values, akin to the metric sub-regularity constant. In Theorem \ref{thm:OG-SDG}, we managed to establish an upper bound that depends on the square root of $\eta$, while maintaining its multiplication with $\G(z)$ instead of $\sqrt{\G(z)}$. Nevertheless, it is noteworthy that an alternative bound could be formulated directly in terms of $\eta$ itself by using Corollary \ref{coro:bound G(*, .;.,*)}.
\end{remark}
\begin{theorem} \label{thm:OG-PDG}
    Let $f \in \Gamma_0(\X), x^{\star} = \proj_{\X^{\star}}(x)$, $\beta = (\beta_x, \beta_y) \in (0, +\infty)^2$, and $\beta_{\min} = \min(\beta_x, \beta_y)$. Then, under Assumption \ref{assump:QEBSG}, the \acrlong{og} defined in (\ref{eqn:OGFG}) could be approximated in terms of the \acrlong{pdg} defined in (\ref{eqn:PDG}) for (\ref{Saddle pt. prob.}). More precisely, for all $z \in \Z$, 
    \begin{equation} \label{eqn:OG-PDG}
        f(x) - f^{\star} \leq \left( 1 + \|x\| + \sqrt{\frac{2}{\eta}} \sqrt{ (1 + \|x\| + \|y\|) \sqrt{\D(z)} + \frac{1}{2\beta_{\min}}\D(z)} \right) \sqrt{\D(z)}
    \end{equation}
\end{theorem}
In the proof of this theorem, we will use some arguments that will come later in this paper. So, we will provide its proof later on as well.

\begin{remark}
Designing a new regularity assumption that inherently fits \acrshort{pdg} is out of the scope of the paper. Instead, we take advantage of assuming \acrshort{qebsg} to derive our result. 
\end{remark}

%% file: CB_intro.tex
Our subsequent objective is motivated by the belief that the newly proposed optimality measure, the \acrlong{sdg}, might serve as a more appropriate stopping criterion compared to the others. This belief comes to light from the smoothed duality gap's definition; whereas the \acrlong{kkt} relies on the sub-differential of the objective function to assess optimality, the \acrlong{sdg} is more directly based on the objective function itself. Hence, in this section, our goal is to conduct a comparative analysis between the \acrlong{sdg}, the \acrlong{kkt}, and the \acrlong{pdg}. We aim to investigate and identify the conditions under which these measures could serve as approximations for the \acrlong{sdg} and vice versa.

%% file: SDG_KKT_bounds.tex
\subsection{SDG -- KKT bounds} \label{subseq:SDG-KKT}
In this subsection, we start our comparative analysis between the \acrlong{sdg} and the \acrlong{kkt}. Initially, we present the upper bound obtained for \acrshort{SDG} in terms of the \acrshort{kkt}. Subsequently, we demonstrate the reverse relationship.
\begin{theorem} \label{thm:SDG-KKT}
    Given $\beta = (\beta_x, \beta_y) \in (0, +\infty)^2$, $z \in \Z$, and a function $f \in \Gamma_0(\X)$. Assume $\|q\| = \left\|\partial f(x) + A^Ty\right\|_0$. Then, for the \acrlong{kkt} and the \acrlong{sdg} defined, respectively, in (\ref{eqn:KKT error}) and (\ref{eqn:SDG}) we have: 
    \begin{align}
        \G(z) &\leq \underaccent{\bar}{\beta} \K(z) \label{eqn:SDG-KKT} \\ 
        \underaccent{\bar}{\beta} &= \max\left(\frac{1}{\beta_x}, \frac{1}{2\beta_y}\right) \label{eqn:beta_bar}
    \end{align}
\end{theorem}
\begin{proof}
First of all, let us observe the following: 
\begin{itemize}
    \item Using the definition of the sub-differential, we obtain: 
    \begin{align*}
        q \in \partial f(x) + A^Ty &\iff \forall u \in \R^n, ~ f(u) \geq f(x) + \langle q - A^Ty, u - x \rangle \\
        &\stackrel{u = p}{\Longrightarrow} f(x) - f(p) \leq \langle q - A^Ty, x - p \rangle \stepcounter{equation}\tag{\theequation}\label{SDG-KKT:fx-fp} \\ 
    \end{align*}
    \item From Lemma \ref{lem:p and fermat}, we know that $$ p = \prox_{\beta_x^{-1}f} \left(x - \frac{1}{\beta_x}A^Ty\right) \iff \beta_x(x - p) \in \partial f(p) + A^Ty$$ Hence, by Lemma \ref{lem:Sub-diff property}, for any $q \in \partial f(x) + A^Ty$, we get: 
    \begin{equation} \label{SDG-KKT:p-x and q}
    \begin{split}
         \langle \beta_x (x - p) - q, p - x \rangle \geq 0  &\iff \langle q, x - p \rangle \geq \beta_x \|p-x\|^2 \\ &\iff  \|q\| \geq \beta_x \|p - x\|
    \end{split}
    \end{equation}
\end{itemize}
Now, from the definition of the SDG: 
\begin{align*}
    \G(z) &= {\color{red} f(x) - f(p)} + \langle A(x - p), y \rangle - \frac{\beta_x}{2} \| p - x \|^2 + \frac{1}{2\beta_y}\|Ax - b\|^2 \\ 
    &\stackrel{(\ref{SDG-KKT:fx-fp})}{\leq} {\color{red} \langle q - A^Ty, x - p \rangle } + \langle A(x - p), y \rangle - \frac{\beta_x}{2} \| p - x \|^2 + \frac{1}{2\beta_y}\|Ax - b\|^2 \\ 
    &\leq \langle q, x - p \rangle + \frac{1}{2\beta_y}\|Ax - b\|^2 \\
    &\leq \|q\| {\color{red} \|x - p\|} + \frac{1}{2\beta_y}\|Ax - b\|^2 \\
    &\stackrel{(\ref{SDG-KKT:p-x and q})}{\leq} \|q\| {\color{red} \frac{1}{\beta_x} \|q\|} + \frac{1}{2\beta_y}\|Ax - b\|^2 \\
    &\stackrel{(\ref{eqn:beta_bar})}{\leq} \underaccent{\bar}{\beta} \left[\left\|\partial f(x) + A^Ty\right\|^2 + \|Ax - b\|^2\right] \\
    &= \underaccent{\bar}{\beta} \K(z) \smartqed 
\end{align*}
\end{proof}
\begin{theorem} \label{thm:KKT-SDG}
    Given $\beta = (\beta_x, \beta_y) \in (0, +\infty)^2$, $z \in \Z$, and a differentiable function $f \in \Gamma_0(\X)$ that has an $L-$Lipschitz gradient. Then, for the \acrlong{kkt} and the \acrlong{sdg} defined, respectively, in (\ref{eqn:KKT error}) and (\ref{eqn:SDG}) we have: 
    \begin{align}
        \K(z) &\leq \Bar{\beta}_L \G(z) \label{eqn:KKT-SDG}\\ 
        \Bar{\beta}_L &= \max\left(\frac{2(L + \beta_x)^2}{\beta_x}, 2\beta_y \right) \label{eqn:beta_L}
    \end{align}
\end{theorem}
\begin{proof}
 We start by observing the following: 
\begin{itemize}
    \item From Lemma \ref{lem:p and fermat}, we know that $$ p = \prox_{\beta_x^{-1}f} \left(x - \frac{1}{\beta_x}A^Ty\right) \iff \beta_x(x - p) \in \partial f(p) + A^Ty$$ Hence, keeping in mind that $f$ is differentiable, we get:
    \begin{equation} \label{KKT-SDG:norm of x-p}
        \beta_x\|x - p\| = \left\|\nabla f(p) + A^Ty\right\| 
    \end{equation}
    \item Since the gradient of $f$ is $L-$Lipschitz, then by the reverse triangle inequality, we get: 
    \begin{align*}
        \left\|\nabla f(x) + A^Ty \right\| - \left\|\nabla f(p) + A^Ty \right\| &\leq \left\|\left(\nabla f(x) + A^Ty\right) - \left(\nabla f(p) + A^Ty\right) \right\| \\
        &\leq L \|x - p\| \\ 
    \end{align*}
    Hence, 
    \begin{align*}
      \left\|\nabla f(x) + A^Ty \right\| & \leq L \|x - p\| + {\color{red} \left\|\nabla f(p) + A^Ty \right\|} \\ 
        &\stackrel{(\ref{KKT-SDG:norm of x-p})}{\leq} L \|x - p \| + {\color{red} \beta_x \|x - p\|} \\ 
        &= (L + \beta_x) \|x - p\| \stepcounter{equation}\tag{\theequation}\label{KKT-SDG:x-p lower bound}
    \end{align*}
    Therefore, starting from Lemma \ref{lem:G lower bound}, we get: 
    \begin{align*}
        \G(z) &\geq {\color{red} \frac{\beta_x}{2} \|x - p\|^2} + \frac{1}{2\beta_y} \|Ax - b\|^2 \\ 
        & \stackrel{(\ref{KKT-SDG:x-p lower bound})}{\geq} \frac{\beta_x}{2} {\color{red} \frac{\|\nabla f(x) + A^Ty \|^2}{(L + \beta_x)^2}} + \frac{1}{2\beta_y} \|Ax - b\|^2 \\ 
        &{\geq} \min\left(\frac{\beta_x}{2(L + \beta_x)^2}, \frac{1}{2\beta_y} \right) \left[\|\nabla f(x) + A^Ty \|^2 + \|Ax - b\|^2\right] 
        \end{align*}
        which implies the result: 
        \begin{equation*}
            \K(z) \leq \bar{\beta}_L \G(z) \smartqed 
        \end{equation*}
\end{itemize}
\end{proof}
\begin{counterex}

Theorem \ref{thm:KKT-SDG} necessitates the assumption that the objective function is differentiable and has an $L$-Lipschitz gradient. Without these conditions, Theorem \ref{thm:KKT-SDG} may not hold. This is exemplified in this counterexample.

Let $\beta = (\beta_x, \beta_y) = (1, 1)$, and consider the unconstrained optimization problem $\displaystyle \min_{x \in \R} |x|$.

Then, the associated Lagrangian is $\Lag(z) = \Lag(x) = |x|$. Hence, 
\begin{itemize}
    \item  The smoothed duality gap is defined as: 
    \begin{align*}
    \G(z) = \G(x) &= |x| - \min_{x'} |x'| + \frac{1}{2} (x' - x)^2 \\
    &= |x| - \left|\prox_{|.|}(x)\right| - \frac{1}{2} \left(\prox_{|.|}(x) - x\right)^2 \\ 
    &= |x| - \left|[|x| - 1]_+\sgn(x)\right| - \frac{1}{2} \left([|x| - 1]_+\sgn(x) - x\right)^2 \\
     &= |x| - [|x| - 1]_+ - \frac{1}{2} \left( [|x| - 1]_+\sgn(x) - x\right)^2
\end{align*}
\item The KKT error is defined as: 
\begin{equation*}
    \K(z) = \K(x) = \|\partial f(x) \|_0^2 ~\text{with}~ \partial f(x) = \begin{cases}
        -1 & \text{If}~ x < 0 \\ [-1, 1] & \text{If}~ x = 0 \\ 1 &\text{If}~ x> 0
    \end{cases}
\end{equation*}
Then, 
\begin{align*}
    \lim_{x \rightarrow 0} \K(x) = 1 && \textbf{but} && \lim_{x \rightarrow 0} \G(x) = \lim_{x \rightarrow 0} |x| - \frac{1}{2}x^2 = 0 \smartqed
\end{align*}
which concludes the example.
\end{itemize}
\end{counterex}

%% file: SDG_PDG_bounds.tex
\subsection{SDG -- PDG bounds}
In this subsection, we proceed with our comparative analysis, this time between the \acrlong{sdg} and the \acrlong{pdg}. Firstly, we present the upper bound acquired for the \acrlong{sdg} in terms of the \acrlong{pdg}. Following this, we illustrate the converse relationships. Furthermore, we will revisit Theorem \ref{thm:OG-PDG} and derive its proof.
\begin{theorem} \label{thm:SDG-PDG}
    Given $\beta = (\beta_x, \beta_y) \in (0, +\infty)^2$, $z \in \Z$, and a function $f \in \Gamma_0(\X)$. Let $\beta_{\min} = \min(\beta_x, \beta_y)$. Then, for the \acrlong{pdg} and the \acrlong{sdg} defined, respectively, in (\ref{eqn:PDG}) and (\ref{eqn:SDG}) we have: 
    \begin{equation} \label{eqn:SDG-PDG}   
        \G(z) \leq (1 + \|x\| + \|y\|) \sqrt{\D(z)} + \frac{1}{2\beta_{\min}}\D(z) 
    \end{equation}
\end{theorem}
\begin{proof}
By the Fenchel-conjugate definition, we have: 
\begin{equation} \label{SDG-PDG:fp bound}
    f^*(a) = \sup_x ~~ \langle a, x \rangle - f(x) \stackrel{x = p}{\geq} \langle a, p \rangle - f(p) \Longrightarrow - f(p) \leq f^*(a) - \langle p, a \rangle 
\end{equation}
Thus, 
\begin{align*}
    \G(z) &= f(x) {\color{red} - f(p)} + \left\langle A(x - p), y \right\rangle - \frac{\beta_x}{2}\|x - p\|^2 + \frac{1}{2\beta_y} \|Ax - b\|^2 \\ 
    &\stackrel{(\ref{SDG-PDG:fp bound})}{\leq} f(x) {\color{red} + f^*(a)} + \left\langle Ax, y \right\rangle - \left\langle p, {\color{red}a} + A^Ty \right\rangle -\frac{\beta_x}{2}\|x - p\|^2 + \frac{1}{2\beta_y} \|Ax - b\|^2\\
    &= \begin{multlined}[t] f(x) + f^*(a) + \left\langle Ax, y \right\rangle - \left\langle {\color{blue} x}, a + A^Ty \right\rangle + \left\langle {\color{blue}x} - p, a + A^Ty \right\rangle \\ -\frac{\beta_x}{2}\|x - p\|^2 + \frac{1}{2\beta_y} \|Ax - b\|^2\end{multlined} 
\end{align*}
Now, by making use of Young's inequality \textit{(Proposition \ref{props:YI})} with $\mathbf{u} = x - p, \\ \mathbf{v} = a + A^Ty$ and $\lambda = \sqrt{\beta_x}$, we get: 
\begin{equation} \label{SDG-PDG:young}
    \left\langle x - p, a + A^Ty \right\rangle \leq \frac{\beta_x}{2} \|x - p\|^2 + \frac{1}{2\beta_x}\left\|a + A^Ty\right\|^2 
\end{equation}
Therefore, 
\begin{align*}
    \G(z) &= \begin{multlined}[t] f(x) + f^*(a) + \left\langle Ax, y \right\rangle - \left\langle x, a + A^Ty \right\rangle + {\color{red} \left\langle x- p, a + A^Ty \right\rangle} \\ -\frac{\beta_x}{2}\|x - p\|^2 + \frac{1}{2\beta_y} \|Ax - b\|^2 \end{multlined} \\ 
    &\stackrel{(\ref{SDG-PDG:young})}{\leq} \begin{multlined}[t] f(x) + f^*(a) + \left\langle Ax, y \right\rangle - \left\langle x, a + A^Ty \right\rangle + {\color{red} \frac{\beta_x}{2} \|x - p\|^2} \\ {\color{red} + \frac{1}{2\beta_x}\left\|a + A^Ty\right\|^2 }  -\frac{\beta_x}{2}\|x - p\|^2  + \frac{1}{2\beta_y} \|Ax - b\|^2\end{multlined} \\ 
    &= \begin{multlined}[t]  f(x) + f^*(a) {\color{blue} + \langle b, y\rangle} + \left\langle Ax {\color{blue} -b}, y \right\rangle - \left\langle x, a + A^Ty \right\rangle + \frac{1}{2\beta_x}\left\|a + A^Ty\right\|^2  \\ + \frac{1}{2\beta_y} \|Ax - b\|^2 \end{multlined} \\ 
    &\leq \begin{multlined}[t] |f(x) + f^*(a) + \langle b, y\rangle| + \left\|Ax - b\right\| \|y\| + \|x\| \left\|a + A^Ty\right\|\\  + \frac{1}{2\beta_x}\left\|a + A^Ty\right\|^2  + \frac{1}{2\beta_y} \|Ax - b\|^2 \end{multlined} \\ 
    &\leq \begin{multlined}[t]
     |f(x) + f^*(a) + \langle b, y\rangle| + \left\|Ax - b\right\| \|y\| + \|x\| \left\|a + A^Ty\right\| \\ + \frac{1}{2\beta_{\min}}\left(\left\|a + A^Ty\right\|^2 + \left\|Ax - b\right\|^2\right) \end{multlined} \\ 
    &\stackrel{(\ref{eqn:PDG})}{\leq} (1 + \|x\| + \|y\|) \sqrt{\D(z)} + \frac{1}{2\beta_{\min}} \D(z)    \smartqed
\end{align*}
\end{proof}
Now, we will derive Theorem \ref{thm:OG-PDG}. You'll notice in the proof that at a certain stage, we will establish a bound in terms of the \acrlong{sdg}. Therefore, it was necessary to acquire an approximation for the \acrlong{sdg} in terms of the \acrlong{pdg} before completing the proof. With the assistance of Theorem \ref{thm:SDG-PDG}, we now possess this approximation and can proceed to derive Theorem \ref{thm:OG-PDG}.
\begin{proof}
Starting with the definition of the Fenchel-Conjugate function at an optimal solution, we get: 
\begin{align*}
    f(x^{\star}) = \sup_{\mu \in \X} ~ \langle \mu, x^{\star} \rangle - f^*(\mu) & \stackrel{\mu = a}{\geq} \langle a, x^{\star} \rangle - f^*(a) \\ 
    &= \left\langle x^{\star}, {\color{red} -A^Ty} \right\rangle - f^*(a) + \left\langle x^{\star}, a {\color{red} + A^Ty} \right\rangle \\ 
    & \geq - \langle b, y \rangle - f^*(a) - \|x^{\star}\|\left\|a + A^Ty\right\|
\end{align*}
Thus, 
\begin{align*}
    f(x) - f^{\star} &\leq |f(x) + f^*(a) + \langle b, y \rangle | + \|x^{\star}\|\left\|a + A^Ty\right\| \\ 
    &\stackrel{(\ref{eqn:PDG})}{\leq} \left(1 + \|x^{\star}\|\right) \sqrt{\D(z) } \\ 
    &\leq ( 1 + \|x\| + \|x - x^{\star}\|) \sqrt{\D(z)} \\
    &\stackrel{(\ref{eqn:QEBSG})}{\leq} \left(1 + \|x\| + \sqrt{\frac{2\G(z)}{\eta}}\right) \sqrt{\D(z)} \\ 
    &\stackrel{(\ref{eqn:SDG-PDG})}{\leq} \left( 1 + \|x\| + \sqrt{\frac{2}{\eta}} \sqrt{ (1 + \|x\| + \|y\|) \sqrt{\D(z)} + \frac{1}{2\beta_{\min}}\D(z)} \right) \sqrt{\D(z)} \smartqed
\end{align*}
\end{proof}
Our latest finding establishes the upper bound of the \acrlong{pdg} in terms of the \acrlong{sdg}. This last part is quite technical and incorporates manifold and convex optimization concepts. Therefore, we outline the main result here, deferring the detailed proofs to Appendix \ref{appendix:PDG in terms of SDG}.
\begin{theorem} \label{thm:PDG-SDG -- manifold}
Given $\beta = (\beta_x, \beta_y) \in (0, +\infty)^2$, $z \in \Z$, and a function $f \in \Gamma_0(\X)$. Then, under the following set of assumptions (we denote it $\mathcal{E}$):
\begin{itemize}
    \item The Fenchel-Conjugate of the objective function, $f^*$, could be written in a separable way:
    \begin{equation*}  f^*(\mu) = f_1^*(\mu_1) + f_2^*(\mu_2), \hspace{1.5cm} \mu \in \X \end{equation*}
    \item $f_1^*$ is $L_{f_1^*}-$Lipschitz on its domain, $\dom f_1^*$. 
    \item The domain of $f^*_2$ is a non-empty affine space. 
    \item Let $\mu_0 \in \dom f_2^*$, then $\forall \mu_2 \in\dom f_2^*$, we define 
    \begin{equation*}  g(\lambda) = f_2^*\left(\mu_0 + \phi^{-1}(\lambda)\right) = f_2^*(\mu_2)\end{equation*} where $\phi$ is the diffeomorphism defined in Lemma \ref{lem:diffeom of affine space} (eqn:\ref{eqn:phi}) in Appendix \ref{appendix:PDG in terms of SDG}. 
    \item The function $g$ is differentiable and has an $L_g-$Lipschitz gradient. 
\end{itemize}
The \acrlong{pdg} and the \acrlong{sdg} defined, respectively, in (\ref{eqn:PDG}) and (\ref{eqn:SDG}) satisfy: 
\begin{align}\label{eqn:PDG<SDG -- manifold}
    \D(z) &\leq \begin{multlined}[t] \left(\left(3 + \beta_x L_g\right)\G(z) + \left(\sqrt{2\beta_x}\left(2\|x\| + L_{f_1^*}\right)+ \sqrt{2\beta_y}\|y\|\right)\sqrt{\G(z)}\right)^2 \\ + 2\beta_{\max} \G(z)\end{multlined} \\ 
    a &= \proj_{\dom f^*} \left(-A^Ty\right)  \hspace{2cm} p= \prox_{\beta_x^{-1}f}\left(x - \frac{1}{\beta_x} A^Ty\right) 
\end{align}
\end{theorem}

Next, we will illustrate our theoretical findings through numerical experiments. The directed graph depicted in Figure \ref{fig:results summary} summarizes our theoretical findings and highlights their main assumptions.
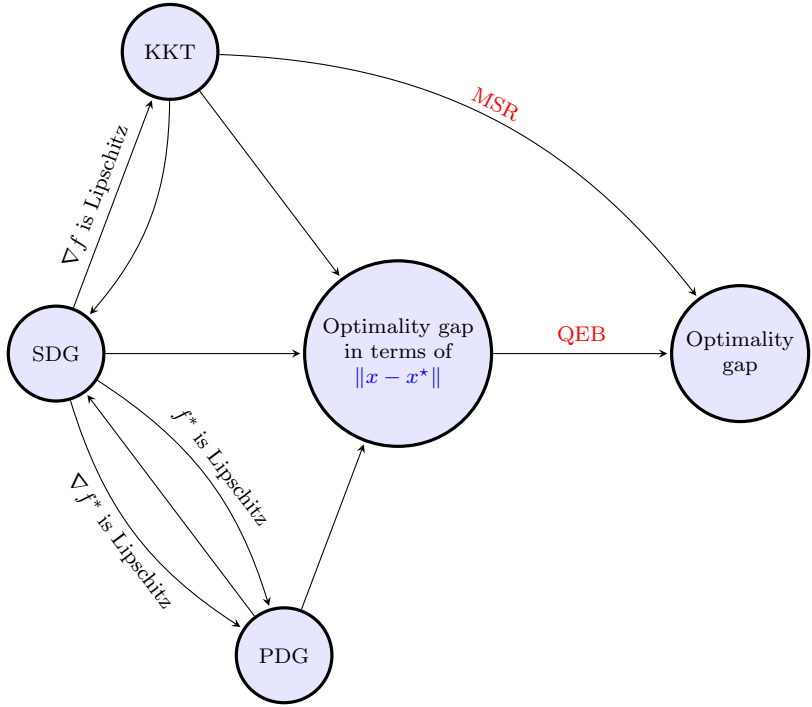
\begin{figure}[htbp]
    \centering
    \begin{tikzpicture}%
  [>=stealth,
   shorten >=1pt,
   node distance=4cm,
   on grid,
   auto,
   every state/.style={draw=black, fill=blue!10, very thick, align=center},
  ]
\node[state] (x_star) at (0, 0) {Optimality gap \\ in terms of \\ {\color{blue}$\|x - x^{\star}\|$}};
\node[state] (OG) [right=4.5cm of x_star] {Optimality \\ gap};
\node[state, text width = 1cm] (kkt) [above left=4cm and 3cm of x_star] {KKT};
\node[state, text width = 1cm] (pdg) [below left=4cm and 1.5cm of x_star] {PDG}; 
\node[state, text width = 1cm] (sdg) [left=4.5cm of x_star] {SDG};

\path[->]
   (sdg)      edge                  node[sloped]        {$\nabla f$ is Lipschitz}   (kkt) 
              edge[bend left=20]    node[sloped]        {$f^*$ is Lipschitz}        (pdg)
              edge[bend right=20]   node[sloped, swap]  {$\nabla f^*$ is Lipschitz} (pdg)
              edge                  node                {}                          (x_star)
   (kkt)      edge                  node                {}                          (x_star)
              edge[bend left=20]    node                {}                          (sdg)
              edge[bend left=25]    node[sloped]        {\acrshort{msr}}            (OG)
   (pdg)      edge                  node[sloped]        {}                          (sdg) 
              edge                  node                {}                          (x_star)
   (x_star)   edge                  node[sloped]        {\acrshort{qeb}}            (OG) 
   ;
\end{tikzpicture}
    \caption{Summary of our findings, the representation: $\mathcal{M}_1 \xlongrightarrow{\mathcal{S}} \mathcal{M}_2$ means that the measure $(\mathcal{M}_1)$ provides an upper-bound for the measure $\mathcal{M}_2$ under the assumption $\mathcal{S}$}
    \label{fig:results summary}
\end{figure}

%% file: EXP_intro.tex
In the last section, we present a series of numerical experiments aimed at illustrating and validating our findings and their tightness. Initially, we investigate linearly-constrained least-squares problems where the objective is convex and smooth. Subsequently, we extend our findings beyond affine equality constraints, demonstrating their applicability to more generalized scenarios involving inequality constraints. Finally, we explore a non-smooth problem that showcases the superior stability of the smoothed duality gap compared to the Karush-Kuhn-Tucker error.

One could observe that while the measures of optimality themselves are not complex to implement programmatically, each one necessitates specific information that may be tricky. For example, computing the sub-differential of the objective function, crucial for the \acrlong{kkt}, can be exceptionally complicated for certain functions. Similarly, while the \acrlong{sdg} prefers an easily computable proximal operator of the objective, the \acrlong{pdg} favors an objective with a tractable Fenchel-conjugate, allowing projection onto its domain. Moreover, our results rely on additional assumptions such as the \acrlong{msrsdl} and \acrlong{qebsg}, requiring the determination of associated constants, a task often non-trivial in itself. Hence, for each experiment, we will provide all non-trivial computations along with detailed steps wherever necessary.

\textbf{Primal-Dual Hybrid Gradient:} we implement the \textit{Primal-Dual Hybrid Gradient} (PDHG) algorithm \cite{PDHG,PDHG1,PDHG2,PDHG3} to tackle our experiments. We chose it because it's famous for its efficiency in handling large-scale problems due to its low cost per iteration, the PDHG algorithm enjoys linear convergence on the problems under consideration, which makes it more insightful for showing the numerical performance of the measures we consider. Additionally, it provides primal-dual solutions at each iteration.

\begin{table}[!ht]
\centering
\begin{tabular}{||c||}
\hline 
\multicolumn{1}{||l||}{\textbf{Algorithm:} Primal-Dual Hybrid Gradient (PDHG)} \\
\hline \hline
{$\!\begin{aligned} 
\bar{x}_{k+1} &= \prox_{\tau f}\left(x_k - \tau A^Ty_k\right) \\
\bar{y}_{k+1} &= y_k + \sigma \left(A\bar{x}_{k+1} - b\right)\\
x_{k+1} &= \bar{x}_{k+1} - \tau A^T\left(\bar{y}_{k+1} - y_k\right)  \\
y_{k+1} &= \bar{y}_{k+1} \\
\end{aligned}$} \\
  \hline 
\end{tabular}
\end{table}

The convergence of the algorithm is guaranteed when the step sizes $\tau$ and $\sigma$ satisfy the inequality: 
\begin{equation*}
    \tau \sigma \|A\|^2 < 1 
\end{equation*}
Therefore, we choose the following step sizes: 
\begin{align*}
    \tau = \frac{0.95}{\|A\|} && \sigma = \frac{1}{\|A\|}
\end{align*}
\textbf{Smoothing parameter selection:} to determine the smoothing parameter $\beta \in ]0, +\infty[^2$, we adopt the following approach:
\begin{itemize}
    \item We set the same primal-dual smoothing parameters, i.e., $\beta_x = \beta_y$.
    \item We construct a predefined geometric list of values for $\beta$, denoted by $\mathcal{I}$, with a total of 41 values. This list includes the \acrlong{fe} at the current iteration $\|Ax_k - b\|$, along with 40 equally-divided values in logarithmic scale ranging from $10^{-8}$ to 100:
    \begin{equation*}
    \mathcal{I} := \left\{10^{-8}, 1.91 \times 10^{-7}, \dots, 5.25 \times 10^{-1}, 100, \|Ax - b\|\right\}
    \end{equation*}
    \item Then, for each iteration $k \in \mathbb{N}$, we assign a value for $\beta$ from $\mathcal{I}$ according to the following criteria:
    \begin{itemize}
        \item For results of the form: $g(z_k) \leq h_{\beta}(z_k)$, we choose: 
        \begin{equation*}
            \Tilde{\beta} = \arg\min_{\beta \in \mathcal{I}} h_{\beta}(z_k) 
        \end{equation*}
        \item For results of the form: $g_{\beta}(z_k) \leq h_{\beta}(z_k)$, we choose:
        \begin{equation*}
            \Tilde{\beta} = \arg\min_{\beta \in \mathcal{I}} \frac{h_{\beta}(z_k)}{g_{\beta}(z_k)}
        \end{equation*}
    \end{itemize}
\end{itemize}
\textbf{Computer configurations:} our numerical results were generated on macOS using an Apple M1 Pro processor with 16 GB of RAM. We employed Python 3.10.9 software for the computations.

%% file: LCLS.tex
\subsection{Linearly-Constrained Least-Squares (LC-LS)}

We are interested in illustrating several instances of the following Least-Squares problem \cite{LCLS1,LCLS}: 
\begin{equation} \tag{LC-LS} \label{eqn: LS -- prob}
\begin{split}
    \min_{x \in \R^n} ~  \frac{1}{2} \|Qx - c\|^2   \hspace{1cm} \text{subject to} \hspace{1cm} Ax = b
\end{split}
\end{equation}
or, equivalently, the associated saddle point problem: 
\begin{equation*}
    \min_{x \in \R^n} \max_{y \in \R^m} \Lag(x, y) = \frac{1}{2} \|Qx - c\|^2 + \langle Ax - b, y \rangle 
\end{equation*}
where the objective function $f(x) = \frac{1}{2} \|Qx - c\|^2$ is smooth $\left( \nabla f(x) = Q^T\left(Qx - c\right) \right)$, and convex $\left( \nabla^2 f(x) =  Q^TQ \succcurlyeq 0 \right)$. The matrix $A \in \R^{m \times n}$ and the vectors $c ~\& ~ b \in \R^m$ are given. Before getting started with the instances, let's analyze the problem. 

\subsubsection{Problem Analysis}
\begin{enumerate}
    \item \textbf{The gradient}, $\nabla f$, is $L-$Lipschitz with $L = \left\|Q^TQ\right\|_{op} = \lambda_{\max}\left(Q^TQ\right)$
    \item \textbf{The proximal operator is:} 
    \begin{align*}
        \bar{u} = \prox_{sf}(x) &= \arg\min_{u \in \R^n} \frac{1}{2} \left\|Qu - c\right\|^2 + \frac{1}{2s} \|u - x\|^2 \\ 
        &\iff 0 \in \partial\left(\frac{1}{2} \left\|Q . - c\right\|^2 + \frac{1}{2} \|. - x\|^2 \right)(\bar{u}) \\ 
        &\iff \bar{u} = \left(Q^TQ + \frac{1}{s} \Id \right)^{-1} \left(Q^Tc + \frac{1}{s}x\right) \stepcounter{equation}\tag{\theequation}\label{eqn: LS -- prox}
    \end{align*}
    \item \textbf{The Fenchel-Conjugate is}: $\displaystyle f^*(\mu) = \sup_{x \in \R^n} \langle \mu, x \rangle - \frac{1}{2} \left\|Qx - c\right\|^2$, let: 
    \begin{align*}
        \bar{x} &= \arg\max_{x \in \R^n} \langle \mu, x \rangle - \frac{1}{2} \left\|Qx - c\right\|^2 \\
            &\iff 0 \in \partial\left(\langle \mu, . \rangle - \frac{1}{2} \left\|Q . - c\right\|^2 \right)(\bar{x}) \\ 
            &\iff Q^TQ\bar{x} = Q^Tc + \mu
    \end{align*}
    There are two cases: 
    \begin{itemize}
        \item If $\mu \notin \Ran\left(Q^T\right)$, which means that $\mu = \mu_1 + \mu_2$ such that $\mu_1 \in \Ran\left(Q^T\right)$ and $\mu_2 \in \ker(Q)$, then one can find a maximizing sequence defined as: 
        \begin{equation*}
            x_n = n \mu_2, \hspace{2cm} \forall n \in \mathbb{N}
        \end{equation*}
        Then, 
        \begin{align*}            
        f^*(x_n) &= \sup_{n \in \mathbb{N}}~ \langle \mu, x_n \rangle - \frac{1}{2} \left\|Qx_n - c \right\|^2 \\ 
        &= \sup_{n \in \mathbb{N}}~ n\|\mu_2\|^2 - \frac{1}{2}\|c\|^2 \xrightarrow[n \rightarrow +\infty]{} +\infty
        \end{align*}
        \item If $\mu \in \Ran\left(Q^T\right)$, then 
        \begin{equation*}
            \bar{x} = \left(Q^TQ\right)^{\dag}\left( Q^Tc + \mu\right)
        \end{equation*}
        and $f^*$ could be simplified to get: 
        \begin{equation*}
            f^*(\mu) = \frac{1}{2}\left\langle \mu,  \left(Q^TQ\right)^{\dag} \mu \right\rangle + \left\langle \mu,  \left(Q^TQ\right)^{\dag}  Q^Tc \right\rangle -\frac{1}{2} \left\|Q \left(Q^TQ\right)^{\dag} Q^Tc - c\right\|^2 
        \end{equation*}
    \end{itemize}
    Thus, 
    \begin{equation} \label{eqn: LS -- f*}
        f^*(\mu) = \begin{multlined}[t] \frac{1}{2}\left\langle \mu,  \left(Q^TQ\right)^{\dag} \mu \right\rangle + \left\langle \mu,  \left(Q^TQ\right)^{\dag}  Q^Tc \right\rangle -\frac{1}{2} \left\|Q \left(Q^TQ\right)^{\dag} Q^Tc - c\right\|^2  \\+ \imath_{\Ran\left(Q^T\right)}(\mu) \end{multlined}\\
    \end{equation}
    \item Rewriting $f^*(\mu) = f_1^*(\mu)$ and $f_2^* = 0$ satisfies the set of assumptions, $\mathcal{E}$, of Theorem \ref{thm:PDG-SDG -- manifold}. Thus, we define: 
    \begin{equation*}
    \begin{split}
        g(\lambda) &= f^*(\mu_0 + \phi^{-1}(\lambda)) \\ 
        &= \frac{1}{2}\left\langle \mu,  \left(Q^TQ\right)^{\dag} \mu \right\rangle + \left\langle \mu,  \left(Q^TQ\right)^{\dag}  Q^Tc \right\rangle -\frac{1}{2} \left\|Q \left(Q^TQ\right)^{\dag} Q^Tc - c\right\|^2  \\ 
        \lambda &= \phi(\mu) \hspace{2cm} \forall \mu \in \dom f^*
    \end{split}
    \end{equation*}
    Hence, $g$ is \textbf{differentiable} with \textbf{Lipschitz constant} $$L_g = \left\|\left(Q^TQ\right)^{\dag}\right\|_{op} = \lambda_{\max}\left(\left(Q^TQ\right)^{\dag}\right)$$
    \item \textbf{Projection} onto $\dom f^* = \Ran\left(Q^T\right)$
    \begin{align*}
        \proj_{ \Ran\left(Q^T\right)} (\mu) &= \arg\min_{u \in \Ran\left(Q^T\right)} \|u - \mu\|^2 \\
        & \stackrel{u = Q^Tv}{=} Q^T \arg\min_{v \in \R^m} \left\|Q^Tv - \mu \right\|^2 \\
        & = Q^T \left(QQ^T\right)^{\dag} Q \mu \\
    \end{align*}
\end{enumerate}
We will show in the sequel that finding the values of the \acrshort{qebsg} and \acrshort{msr} constants $\eta$ and $\gamma$, respectively, is tractable for this problem. 
\begin{enumerate}
  \setcounter{enumi}{5}
  \item \textbf{Metric sub-regularity constant}, $\gamma$. 
\begin{lemma} \label{lem:LC-LS--MSR}
    For the \ref{eqn: LS -- prob} problem, let $z^{\star} = \proj_{\Z^{\star}}(z)$. Then, for any $z \in \Z$ the Lagrangian's sub-differential satisfies:
    \begin{align}
        \left\|\partial_x \Lag(x, y)\right\| + \left\|\nabla_y \Lag(x, y)\right\| &\geq \left|\lambda(\mathcal{M})\right|_{\min} \dist\left(z, \Z^{\star}\right) \\
        \mathcal{M} &= \begin{bmatrix} Q^TQ & A^T \\ A & 0\end{bmatrix}
    \end{align}
    where $\left|\lambda\left(\mathcal{M}\right)\right|_{\min}$ is the smallest absolute value of the non-zero eigenvalues of $\mathcal{M}$.
    
\textit{Proof.}  See Appendix \ref{appendix:MSR et QEB, LC-LS}
\end{lemma}
Thanks to this Lemma, one can take $\gamma = \left|\lambda\left(\mathcal{M}\right)\right|_{\min}$
\item \textbf{Quadratic error bound of the smoothed gap constant}, $\eta$. 
\begin{lemma} \label{lem:LC-LS--QEB}
    For the \ref{eqn: LS -- prob} problem, the self-centered \acrlong{sdg} could be reformulated into a quadratic form. That is, for any $z \in \Z$:
    \begin{equation}
        \G(z) = z^T\mathcal{H} z + \langle z, v \rangle + cst 
    \end{equation}
    where, 
    \begin{align}
       v_{(n + m)} = \begin{bmatrix} v_x \\ v_y\end{bmatrix} && \mathcal{H}_{{(n+m)\times(n+m)}} = \begin{bmatrix}
            M_{xx}& \frac{1}{2} M_{xy} \\ \frac{1}{2} M_{xy}^T & M_{yy}
        \end{bmatrix}
    \end{align}    
\begin{table}[!ht]
    \centering
    \begin{tabular}{||l|c||} 
    \hline
    \multicolumn{1}{||c}{\textbf{Vector - Matrix}} & \multicolumn{1}{|c||}{\textbf{Size}}  \\
    \hline \hline 
    $v_x = \frac{\beta_x}{\tau}B^{-1}Q^Tc -Q^Tc$ &  $n \times 1$ \\
    $v_y = -AB^{-1}Q^Tc$ & $m \times 1$ \\
    $B = Q^TQ + \frac{\beta_x}{\tau} \Id_{n}$ & $n \times n$  \\
    $M_{xx} = \frac{1}{2}Q^TQ + \frac{\sigma}{2\beta_y} A^TA + \frac{\beta_x^2}{2\tau^2} B^{-1} - \frac{\beta_x}{2\tau} \Id_{n}$ & $n \times n$ \\
     $M_{xy} = A^T - \frac{\beta_x}{\tau} B^{-1}A^T$ & $n \times m$\\
     $M_{yy} = \frac{1}{2} AB^{-1}A^T$ & $m \times m$ \\
      \hline 
    \end{tabular}
    \label{tab:LS-eta}
\end{table}

\textit{Proof.}  See Appendix \ref{appendix:MSR et QEB, LC-LS}
\end{lemma}
The advantage of the lemma is that, for \ref{eqn: LS -- prob}, the self-centered smoothed duality gap is convex and $\lambda_{\min}(\mathcal{H})-$metrically sub-regular, with $\lambda_{\min}(\mathcal{H})$ being the smallest positive eigenvalue of the positive semi-definite matrix $\mathcal{H}$. Hence, by Proposition \ref{lem:MSR-->QEB}, we conclude that $\G(z)$ has a $\eta = \frac{\lambda_{\min}\left(\mathcal{H}\right)}{2}-$\acrshort{qeb}. 
\end{enumerate}

\subsubsection{Problem instances}
In this part, we explore various instances of the  \ref{eqn: LS -- prob} problem. We start by examining a straightforward one-dimensional case, for which we can determine the precise minimizer. Subsequently, we extend our analysis to a higher dimensional problem with independently and identically distributed (i.i.d.) Gaussian matrices $Q$ and $A$. Finally, we delve into a more complex scenario where we generate $Q$ and $A$ with non-trivial covariance matrices.
\begin{itemize}
    \item \textbf{One-dimensional} problem
    
\begin{equation} \tag{1D} \label{eqn: 1D -- prob}
    \begin{split}
        \min_{x \in \R} ~ \frac{1}{2} &\left(\frac{1}{9}x - 2\right)^2 \\ 
        &~ 9x = 7
    \end{split}
\end{equation}
The \textit{primal-feasibility} condition within the KKT conditions implies that $x^{\star} = \frac{7}{9}$. Consequently, in this particular problem, we can gain a more precise assessment of the quality of our approximations for the \acrlong{og}.
\item \textbf{I.I.D. Gaussian matrices:} We examine the \ref{eqn: LS -- prob} problem with dimensions set to $n = 20$ and $m = 10$. In this scenario, we generate independently and identically distributed (i.i.d.) Gaussian matrices $Q$ and $A$. That is:
\begin{equation} \tag{IIDG} \label{eqn:IIDG}
\begin{split}    
Q \in \R^{m \times n} ~\text{s.t.} ~\forall (i, j) \in \llbracket 1, m \rrbracket \times \llbracket 1, n \rrbracket, ~ Q_{ij} \sim \mathcal{N}(0, 1) ~ i.i.d. \\ 
    A \in \R^{m \times n} ~\text{s.t.} ~\forall (i, j) \in \llbracket 1, m \rrbracket \times \llbracket 1, n \rrbracket, ~ A_{ij} \sim \mathcal{N}(0, 1) ~ i.i.d. 
\end{split}
\end{equation}
\item \textbf{Non-trivial covariance:} We investigate the \ref{eqn: LS -- prob} problem with dimensions set to $n = 20$ and $m = 10$. In this instance, we generate matrices $Q$ and $A$ with non-trivial covariance matrices, defined as follows:
\begin{align} \tag{NTC} \label{eqn:NTC}
        A = \Sigma_a X_a && Q = \Sigma_q X_q 
\end{align}
where
\begin{itemize}
    \item The matrices $X_a, X_q \in \R^{m \times n}$ are Gaussian matrices, following the pattern established in the previous case.
    \item The matrices $\Sigma_a, \Sigma_q \in \R^{m \times m}$ serve as covariance matrices, generated using the Python built-in function \texttt{\href{https://docs.scipy.org/doc/scipy/reference/generated/scipy.linalg.toeplitz.html}{toeplitz}}.
\end{itemize}
\end{itemize}

\begin{figure}[htbp]
\centering
\subfloat[One-dimensional, $L \approx 0.012$]{\includegraphics[width=0.34\linewidth]{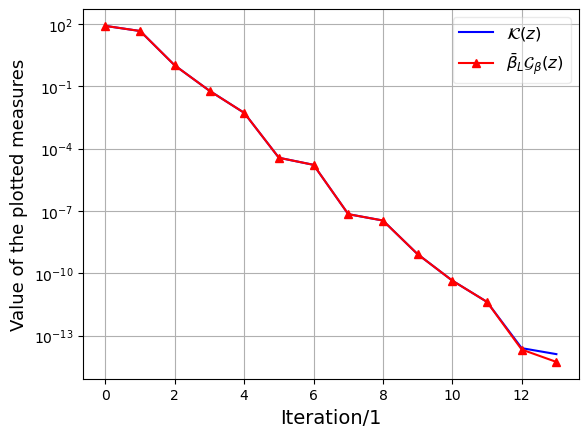} \label{fig:1D}}
\subfloat[I.I.D Gaussian $L \approx 49$]{ \includegraphics[width=0.34\linewidth]{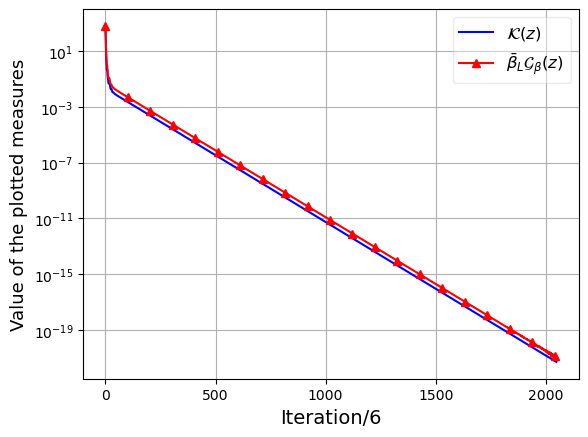} \label{fig:IID}}
\subfloat[Non-trivial covariance]{\includegraphics[width=0.34\linewidth]{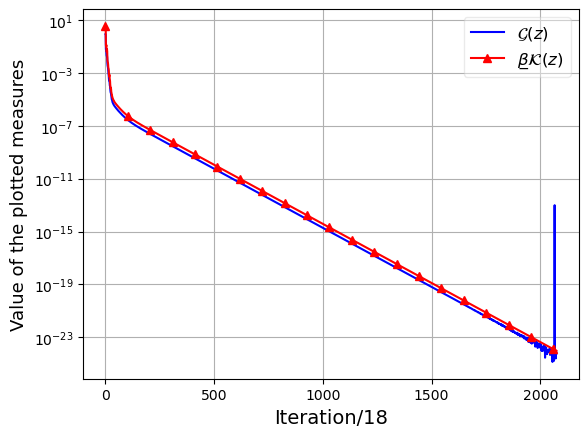} \label{fig:covariance}}
\caption{Numerical illustration of Theorems (\ref{thm:SDG-KKT} ~\&~ \ref{thm:KKT-SDG}) on several \ref{eqn: LS -- prob} instances. $L$ denotes the Lipschitz constant of the gradient}
\label{fig:LCLS}
\end{figure}

Figure \ref{fig:LCLS} validates our comparability bound between the \acrlong{sdg} and the \acrlong{kkt}, presented in subsection \ref{subseq:SDG-KKT}. It illustrates their efficiency and tightness across various instances. In addition, we note that the smoother the function (i.e. with smaller $L$) the tighter the bound as in (\ref{fig:1D} and \ref{fig:IID}).

%% file: DO.tex
\subsection{Distributed Optimization}
We examine the unconstrained optimization problem presented as follows \cite{DO}:
\begin{equation} \tag{UC-DO} \label{eqn: DO -- uncons}
    \min_{x \in \R^n} ~ \frac{1}{2} \sum_{i = 1}^M  \left\| Q_ix - c_i\right\|^2 
\end{equation}

In this formulation, for each $i \in \llbracket 1, M \rrbracket $, the matrix $Q_i \in \R^{m \times n}$ and the vector $c_i \in \R^m$ are real data sourced from the \texttt{\href{https://www.csie.ntu.edu.tw/~cjlin/libsvmtools/datasets/regression.html}{bodyfat}} dataset. This data comprises 252 data-points, each characterized by 14 features. In our case, we divide the dataset into $M = 3, n = 14$, and $m = 84$.

To address security concerns, we assume that the data $Q_i$ and $c_i$ are distributed across $M$ distinct computers. These computers collaborate to solve the problem in a manner where each computer $i$ utilizes the data $Q_i$ and $c_i$ to derive a partial solution, which is then transmitted to the subsequent computer $i+1$. Repeatedly until finding the optimum. To facilitate this collaborative approach, we need to reformulate the unconstrained problem (\ref{eqn: DO -- uncons}) into a constrained optimization problem. This involves introducing additional variables and constraints to model the communication process among the various computers.

\begin{equation*} 
\begin{split}
     \min_{x_1, \dots, x_M \in \R^{n}} ~ \frac{1}{2}& \sum_{i = 1}^M  \left\| Q_ix_i - c_i\right\|^2\\
       \text{s.t.}~~  & x_i = x_{i+1}, \hspace{1.5cm} \forall i \in \llbracket 1, M-1 \rrbracket
\end{split}
\end{equation*}

Furthermore, 

\begin{itemize}
    \item The constraints can be expressed in matrix form as $AX = 0$, where the matrix $A$ is constructed as follows:
\begin{align*}
    A_{(M-1)n \times Mn} := \begin{bmatrix}
        \Id_n & -\Id_n & \mathbf{0}_n & \dots & \mathbf{0}_n \\ 
        \mathbf{0}_n  & \Id_n & -\Id_n  & \dots & \mathbf{0}_n \\
        \vdots & \dots & \ddots & \ddots & \dots\\
        \mathbf{0}_n & \dots &  \mathbf{0}_n & \Id_n & -\Id_n 
    \end{bmatrix} \quad \text{and} \quad X_{Mn} := \begin{bmatrix}
        x_1 \\ \vdots \\ \vdots \\ x_M
    \end{bmatrix}
\end{align*}

\item  The objective function $\displaystyle \left(f(X) = \frac{1}{2}\sum_{i = 1}^M \left\|Q_i x_i - c_i \right\|^2\right)$ can be rewritten as:

\begin{equation*}
    f(X) = \frac{1}{2} \left\|\Qb X - \cb \right\|^2
\end{equation*}

Here, $\Qb$ is formed by arranging the matrices $Q_1, Q_2, \ldots, Q_M$ along the diagonal, and $\cb$ is a stacked vector of $c_1, c_2, \ldots, c_M$. That is:
    \begin{align*}
    \Qb_{Mm \times Mn} := \begin{bmatrix}
        Q_1 & \0_{m \times n} & \dots & \0_{m \times n} \\ \0_{m \times n} & Q_2 & \dots & \0_{m \times n} \\ \vdots & \dots & \ddots & \dots \\ \0_{m \times n} & \dots & \0_{m \times n} & Q_M
    \end{bmatrix} && \cb_{Mm} := \begin{bmatrix}
        c_1 \\ \vdots \\ c_M
    \end{bmatrix}
\end{align*}
\end{itemize}
Therefore, the unconstrained problem can be reformulated as an instance of the \ref{eqn: LS -- prob} problem as follows: 

\begin{equation} \tag{DO} \label{eqn: DO -- prob}
    \min_{X \in \R^{Mn}} ~ \frac{1}{2} \left\| \Qb X  - \cb \right\|^2 \hspace{1cm} \text{subject to} \hspace{1cm} AX = 0
\end{equation}

It is important to note that all the formulations of the problem are equivalent, specifically (\ref{eqn: DO -- uncons} $\equiv$ \ref{eqn: DO -- prob}). Consequently, since (\ref{eqn: DO -- uncons}) represents an unconstrained smooth problem, the minimizer $x^{\star}$ can be easily found:
\begin{align*}
    x^{\star} &= \arg\min_{x \in \R^n} \frac{1}{2} \sum_{i = 1}^M \left\|Q_ix - c_i\right\| ^2 \\
    &\iff 0 = \nabla \left(\frac{1}{2} \sum_{i = 1}^M \left\|Q_ix - c_i\right\|^2\right)(x^{\star}) \\ 
    &\iff \sum_{i = 1}^M Q^T_i\left(Q_ix^{\star} - c_i\right) = 0 \\
    &\iff \left(\sum_{i = 1}^M Q^T_iQ_i\right)x^{\star} = \sum_{i = 1}^M Q^T_i c_i \stepcounter{equation}\tag{\theequation}\label{eqn: DO -- x_star}
\end{align*}

\begin{figure}[htbp]
    \centering
    \includegraphics[width=0.8\linewidth]{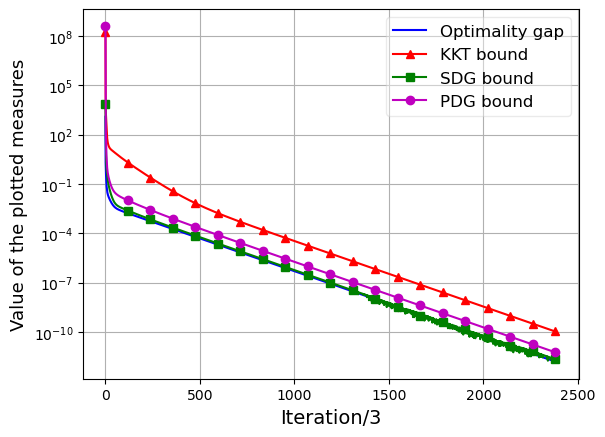}
    \caption{Numerical illustration of Theorems (\ref{thm:OG-KKT}, \ref{thm:OG-SDG}, and \ref{thm:OG-PDG}), \acrlong{og} vs. its bounds. Target: $\varepsilon = 10^{-10}$}
    \label{fig:DO-OG}
\end{figure}
Figure \ref{fig:DO-OG} validates our computable approximations for the \acrlong{og}, as presented in subsection \ref{subsec:final bounds}. Notably, in this experiment, we were able to precisely plot the \acrlong{og} due to our analytical knowledge of the minimizer (eqn:\ref{eqn: DO -- x_star}). Furthermore, the figure highlights the superior efficiency and tightness of the \acrshort{sdg} bound compared to the others, as it closely aligns with the \acrlong{og} curve.

%% file: QP.tex
\subsection{Quadratic Programming}
As we have seen, our theoretical findings have primarily focused on optimization problems under equality affine constraints. However, in this experiment, we aim to extend the applicability of our theoretical insights to address optimization problems that incorporate inequality constraints. Therefore, in this subsection, we delve into the same \ref{eqn: LS -- prob} problem as discussed earlier, but now incorporating the additional requirement of non-negativity constraints on the weights \cite{QP}.
\begin{equation} \tag{QP} \label{QP: initial prob}
\begin{split}
    \min_{x \in \R^n}  \frac{1}{2} & \|Qx - c\|^2 \\ 
    \text{s.t.} ~~ & Ax = b \\ 
     & \hspace{0.3cm} x \geq 0
\end{split}
\end{equation}
Applying our theoretical results to address this problem necessitates a reformulation to align it with our framework. A feasible approach involves encapsulating the non-negativity constraint by introducing an indicator function within the objective function. Thus, the reformulation is as follows:
\begin{equation*} 
\begin{split}
    \min_{x \in \R^n} \frac{1}{2} \|Qx - c\|^2 & + \imath_{\R_+^n}(x)\\ 
    \text{s.t.} ~~  Ax &= b
\end{split}
\end{equation*}
Despite the current appropriate formulation, determining the proximal operator of the objective function, essential for computing the \acrlong{sdg}, has become non-trivial. However, introducing an additional variable, $\Tilde{x}$, and incorporating an extra constraint, $x = \Tilde{x}$, would render the computation more manageable. That is:
\begin{equation*} 
\begin{split}
    \min_{x, \Tilde{x} \in \R^n}  \frac{1}{2}  \|Qx - c\|^2 &+ \imath_{\R_+^n}(\Tilde{x})\\ 
    \text{s.t.} \hspace{1.65cm}  Ax &= b \\ 
     x &= \Tilde{x}
\end{split}
\end{equation*}
Thanks to this trick, our objective function is now separable, facilitating the computation of its proximal operator. By combining the two constraints into a single matrix form, our problem can be expressed as follows:
\begin{equation}  \tag{PQP} \label{QP: prob}
\begin{split}
    \min_{X \in \R^{2n}}  F(X) = \frac{1}{2}  \|Qx - c\|^2 & + \imath_{\R_+^n}(\Tilde{x})\\ 
    \text{s.t.} \hspace{2.5cm}  \Tilde{A}X &= B 
\end{split}
\end{equation}
where 
\begin{align}
    X := \begin{bmatrix}
        x \\ \Tilde{x}
    \end{bmatrix} && \Tilde{A} := \begin{bmatrix}
        A & 0 \\ \Id & - \Id 
    \end{bmatrix} && B := \begin{bmatrix}
        b \\ 0
    \end{bmatrix}
\end{align}
\subsubsection{Problem Analysis}
\begin{enumerate}
    \item \textbf{The sub-differential of the objective function}
    \begin{align*}
        \partial F(x, \Tilde{x}) &= \partial \left( {\color{red} \frac{1}{2} \left\|Qx - c\right\|^2 } + {\color{blue}\imath_{\R^n_+}(\Tilde{x}) } \right)\\ 
        &\stackrel{(\ref{eqn:separable, sub-diff})}{=} {\color{red} \nabla \left(\frac{1}{2} \left\|Qx - c\right\|^2 \right)} \times \partial{\color{blue} \left(\imath_{\R^n_+}(\Tilde{x}) \right)} \\
        &= {\color{red}Q^T\left(Qx - c\right)} \times \partial{\color{blue} \left(\sum_{i = 1}^n \imath_{\R_+}(\Tilde{x}_i) \right)} \\ 
        &\stackrel{(\ref{eqn:separable, sub-diff})}{=} Q^T\left(Qx - c\right) \times {\color{blue}\prod_{i = 1}^n \partial \imath_{\R_+}(\Tilde{x}_i)} 
        \end{align*}
    Thus,
    \begin{equation}  \label{QP: sub-diff}
        \begin{split}
        \partial F(X)  = & \overbrace{\underbrace{Q^T\left(Qx - c\right)}_{n-\text{tuple}} \times \underbrace{\prod_{i = 1}^n \nor_{\R_+}(\Tilde{x}_i)}_{n-\text{tuple}}}^{2n-\text{tuple}}  \\ 
        \text{where} ~  & ~ \nor_{\R_+}(\nu) = \begin{cases}
            \emptyset, & \nu < 0 \\ \R_-, & \nu = 0 \\ \{0\}, & \nu > 0
        \end{cases}
        \end{split}
    \end{equation}
    \item \textbf{The stationarity part of the KKT error:}
    
    Let $J = \Tilde{A}^TY \in \R^{2n}$ and define: 
    \begin{align*}
        \underaccent{\bar}{J} := \{J_1, \dots, J_n\} && \Bar{J} := \{J_{n+1}, \dots, J_{2n}\}
    \end{align*}
    Then, 
    \begin{align*}
    \left\| \partial F(x, \Tilde{x}) + \Tilde{A}^TY \right\|_0^2 &=  \left\| \partial F(X) + J \right\|_0^2 \\
    &= \min\left(\left\{\|q\| ~|~ q \in \partial F(X) + J \right\}\right)^2\\
    &= \min \left(\left\{\left\|\left(\left(Q^T\left(Qx - c\right)\right) \times \prod_{i = 1}^n \nor_{\R_+}(\Tilde{x}_i)\right) + J \right\| \right\} \right)^2 \\
    &= \min \left\{\left\|\left(\left(Q^T\left(Qx - c\right)\right) \times \prod_{i = 1}^n \nor_{\R_+}(\Tilde{x}_i)\right) + J \right\|^2 \right\} \\
    &= \min \left\{{\color{red} \left\|Q^T\left(Qx - c\right) + \underaccent{\bar}{J} \right\|^2} + \left\| \prod_{i = 1}^n \nor_{\R_+}(\Tilde{x}_i) + \Bar{J} \right\|^2 \right\} \\
    &= {\color{red} \left\|Q^T\left(Qx - c\right) + \underaccent{\bar}{J} \right\|^2 } + \min \left\{\left\| \prod_{i = 1}^n \nor_{\R_+}(\Tilde{x}_i) + \Bar{J} \right\|^2 \right\}
\end{align*}
Now, let's analyze the second term: 
\begin{align*}
     \min \left\{\left\| \prod_{i = 1}^n \nor_{\R_+}(\Tilde{x}_i) + \Bar{J} \right\|^2 \right\} &=  \min \left\{\sum_{i = 1}^n \left(\nor_{\R_+}(\Tilde{x}_i) + \Bar{J}_i \right)^2 \right\} \\ 
     &= \sum_{i = 1}^n \min \left\{\left(\nor_{\R_+}(\Tilde{x}_i) + \Bar{J}_i \right)^2 \right\} \\ 
     &= \sum_{i = 1}^n \min \begin{cases}
         +\infty, & \Tilde{x}_i < 0 \\ \left(\R_- + \Bar{J}_i\right)^2, & \Tilde{x}_i = 0 \\ \Bar{J}_i^2, & \Tilde{x}_i > 0
     \end{cases} \\ 
     &= \sum_{i = 1}^n \begin{cases} 
         +\infty, & \Tilde{x}_i < 0 \\ \displaystyle \min_{u \in \R_-} \left(u + \Bar{J}_i\right)^2, & \Tilde{x}_i = 0 \\ \Bar{J}_i^2, & \Tilde{x}_i > 0
     \end{cases} \\ 
     &= \sum_{i = 1}^n \begin{cases}
         +\infty, & \Tilde{x}_i < 0 \\ \left(\proj_{\R_-}\left(-\Bar{J}_i\right)  + \Bar{J}_i\right)^2, & \Tilde{x}_i = 0 \\ \Bar{J}_i^2, & \Tilde{x}_i > 0
     \end{cases} \\ 
     &= \sum_{i = 1}^n \begin{cases}
         +\infty, & \Tilde{x}_i < 0 \\ \min\left(0, \Bar{J}_i\right)^2, & \Tilde{x}_i = 0 \\ \Bar{J}_i^2, & \Tilde{x}_i > 0
     \end{cases} \\ 
\end{align*}
    \item \textbf{The proximal operator is:} 
    \begin{align*}
        \prox_{sf}(x, \Tilde{x}) &= \arg\min_{u, \Tilde{u} \in \R^n} \underbrace{{\color{red} \frac{1}{2} \left\|Qu - c\right\|^2 + \frac{1}{2s} \|u - x\|^2}}_{h(u)} + {\color{blue} \imath_{\R^n_+}(\Tilde{u}) + \frac{1}{2s}\|\Tilde{u} - \Tilde{x}\|^2} \\ 
        &\stackrel{(\ref{eqn:separable, prox})}{=} \left({\color{red} \arg\min_{u \in \R^n} h(u)}, {\color{blue} \arg\min_{\Tilde{u} \in \R^n_+} \frac{1}{2s}\|\Tilde{u} - \Tilde{x}\|^2}\right) \\
        &\stackrel{(\ref{eqn: LS -- prox})}{=} \left({\color{red} \left(Q^TQ + \frac{1}{s} \Id \right)^{-1}\left(Q^Tc + \frac{1}{s}x\right)} , {\color{blue} \proj_{\R^n_+}(\Tilde{x})} \right) \\
        &= \left(\left(Q^TQ + \frac{1}{s} \Id \right)^{-1}\left(Q^Tc + \frac{1}{s}x\right), \max(\Tilde{x}, 0)) \right)
    \end{align*}
    \item \textbf{The Fenchel-Conjugate is:} 
    \begin{align*}
        f^*(\mu, \Tilde{\mu}) &= \sup_{X \in \R^{2n}} {\color{red} \langle \mu, x \rangle } {\color{blue} + \langle \Tilde{\mu}, \Tilde{x} \rangle} {\color{red} - \frac{1}{2} \left\|Qx - c \right\|^2} {\color{blue} - \imath_{\R^n_+}(\Tilde{x})} \\ 
        & \stackrel{(\ref{eqn:separable, sup})}{=}  \sup_{x \in \R^{n}} {\color{red} \left( \langle \mu, x \rangle - \frac{1}{2} \left\|Qx - c \right\|^2 \right)} + \sup_{\Tilde{x} \in \R^n} {\color{blue} \left(\langle \Tilde{\mu}, \Tilde{x} \rangle  - \imath_{\R^n_+}(\Tilde{x}) \right)} \\
        &\stackrel{(\ref{eqn: LS -- f*})}{=} \begin{multlined}[t] {\color{red} \frac{1}{2}\left\langle \mu,  \left(Q^TQ\right)^{\dag} \mu \right\rangle + \left\langle \mu,  \left(Q^TQ\right)^{\dag}  Q^Tc \right\rangle -\frac{1}{2} \left\|Q \left(Q^TQ\right)^{\dag} Q^Tc - c\right\|^2 } \\ + {\color{red} \imath_{\Ran\left(Q^T\right)}(\mu) } + {\color{blue} \imath_{\R^n_-} (\Tilde{\mu})} \end{multlined}
    \end{align*}
    \item Rewriting $f^*(\mu) = f_1^*(\Tilde{\mu}) + f_2^*(\mu)$ such that: 
        \begin{align*}
            f_1^*(\Tilde{\mu}) &=  \imath_{\R_-^n}(\Tilde{\mu}) \\
            f_2^*(\mu) &= \begin{multlined}[t] \frac{1}{2}\left\langle \mu,  \left(Q^TQ\right)^{\dag} \mu \right\rangle + \left\langle \mu,  \left(Q^TQ\right)^{\dag}  Q^Tc \right\rangle -\frac{1}{2} \left\|Q \left(Q^TQ\right)^{\dag} Q^Tc - c\right\|^2 \\ + \imath_{\Ran\left(Q^T\right)}(\mu) \end{multlined}
        \end{align*}
        satisfies the set of assumptions, $\mathcal{E},$ of Theorem \ref{thm:PDG-SDG -- manifold}. Thus, 
        \begin{itemize}
            \item $f_1^*$ is $L_{f_1^*} = 0-$Lipschitz on its domain, $\R_-^n$. 
            \item $\nabla f_2^*$ is $L_g =  \lambda_{\max}\left(\left(Q^TQ\right)^{\dag}\right)-$Lipschitz using the earlier analysis we did for \ref{eqn: LS -- prob}. 
        \end{itemize}
    \item \textbf{Projection onto the $\dom f^* = \Ran\left(Q^T\right) \times \R^n_-$} 
    \begin{equation*} \label{eqn: QP -- proj}
            \proj_{\dom f^*} (\mu) = \left({\color{red} Q^T\left(QQ^T\right)^{\dag} Q \mu},~ {\color{blue} \min\left(\Tilde{\mu}, 0\right)} \right)
    \end{equation*}
    \item While it is theoretically possible to determine the constants for the \acrlong{msrsdl} and the \acrlong{qebsg}, the process is highly complex. Due to the complexity involved, we opt for a practical approach by assigning arbitrary small values to these constants. In this experiment, we set $\gamma = \eta = 10^{-8}$ as convenient and satisfactory choices, allowing us to proceed without extensive efforts in identifying precise values.
 \end{enumerate}

 \begin{figure}[htbp]
\centering
\subfloat[\acrshort{pdg} bound, (eqn:\ref{eqn:SDG-PDG})]{\includegraphics[width=0.482\linewidth]{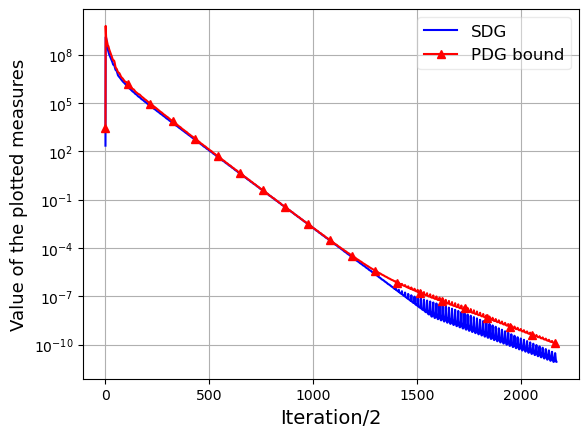} \label{fig:G<D}}
\subfloat[\acrshort{sdg} bound, (eqn:\ref{eqn:PDG-SDG})]{\includegraphics[width=0.485\linewidth]{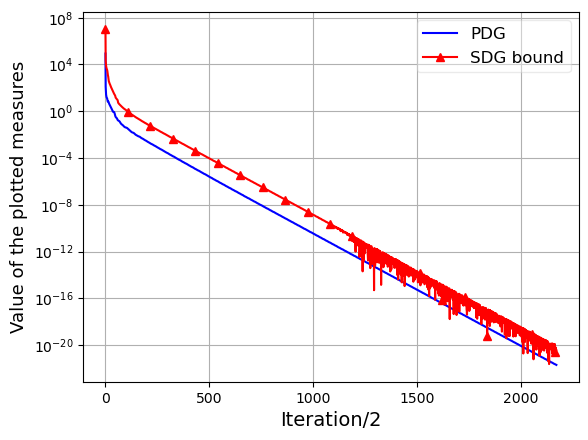} \label{fig:D<G}}
\caption{Numerical illustration of Theorems (\ref{thm:SDG-PDG} and \ref{thm:PDG-SDG -- manifold}) with $n = 20$ and $m = 10$}
\end{figure}

%% file: BP.tex
\subsection{Basis Pursuit}
Our final experiment involves a non-smooth convex minimization problem known as the \textit{Basis Pursuit} problem \cite{Chen1994BasisP,BP}. It aims to minimize the $\ell_1$ norm while satisfying a system of linear equations and is mathematically formulated as follows:

\begin{equation} \tag{BP} \label{eqn: BP -- prob}
        \min_{x \in \R^n} ~ \|x\|_1 \hspace{1cm} \text{subject to} \hspace{1cm}  Ax = b
\end{equation}

Here, we set $n = 20$ and $m = 10$, and generate an i.i.d. Gaussian matrix $A \in \R^{m \times n}$ and vector $b \in \R^m$.

The primary objective of this experiment is to highlight the superior stability of the \acrlong{sdg} compared to the \acrlong{kkt}, as will be demonstrated subsequently.
\subsubsection{Problem Analysis}
\begin{enumerate} 
    \item \textbf{The sub-differential of the objective function}
    \begin{equation*} \label{eqn: BP -- sub-diff}
        \partial f(x) = \partial (\|x\|_1) = \partial \left(\sum_{i = 1}^n |x_i|\right) \stackrel{(\ref{eqn:separable, sub-diff})}{=} \prod_{i = 1}^n \partial |x_i|
    \end{equation*}
    where \begin{equation*} \label{eqn: BP -- abs sub-diff}
        \partial |\nu| = \begin{cases}
        \{-1\}, & \nu < 0 \\ [-1, 1], & \nu = 0 \\ \{1\}, & \nu > 0 
    \end{cases}        
    \end{equation*}
    \item \textbf{The stationarity part of the KKT error:} Applying the same approach as we have previously done for (\ref{QP: prob}), we obtain:
    \begin{align*}
    \left\| \partial f(x) + A^Ty \right\|_0^2 = \sum_{i = 1}^n \begin{cases}
         \left(\mathrm{sgn}(x_i) + (A^Ty)_i\right)^2, & x_i \neq 0 \\
         \left(\min\left(\max\left(-\left(A^Ty\right)_i, -1\right), 1\right) + (A^Ty)_i\right)^2, & x_i = 0 
    \end{cases} 
\end{align*}
    \item \textbf{The proximal operator is:} (\textit{Example 2.16, \cite{soft-thresholding}}) 
    \begin{equation*} \label{eqn: BP -- prox}
        \prox_{sf}(x) = \left( \left[|x_i| - s\right]_+ \mathrm{sgn}(x_i)\right)_{1 \leq i \leq n}
    \end{equation*}
    \item \textbf{The Fenchel-Conjugate is:} (\textit{Example 3.26, \cite{convex_opt}}) 
     \begin{equation*}  \label{eqn: BP -- conj}
            f^*(\mu) = \imath_{\mathcal{B}_{\infty}(0, 1)}(\mu)
    \end{equation*}
    \item \textbf{The conjugate function is $L_{f^*} = 0-$Lipschitz on its domain}, $\mathcal{B}_{\infty}(0, 1)$. So, we can take advantage of using Proposition \ref{prop:PDG-SDG} in Appendix \ref{appendix:PDG in terms of SDG} that has a slightly tighter bound than Theorem \ref{thm:PDG-SDG -- manifold}. 
    \item \textbf{Projection onto} $\dom f^* =\mathcal{B}_{\infty}(0, 1)$.
    \begin{equation*} \label{eqn: BP -- proj}
            \proj_{\mathcal{B}_{\infty}(0, 1)}(\mu) = \left(\begin{cases} \sgn\left(\mu_i\right), & |\mu_i| > 1 \\ 
            \mu_i, & |\mu_i| \leq 1 
            \end{cases}\right)_{1 \leq i \leq n}
    \end{equation*}
\end{enumerate}
\subsubsection{Two versions of PDHG}
In this sub-subsection, we consider two versions of the PDHG algorithm designed to solve two formulations of the saddle point problem of (\ref{eqn: BP -- prob}):
\begin{align*}
    &\max_y \min_x \|x\|_1 + \left\langle Ax - b, y \right\rangle \stepcounter{equation}\tag{\theequation}\label{eqn: BP -- saddle pt. prob. 1} \\ 
    \equiv & \max_x \min_y \|y\|_1 + \left \langle Ay - b, x \right\rangle \\ 
    \equiv & \max_x \min_y - \left(-\|y\|_1 - \left\langle Ay - b, x \right\rangle \right) \\ 
    \equiv & - \min_x \max_y \left(-\|y\|_1 - \left\langle Ay - b, x \right\rangle \right) \stepcounter{equation}\tag{\theequation}\label{eqn: BP -- saddle pt. prob. 2}
\end{align*}
Here, (\ref{eqn: BP -- saddle pt. prob. 1}) and (\ref{eqn: BP -- saddle pt. prob. 2}) can be interpreted as swapping the roles of the primal and dual variables ($x$ and $y$, respectively). Applying the PDHG algorithm to solve (\ref{eqn: BP -- saddle pt. prob. 1}) and (\ref{eqn: BP -- saddle pt. prob. 2}) leads us to the following two versions: 
    \begin{table}[!ht]
    \centering
    \begin{tabular}{||l  l||} 
    \hline 
    \multicolumn{1}{||c}{\textbf{Version 1}} &  \multicolumn{1}{c||}{\textbf{Interpretation}} \\
    \hline \hline
    $\bar{x}_{k+1} = \prox_{\tau f}\left(x_k - \tau A^Ty_k\right)$ & Primal Forward-Backward step\\
    $\bar{y}_{k+1} = y_k + \sigma \left(A\bar{x}_{k+1} - b\right)$ & Dual Forward-Backward step\\
    $x_{k+1} = \bar{x}_{k+1} - \tau A^T\left(\bar{y}_{k+1} - y_k\right)$  & Primal Extrapolation step  \\
    $y_{k+1} = \bar{y}_{k+1}$ &  Dual Extrapolation step \\
      \hline 
    \end{tabular}
    \caption{Interpretation of each step of PDHG for solving (\ref{eqn: BP -- saddle pt. prob. 1})}
    \label{alg:PDHG v1}
\end{table}
    \begin{table}[!ht]
    \centering
    \begin{tabular}{||l  l||} 
    \hline
    \multicolumn{1}{||c}{\textbf{Version 2}} &  \multicolumn{1}{c||}{\textbf{Interpretation}}  \\
    \hline \hline  
    $\bar{y}_{k+1} = y_k + \sigma \left(Ax_{k} - b\right)$ &Dual Forward-Backward step\\
    $\bar{x}_{k+1} = \prox_{\tau f}\left(x_k - \tau A^T\bar{y}_{k+1}\right)$ &Primal Forward-Backward step\\
    $y_{k+1} = \bar{y}_{k+1} + \sigma A\left(\bar{x}_{k+1} - x_k\right)$  & Dual Extrapolation step \\
    $x_{k+1} = \bar{x}_{k+1}$ &  Primal Extrapolation step \\
      \hline 
    \end{tabular}
    \caption{Interpretation of each step of PDHG for solving (\ref{eqn: BP -- saddle pt. prob. 2})}
    \label{alg:PDHG v2}
\end{table}

The key advantage of implementing the two versions of PDHG for solving the saddle point problems (\ref{eqn: BP -- saddle pt. prob. 1}) and (\ref{eqn: BP -- saddle pt. prob. 2}) is to demonstrate the superior stability of the \acrlong{sdg} over the \acrlong{kkt}. To understand this, let's revisit the stationarity part of the \acrshort{kkt}:
\begin{equation*}
    \left\|\partial f(x) + A^Ty\right\|^2 = \sum_{i = 1}^n \begin{cases}
         \left(\mathrm{sgn}(x_i) + (A^Ty)_i\right)^2, & x_i \neq 0 \\
         \left(\min\left(\max\left(-\left(A^Ty\right)_i, -1\right), 1\right) + (A^Ty)_i\right)^2, & x_i = 0 
    \end{cases}
\end{equation*}
The instability in the \acrshort{kkt} can be interpreted through the following two remarks:
\begin{enumerate}
    \item The stationarity expression is highly sensitive to whether the component $x_i$ equals zero for each $i$. If $|x_i| \leq \varepsilon$ for some $i$ and for some very small $\varepsilon > 0$, the algorithm may face convergence issues.
    \item In the first version of the algorithm, $\bar{x}$ could be zero since it represents the proximal of the $\ell_1$-norm (as defined in (\ref{eqn: BP -- prox})). However, the subsequent update $x = \bar{x} - \tau A^T(\bar{y} - y)$ may result in $x$ being very close to zero but never exactly zero. Consequently, this characteristic could prevent the algorithm from converging. In contrast, the second version of the algorithm performs the last update of $x$ as the proximal of the $\ell_1$-norm, which could be exactly zero.
\end{enumerate}
\begin{figure}[htbp]
    \centering
    \includegraphics[width=0.8\linewidth, height=6cm]{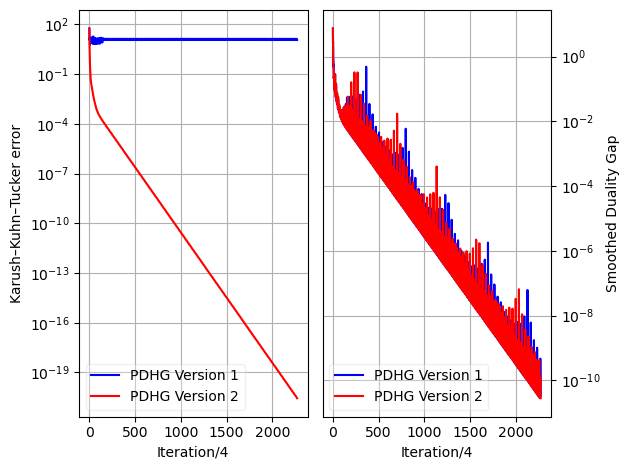}
    \caption{Version 1 vs. version 2 of PDHG for both: the \acrshort{kkt} and \acrshort{sdg}}
    \label{fig:BP-SDG-KKT}
\end{figure}

%% file: Benchmark.tex
\subsection{Benchmark comparison}
In this last subsection, we conduct a benchmark analysis by applying our findings to the six experiments previously discussed. Initially, in Table \ref{table:OG}, we report the number of iterations required by each problem to identify an $\varepsilon = 10^{-8}$ solution, employing various stopping criteria. Note that, the same data is used to compute the 3 measures for each experiment. It provides additional evidence supporting our observations from Figure \ref{fig:DO-OG}, showing that the \acrlong{sdg} consistently achieves $\varepsilon$-solutions with fewer iterations across nearly all of our experiments, thus demonstrating its superior efficiency.

\begin{table}[!ht]
\centering
\begin{tabular}{||c|c|c|c||}
\hline
\diagbox{\textbf{Problems}}{\textbf{Measures}} & \textbf{\acrshort{kkt}} & \textbf{\acrshort{sdg}} & \textbf{\acrshort{pdg}} \\
\hline \hline
One-dimensional (\ref{eqn: 1D -- prob})& 12 & \textbf{11} & 13 \\
I.I.D. Gaussian matrices (\ref{eqn:IIDG})& 5334 & \textbf{5244} & 11538 \\
Non-trivial covariance (\ref{eqn:NTC})& 14274 & \textbf{13968} & 31806 \\
Distributed Optimization (\ref{eqn: DO -- prob})& 5652 & \textbf{4014} & 4659 \\
Quadratic programming (\ref{QP: prob}) & 3492 & \textbf{3172} & 3358 \\
Basis Pursuit (\ref{eqn: BP -- prob}) & $ \gg 10^6$ & \textbf{6920} & 7196 \\
\hline
\end{tabular}
\caption{Iterations needed to identify $\varepsilon = 10^{-8}$ solution using various stopping criteria across experiments}
\label{table:OG}
\end{table}

Subsequently, in Table \ref{table:comparability}, we demonstrate the tightness of our comparability bounds, as outlined in Section \ref{sec:comparability bounds}, across each experiment. This is achieved by presenting the average and standard deviation of the ratio of each result, more specifically, for the bound $\mathcal{M}_1 \leq \mathcal{W}(\mathcal{M}_2)$. The displayed values represent the average and standard deviation of the term $\frac{\mathcal{W}(\mathcal{M}_2)}{\mathcal{M}_1}$. We observe the following: 

\begin{table}[!ht]
\centering
\begin{tabular}{||c|c|c|c|c||}
\hline
\multirow{2}{*}{\textbf{Problem}} & \textbf{Theorem \ref{thm:SDG-KKT}} & \textbf{Theorem \ref{thm:KKT-SDG}} & \textbf{Theorem \ref{thm:SDG-PDG}} & \textbf{Theorem \ref{thm:PDG-SDG -- manifold}} \\
& $\G \leq \underaccent{\bar}{\beta} \K $ & $\K \leq \Bar{\beta}_L \G$ & $\G \leq \mathcal{W}_1(\D)$ & $\D \leq \mathcal{W}_2(\G)$ \\
\hline \hline
(\ref{eqn: 1D -- prob}) & $1.76 \pm 0.6$ & $0.95 \pm 0.15$ & $(6.82 \pm 19.3)10^4$ & $5.0 \pm 3.2$ \\
(\ref{eqn:IIDG}) & $2.02 \pm 0.06$ & $2.2 \pm 0.05$ & $101.65 \pm 25.13$ & $5.23 \pm 5.43$ \\
(\ref{eqn:NTC}) & $2.03 \pm 0.34$ & $5.44 \pm 0.47$ & $15.21 \pm 2.47$ & $42.19 \pm 35.72$\\
(\ref{eqn: DO -- prob}) & $3.72 \pm 3.9$ & $(2.12 \pm 25.9) 10^8$ & $11.8 \pm 9.15$ & $28.17 \pm 10.75$\\
(\ref{QP: prob}) & $1.25 \pm 2.03$ & $+\infty$ & $4.21 \pm 4.7$ & $32.3 \pm 20.98$ \\
(\ref{eqn: BP -- prob}) & $(7.29 \pm 146) 10^{9}$ & $+\infty$ & $2.49 \pm 2.02$ & $20.52 \pm 4.25$ \\
\hline
\end{tabular}
\caption{Average $\pm$ (Standard deviation) of the ratio of each result across each experiment. The ratio of a result, $\mathcal{M}_1 \leq \mathcal{W}(\mathcal{M}_2)$, is: $\frac{\mathcal{W}(\mathcal{M}_2)}{\mathcal{M}_1}$}
\label{table:comparability}
\end{table}

\begin{itemize}
    \item We observe that Theorem \ref{thm:SDG-KKT} provides a tight bound overall. Notably, in the Basis Pursuit experiment, the non-convergence of the \acrshort{kkt} explains the huge average and standard deviation observed.
    \item Concerning Theorem \ref{thm:KKT-SDG}, we note a relatively less tight bound, particularly when the gradient of the objective has a larger Lipschitz constant. The Lipschitz constants for the experiments: the I.I.D. Gaussian matrices, the non-trivial covariance, and the distributed optimization are, approximately: 49, 133.5, and 220.6, respectively. Furthermore, the non-smooth nature of the objective function in both the quadratic programming and basis pursuit experiments justifies the occurrence of $+\infty$ in our results.
    \item Theorems (\ref{thm:SDG-PDG} and \ref{thm:PDG-SDG -- manifold}) provide tight bounds overall. Even in instances where the average ratio may appear elevated, such as observed in Theorem \ref{thm:SDG-PDG} during the one-dimensional experiment, Figure \ref{fig:benchmark} illustrates that this phenomenon occurs primarily within the initial iterations.
\end{itemize}
\begin{figure}[htbp]
    \centering
\subfloat[One-dimensional]{\includegraphics[width=0.33\linewidth]{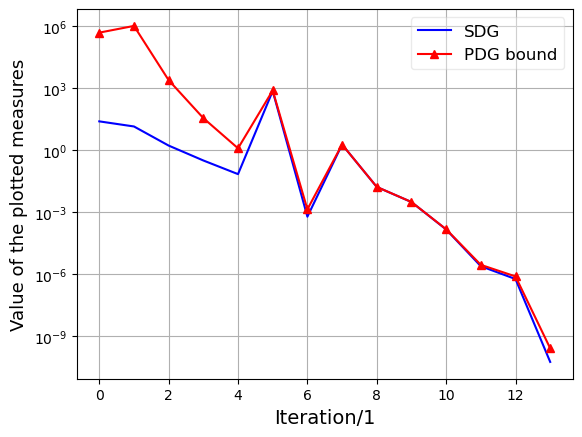} \label{fig:thm7:1D}} 
\subfloat[Non-trivial covariance]{ \includegraphics[width=0.33\linewidth]{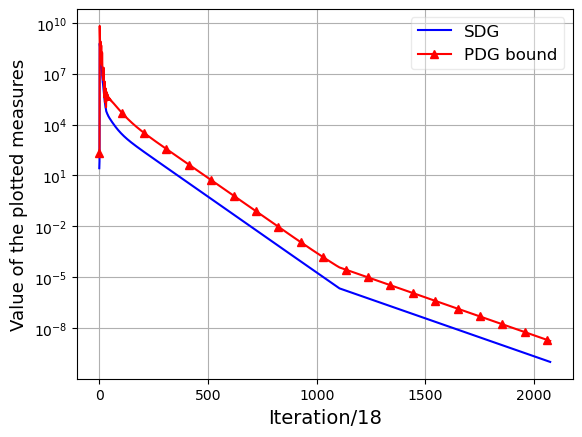} \label{fig:thm7:COV}}
\subfloat[Basis pursuit]{ \includegraphics[width=0.33\linewidth]{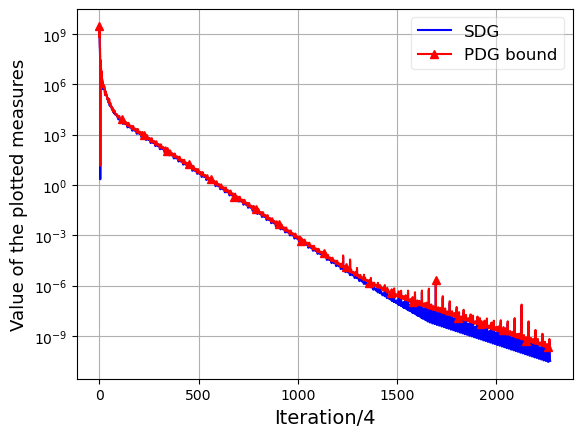} \label{fig:thm7:BP}}
\caption{Numerical illustration of Theorem \ref{thm:SDG-PDG}} 
\label{fig:benchmark}
\end{figure}
\begin{figure}[htbp]
    \centering
\subfloat[One-dimensional]{\includegraphics[width=0.33\linewidth]{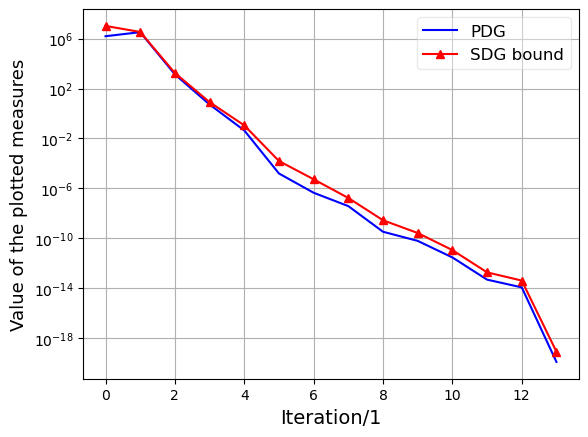} \label{fig:thm8:1D}}
\subfloat[Non-trivial covariance]{\includegraphics[width=0.33\linewidth]{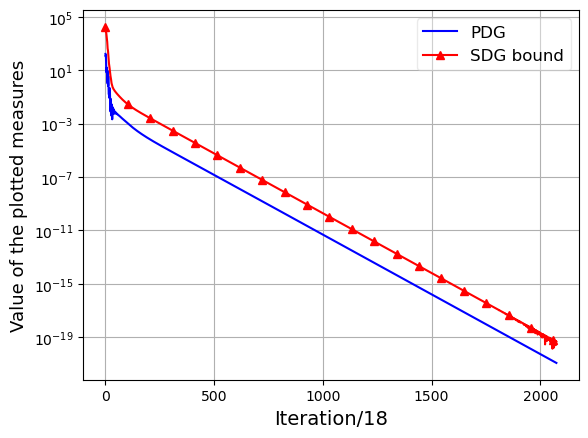} \label{fig:thm8:COV}}
\subfloat[Basis pursuit]{\includegraphics[width=0.33\linewidth]{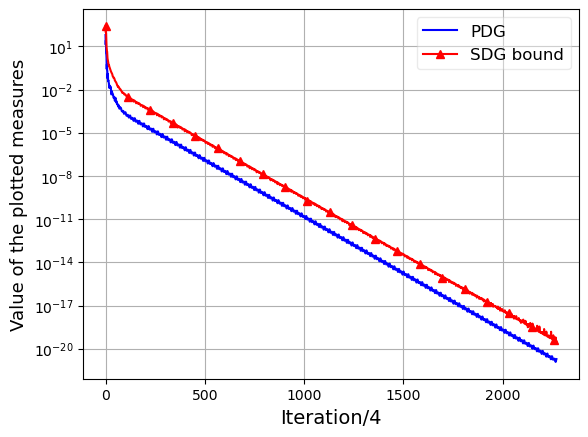} \label{fig:thm8:BP}}
\caption{Numerical illustration of Theorem \ref{thm:PDG-SDG -- manifold}} 
\end{figure}

%% file: Conclusion_and_Perspectives.tex
In this paper, we have studied several stopping criteria: \acrshort{ogfe}, the \acrshort{kkt}, \acrshort{pdg}, and \acrshort{sdg} to determine under which conditions they are accurate to detect $\varepsilon-$solutions. In the realm of convex optimization problems under affine-equality constraints, our findings have led to significant insights:
\begin{itemize}
    \item The efficacy of \acrshort{sdg} stands on par with both: the \acrshort{kkt} and \acrshort{pdg} given specific conditions.
    
    \item By assuming \acrshort{MSRSDL} or leveraging \acrshort{qebsg}, we have derived that the \acrshort{kkt}, or \acrshort{sdg} and \acrshort{pdg}, respectively, serve as practical upper bounds for the \acrlong{og}, providing an effective approximation.
    
    \item Our investigation vividly demonstrates the superior stability of \acrshort{sdg} over the \acrshort{kkt}.
    
    \item Although our methodology is rooted in affine-equality constraints, our findings extend their applicability to encompass other problems entailing inequality constraints.
\end{itemize}
This work opens several perspectives:
\begin{itemize}
    \item Given our utilization of the \acrshort{qebsg} assumption in establishing the \acrshort{pdg} approximation for the \acrshort{og}, it prompts the question: could enhancing \acrshort{pdg} with a distinct regularity assumption potentially improve its performance?
    
    \item Is it feasible to develop a novel algorithm leveraging \acrshort{sdg} as its stopping criterion?
    \item What about the applicability of our findings in non-convex optimization settings?

\end{itemize}

%% file: Proofs.tex
\begin{itemize}[label={$\blacktriangleright$}]
    \item \textbf{Lemma [\ref{lem:LC-LS--MSR}]} \textit{For the \ref{eqn: LS -- prob} problem, let $z^{\star} = \proj_{\Z^{\star}}(z)$. Then, for any $z \in \Z$ the Lagrangian's sub-differential satisfies:
    \begin{align*}
        \left\|\nabla_x \Lag(x, y)\right\| + \left\|\nabla_y \Lag(x, y)\right\| & \geq \left|\lambda(\mathcal{M})\right|_{\min} \dist\left(z, \Z^{\star}\right) \\
        \mathcal{M} &= \begin{bmatrix} Q^TQ & A^T \\ A & 0\end{bmatrix}
    \end{align*}
    where $\left|\lambda\left(\mathcal{M}\right)\right|_{\min}$ is the smallest positive eigenvalue of $\mathcal{M}$.}
\begin{proof}
 For any $z \in \Z$, we have: 
    \begin{align*}
        \left\|\nabla_x \Lag(x, y) \right\| + \left\| \nabla_y \Lag(x, y) \right\| & \geq \left\|\nabla \Lag(x, y) \right\| \\ 
        &= \left\|\nabla \Lag(x, y) - \nabla \Lag(x^{\star}, y^{\star}) \right\| \\ 
        &= \left\|\begin{bmatrix} Q^TQ(x - x^{\star}) + A^T(y - y^{\star}) \\ A(x - x^{\star}) \end{bmatrix}\right\| \\ 
        &= \|\mathcal{M} (z - z^{\star})\| \\ 
        & \stackrel{(\ref{eqn:Rayleigh})}{\geq} |\lambda(\mathcal{M})|_{\min} \dist(z, \Z^*) \smartqed 
    \end{align*}
\end{proof} 
    \item \textbf{Lemma [\ref{lem:LC-LS--QEB}]} \textit{For the \ref{eqn: LS -- prob} problem, the self-centered smoothed gap could be reformulated into a quadratic form. That is, for any $z \in \Z$:
    \begin{equation*}
        \G(z) = z^T\mathcal{H} z + \langle z, v \rangle + cst 
    \end{equation*}
    where, 
    \begin{align*}
       v_{(n + m)} = \begin{bmatrix} v_x \\ v_y\end{bmatrix} && \mathcal{H}_{{(n+m)\times(n+m)}} = \begin{bmatrix}
            M_{xx}& \frac{1}{2} M_{xy} \\ \frac{1}{2} M_{xy}^T & M_{yy}
        \end{bmatrix}
    \end{align*}    
\begin{table}[!ht]
    \centering
    \begin{tabular}{||l|c||} 
    \hline
    \multicolumn{1}{||c}{\textbf{Vector - Matrix}} & \multicolumn{1}{|c||}{\textbf{Size}} \\
    \hline \hline  
    $v_x = \frac{\beta_x}{\tau}B^{-1}Q^Tc -Q^Tc$ &  $n \times 1$  \\
    $v_y = -AB^{-1}Q^Tc$ & $m \times 1$\\
    $B = Q^TQ + \frac{\beta_x}{\tau} \Id_{n}$ & $n \times n$  \\
    $M_{xx} = \frac{1}{2}Q^TQ + \frac{\sigma}{2\beta_y} A^TA + \frac{\beta_x^2}{2\tau^2} B^{-1} - \frac{\beta_x}{2\tau} \Id_{n}$ & $n \times n$ \\
     $M_{xy} = A^T - \frac{\beta_x}{\tau} B^{-1}A^T$ & $n \times m$  \\
     $M_{yy} = \frac{1}{2} AB^{-1}A^T$ & $m \times m$ \\
      \hline 
    \end{tabular}
\end{table}}
\begin{proof}
 We start with the definition of the smoothed duality gap (eqn:\ref{eqn:SDG general}): 
\begin{align*}
    \G(z) &= \sup_{z' \in \Z}~ \Lag(x, y') - \Lag(x', y) -\frac{\beta_x}{2\tau}\|x' - x\|^2 - \frac{\beta_y}{2\sigma} \|y' - y\|^2 \\ 
    &= \begin{multlined}[t] \frac{1}{2}\|Qx - c\|^2 + \langle b, y \rangle + {\color{red} \sup_{y'} \left(\langle Ax - b, y' \rangle - \frac{\beta_y}{2\sigma}\|y' - y\|^2 \right)} \\ + {\color{blue}  \sup_{x'} \bigg(- \frac{1}{2} \|Qx' - c\|^2 - \langle x', A^Ty \rangle - \frac{\beta_x}{2\tau}\|x' - x\|^2 \bigg)} \end{multlined}\\ 
    &= \begin{multlined}[t] \frac{1}{2}\|Qx - c\|^2 + \langle b, y \rangle + {\color{red} \left(\langle Ax - b, \Tilde{y} \rangle - \frac{\beta_y}{2\sigma}\|\Tilde{y} - y\|^2 \right)} \\ - {\color{blue}  \left( \frac{1}{2} \|Q\Tilde{x} - c\|^2 - \langle \Tilde{x}, A^Ty \rangle  - \frac{\beta_x}{2\tau}\|\Tilde{x} - x\|^2 \right)} \end{multlined}
    \end{align*}     
\begin{equation*}
\text{with}  ~ \begin{cases} \displaystyle
        \Tilde{y} &= \arg\max_{y'} \left(\langle Ax - b, y' \rangle - \frac{\beta_y}{2\sigma}\|y' - y\|^2 \right) \\ 
        &= y + \frac{\sigma}{\beta_y} \left(Ax - b\right) \\
        \Tilde{x} &= \arg\max_{x'} \left(- \frac{1}{2} \|Qx' - c\|^2 - \langle x', A^Ty \rangle  - \frac{\beta_x}{2\tau}\|x' - x\|^2 \right) \\ 
        &= B^{-1} \left(Q^Tc - A^Ty + \frac{\beta_x}{\tau}x\right) ~\text{where}~ B = Q^TQ + \frac{\beta_x}{\tau} \mathrm{Id} 
    \end{cases}
\end{equation*}
    Substituting $\Tilde{x}$ and $\Tilde{y}$, simplifying, and gathering the related terms we get: 
    \begin{align*}
     &\G(z) = \begin{multlined}[t] \overbrace{\frac{1}{2}\|Qx\|^2 + \frac{\sigma}{2\beta_y} \|Ax\|^2 -\frac{1}{2} \left\|\frac{\beta_x}{\tau}QB^{-1}x\right\|^2 - \frac{\beta_x}{2\tau}\left\|\left(\frac{\beta_x}{\tau} B^{-1} - \Id\right)x\right\|^2}^{T_{xx}} \end{multlined}\\ 
     &\underbrace{+ \langle x, A^Ty \rangle - \frac{\beta_x}{\tau} \left(\left\langle B^{-1}x, A^Ty \right\rangle  - \left\langle QB^{-1}x, QB^{-1}A^Ty\right\rangle  - \left\langle \left(\frac{\beta_x}{\tau} B^{-1} - \Id\right)x ,B^{-1}A^Ty \right\rangle\right)}_{T_{xy}} \\ 
     &\underbrace{+ \left\langle B^{-1}A^Ty, A^Ty \right\rangle - \frac{1}{2}\left\|QB^{-1}A^Ty\right\|^2 - \frac{\beta_x}{2\tau} \left\|B^{-1}A^Ty\right\|^2}_{T_{yy}} \\ 
     &\underbrace{- \langle Qx, c \rangle - \frac{\sigma}{\beta_y} \langle Ax, b \rangle - \left\langle \frac{\beta_x}{\tau}QB^{-1}x, QB^{-1}Q^Tc - c\right\rangle - \frac{\beta_x}{\tau}\left\langle \left(\frac{\beta_x}{\tau} B^{-1} - \Id\right)x ,B^{-1}Q^Tc\right\rangle}_{T_{x}}\\ 
     &\underbrace{ - \left\langle B^{-1}Q^Tc, A^Ty \right\rangle + \left\langle QB^{-1}A^Ty, QB^{-1}Q^Tc - c \right\rangle + \frac{\beta_x}{\tau} \left\langle B^{-1}Q^Tc, B^{-1}A^Ty \right\rangle}_{T_y} \\ 
     &\underbrace{+ \frac{1}{2}\|c\|^2 + \frac{\sigma}{2\beta_y}\|b\|^2 - \frac{1}{2} \left\|QB^{-1}Q^Tc - c\right\|^2  -\frac{\beta_x}{2\tau} \left\| B^{-1}Q^Tc\right\|^2 }_{cst}
\end{align*}
Now, we will further simplify each sub-term, for instance: 
\begin{align*}
     T_{xx} &= \frac{1}{2}\|Qx\|^2 + \frac{\sigma}{2\beta_y} \|Ax\|^2 -\frac{1}{2} \left\|\frac{\beta_x}{\tau}QB^{-1}x\right\|^2 - \frac{\beta_x}{2\tau}\left\|\left(\frac{\beta_x}{\tau} B^{-1} - \Id\right)x\right\|^2 \\
    &= \begin{multlined}[t]\frac{1}{2} \left\langle Qx, Qx\right\rangle + \frac{\sigma}{2\beta_y} \left\langle Ax, Ax\right\rangle - \frac{\beta_x^2}{2\tau^2} \left\langle QB^{-1}x, QB^{-1}x\right\rangle \\ - \frac{\beta_x}{2\tau} \left\langle \left(\frac{\beta_x}{\tau}B^{-1} - \Id\right)x, \left(\frac{\beta_x}{\tau}B^{-1} - \Id\right)x\right\rangle\end{multlined} \\ 
    &= \left\langle x,  M_{xx} x \right\rangle 
\end{align*}
where 
\begin{align*}
   M_{xx} &= \frac{1}{2}Q^TQ + \frac{\sigma}{2\beta_y} A^TA - \frac{\beta_x^2}{2\tau^2} B^{-1}Q^TQB^{-1}  - \frac{\beta_x}{2\tau}\left(\frac{\beta_x}{\tau}B^{-1} - \Id\right)^T\left(\frac{\beta_x}{\tau}B^{-1} - \Id\right) \\ 
    &= \frac{1}{2}Q^TQ + \frac{\sigma}{2\beta_y} A^TA - \frac{\beta_x^2}{2\tau^2} B^{-1}Q^TQB^{-1} - \frac{\beta_x^3}{2\tau^3} B^{-1}B^{-1} + \frac{\beta_x^2}{\tau^2}B^{-1} - \frac{\beta_x}{2\tau} \Id \\ 
    &= \frac{1}{2}Q^TQ + \frac{\sigma}{2\beta_y} A^TA - \frac{\beta_x^2}{2\tau^2} B^{-1} \underbrace{\left( Q^TQ + \frac{\beta_x}{\tau} \Id \right)}_{B}B^{-1} + \frac{\beta_x^2}{\tau^2}B^{-1} - \frac{\beta_x}{2\tau} \Id \\ 
    &= \frac{1}{2}Q^TQ + \frac{\sigma}{2\beta_y} A^TA - \frac{\beta_x^2}{2\tau^2} B^{-1} B B^{-1} + \frac{\beta_x^2}{\tau^2}B^{-1} - \frac{\beta_x}{2\tau} \Id \\ 
    &=  \frac{1}{2}Q^TQ + \frac{\sigma}{2\beta_y} A^TA + \frac{\beta_x^2}{2\tau^2} B^{-1} - \frac{\beta_x}{2\tau} \Id\\ 
\end{align*}
Same kind of simplifications could be done for the other sub-terms. Therefore, we can rewrite the smoothed duality gap as follows: 
\begin{align*}
    \G(z) &= \left\langle x, M_{xx} x \right\rangle + \left\langle x, M_{xy} y \right\rangle +  \left\langle y, M_{yy} y \right\rangle  +  \langle x, v_x \rangle + \langle y, v_y \rangle  + cst \\ 
    &= z^T\mathcal{H} z + \langle z, v \rangle + cst  \smartqed
\end{align*}
\end{proof}
\end{itemize}

%% file: Last_results.tex
Within this appendix, we provide an exhaustive proof of Theorem \ref{thm:PDG-SDG -- manifold}. We start by defining two vectors pivotal to demonstrating the result, along with outlining their key properties. 
\begin{enumerate}
    \item The first one is:
\begin{equation} \label{eqn:p*}
    p^* := \prox_{\beta_x f^*}\left(\beta_x x - A^Ty\right)
    \end{equation}
Where $f^* \in \Gamma_0(\X)$ is the Fenchel-Conjugate of the function $f$. Then, by Moreau's identity (\ref{eqn: Moreau's identity }), we get: 
\begin{equation} \label{SDG-PDG:Moreau's identity}
    p + \frac{1}{\beta_x} p^* = x 
\end{equation}
\item The second one comes from Lemma \ref{lem:p and fermat}: $p = \prox_{\beta_x^{-1}f} \left(x - \frac{1}{\beta_x}A^Ty\right) \iff \beta_x(x - p) \in \partial f(p) + A^Ty$. So, we define: 
\begin{equation} \label{eqn:a tilde}
    \Tilde{a} :=  -A^Ty + \beta_x (x - p) \in \partial f(p)
\end{equation}
Remark that since $\Tilde{a} \in \partial f(p)$ this ensures that $\Tilde{a} \in \dom f^*$, which is a necessary condition for the following properties to hold. 
\begin{align}
    & \text{By (\ref{eqn:proj - norms})} & \left\|a + A^Ty\right\| \leq \left\|\Tilde{a} + A^Ty\right\| = \|p^*\|  \label{PDG-SDG:proj -- norms}\\ 
   & \text{By (\ref{eqn:proj - inner})} & \left\langle -A^Ty - a, \Tilde{a} - a \right\rangle \leq 0 \label{PDG-SDG:proj -- inner} \\
   &\text{By (\ref{eqn:Fenchel-Young})}  &f(p) + f^*(\Tilde{a}) = \langle p, \Tilde{a} \rangle \label{PDG-SDG:Fencher-Young -- p et a tilde}
\end{align}
\end{enumerate}
\begin{lemma} \label{lem:PDG-SDG:a, a tilde, p*}
    For $a, p^*$, and $\Tilde{a}$ defined, respectively, in (\ref{eqn:PDG - a}), (\ref{eqn:p*}), and (\ref{eqn:a tilde}) we have: 
    \begin{equation} \label{PDG-SDG:a, a tilde, p*}
        \|a - \Tilde{a}\| \leq \|p^*\| 
    \end{equation}
\end{lemma}
\begin{proof}
\begin{align*}
    \|a - \Tilde{a} \|^2 &= \left\|a + A^Ty\right\|^2 + \left\|\Tilde{a} + A^Ty\right\|^2 - 2\left\langle a + A^Ty,  \Tilde{a} + A^Ty \right\rangle \\ 
    &= \left\|a + A^Ty\right\|^2 + \left\|\Tilde{a} + A^Ty\right\|^2 + {\color{blue} 2\left\langle - a - A^Ty,  \Tilde{a} - a \right\rangle} + 2\left\langle - a - A^Ty,  a + A^Ty \right\rangle \\
    &\stackrel{(\ref{PDG-SDG:proj -- inner})}{\leq} \left\|a + A^Ty\right\|^2 + \left\|\Tilde{a} + A^Ty\right\|^2 + {\color{blue} 0} - 2\left\|a + A^Ty \right\|^2 \\
    &\leq  \left\|\Tilde{a} + A^Ty\right\|^2  \smartqed 
\end{align*}
\end{proof}
\begin{corollary} \label{coro:p*<G}
    Given $\beta = (\beta_x, \beta_y) \in [0, +\infty]^2$, then for $p^*$ defined in (\ref{eqn:p*}) and for all $z \in \Z$ the self-centered \acrlong{sdg} satisfies: 
    \begin{equation} \label{eqn:p*<G}
        \G(z) \geq \frac{1}{2\beta_x} \|p^*\|^2 
    \end{equation}
\end{corollary}
\begin{proof} From Lemma \ref{lem:G lower bound}, we know: 
\begin{equation*}
    \G(z) \geq \frac{\beta_x}{2} \|x - p\|^2 + \frac{1}{2\beta_y}\|Ax - b\|^2 \geq \frac{\beta_x}{2} \|x - p\|^2 \stackrel{(\ref{SDG-PDG:Moreau's identity})}{=} \frac{\beta_x}{2} \left\|\frac{1}{\beta_x} p^*\right\|^2 \smartqed
\end{equation*}
\end{proof}
\begin{lemma}  \label{lem:PDG-SDG:DG lower bound}
    Given $\beta = (\beta_x, \beta_y) \in [0, +\infty]^2$, a function $f \in \Gamma_0(\X)$, and let $a = \proj_{\dom f^*}(\mu)$. Then, for any $z \in \Z$ the self-centered \acrlong{sdg} defined in (\ref{eqn:SDG}) satisfies: 
    \begin{equation} \label{PDG-SDG:DG lower bound}
        f(x) + f^*(a) + \langle b, y \rangle \geq -\left( \sqrt{2\beta_x}\|x\| + \sqrt{2\beta_y}\|y\|\right) \sqrt{\G(z)}
    \end{equation}
\end{lemma}
\begin{proof}
\begin{align*}
    f(x) + f^*(a) + \langle b, y \rangle &= {\color{red} f(x) + f^*(a)  - \langle x, a \rangle} +  \left\langle b - Ax, y \right\rangle + \left\langle x, A^Ty + a \right\rangle \\ 
    &\stackrel{(\ref{eqn:Fenchel-Young})}{\geq} {\color{red} 0} + \left\langle y, b - Ax \right\rangle + \left\langle x, a + A^Ty \right\rangle \\ 
    &\geq - \|y\| {\color{blue} \left\|Ax - b \right\|} - \|x\| \left\|a + A^Ty \right\| \\ 
    &\stackrel{(\ref{eqn:FG-SDG})}{\geq} -\|y\| {\color{blue} \sqrt{2\beta_y \G(z)}} - \|x\| {\color{red} \left\|a + A^Ty \right\|} \\ 
    &\stackrel{(\ref{PDG-SDG:proj -- norms})}{\geq} -\|y\| \sqrt{2\beta_y \G(z)} - \|x\| {\color{red}\left\|p^* \right\|} \\ 
    &\stackrel{(\ref{eqn:p*<G})}{\geq} -\|y\| \sqrt{2\beta_y \G(z)} - \|x\| {\color{red} \sqrt{2\beta_x \G(z)}} \smartqed
\end{align*}
\end{proof}
\begin{lemma} \label{lem:PDG-SDG:Approx DG upper bound}
    Given $\beta = (\beta_x, \beta_y) \in [0, +\infty]^2$ and a function $f \in \Gamma_0(\X)$. Then, for $\Tilde{a}$ defined in (\ref{eqn:a tilde}), and for any $z \in \Z$ the self-centered \acrlong{sdg} defined in (\ref{eqn:SDG}) satisfies: 
    \begin{equation} \label{PDG-SDG:Approx DG upper bound}
        f(x) + f^*(\Tilde{a}) - \langle x, \Tilde{a} \rangle \leq \G(z) 
    \end{equation}
\end{lemma}
\begin{proof}
\begin{align*}
    \G(z) &= f(x) {\color{red} - f(p)} + \left\langle A(x - p), y \right\rangle -\frac{\beta_x}{2} \|x - p\|^2 + \frac{1}{2\beta_y} \left\|Ax - b\right\|^2 \\ 
    &\stackrel{(\ref{PDG-SDG:Fencher-Young -- p et a tilde})}{=}  f(x)  {\color{red} + f^*(\Tilde{a}) - \langle p, \Tilde{a} \rangle} + \left\langle A(x - p), y \right\rangle - \frac{\beta_x}{2}\|x - p\|^2 + \frac{1}{2\beta_y}\left\|Ax- b\right\|^2\\ 
    &=  f(x) + f^*(\Tilde{a}) - \left\langle p, -A^Ty + \beta_x(x - p) \right\rangle + \left\langle A(x - p), y \right\rangle - \frac{\beta_x}{2}\|x - p\|^2 + \frac{1}{2\beta_y}\left\|Ax- b\right\|^2 \\ 
    &= f(x) + f^*(\Tilde{a}) - \left\langle {\color{blue}x}, \beta_x(x - p) \right\rangle + \left\langle {\color{blue}x} - p, \beta_x(x - p) \right\rangle - \frac{\beta_x}{2}\|x - p\|^2+ \left\langle Ax, y \right\rangle+ \frac{1}{2\beta_y}\left\|Ax- b\right\|^2 \\ 
    &=  f(x) + f^*(\Tilde{a}) - \left\langle x, \Tilde{a} \right\rangle +  \frac{\beta_x}{2}\|x - p\|^2 + \frac{1}{2\beta_y}\left\|Ax- b\right\|^2\\ 
    &\geq f(x) + f^*(\Tilde{a}) - \langle x, \Tilde{a} \rangle \smartqed 
\end{align*}
\end{proof}
\begin{lemma} \label{lem:PDG-SDG:primal-dual feas bound}
    Given $\beta = (\beta_x, \beta_y) \in [0, +\infty]^2$ and a function $f \in \Gamma_0(\X)$. Let $a = \proj_{\dom f^*}(\mu)$, and $\beta_{\max} = \max(\beta_x, \beta_y)$. Then, the primal-dual feasibility terms in (\ref{eqn:PDG}) could be approximated in terms of the self-centered smoothed gap defined in (\ref{eqn:SDG}). More precisely, for any $z \in \Z$: 
    \begin{equation} \label{PDG-SDG:primal-dual feas bound}
        \left\|a + A^Ty\right\|^2 + \left\|Ax - b\right\|^2 \leq 2\beta_{\max} \G(z) 
    \end{equation}
\end{lemma}
\begin{proof}
 From Lemma \ref{lem:G lower bound}, we know: 
\begin{align*}
    \G(z) &\geq \frac{\beta_x}{2} \|x - p\|^2 + \frac{1}{2\beta_y}\|Ax - b\|^2 \\ 
    &\stackrel{(\ref{SDG-PDG:Moreau's identity})}{=} \frac{1}{2\beta_x} \|p^*\|^2 + \frac{1}{2\beta_y}\|Ax - b\|^2 \\ 
    &\stackrel{(\ref{PDG-SDG:proj -- norms})}{\geq}  \frac{1}{2\beta_x} \left\|a + A^Ty\right\|^2 + \frac{1}{2\beta_y} \left\|Ax - b\right\|^2 \\ 
    &\geq \frac{1}{2\beta_{\max}} \left(\left\|a + A^Ty \right\|^2 + \left\|Ax - b\right\|^2\right) \smartqed 
\end{align*}
\end{proof}
\begin{lemma} \label{lem:diffeom of affine space}
    Let $\mathcal{A} \subseteq \R^n$ be a non-empty affine space, and $x_0 \in \mathcal{A}$. Then for the vector space $V:= \mathcal{A} - x_0$, which could be seen as a smooth manifold, one can define a unique coordinate chart ($\U, \phi$). Moreover, for any $x, \Tilde{x} \in \mathcal{A}$, 
    \begin{enumerate}
        \item The vector $x \in \mathcal{A}$ could be rewritten in terms of the diffeomorphism $\phi$. Say otherwise:
        \begin{equation}  \label{eqn:x = xo + phi inv} x = x_0 + \phi^{-1}(\lambda), \hspace{1cm} \forall \lambda \in \phi (\U)  \end{equation}
        \item The defined diffeomorphism preserves the norm. Say otherwise,
        \begin{equation} \label{eqn:phi preserves the norm}
            \left\|\phi(x - x_0) - \phi(\Tilde{x} - x_0)\right\| = \|x - \Tilde{x}\|
        \end{equation}
        \item Denoting $J_{\phi^{-1}}$ the Jacobian of the diffeomorphism $\phi^{-1}$, we have: 
        \begin{equation} \label{eqn:Jacobian of phi}
            J_{\phi^{-1}} \left(\phi(x - x_0) - \phi(\Tilde{x} - x_0)\right) = x - \Tilde{x}
        \end{equation}
    \end{enumerate}
\end{lemma}
\begin{proof}
\begin{enumerate}
    \item  Since $V$ is a vector space, one can find an orthonormal basis for it. Say, $ \mathcal{B} = \{v_1, \dots, v_n\}$

            Thus, for any $v \in V, \exists!~ \lambda_1, \dots, \lambda_n \in \R$ such that $\displaystyle v = \sum_{i = 1}^n \lambda_i v_i$ 

            Therefore, one could define the unique diffeomorphism as follows: 
        \begin{equation} \label{eqn:phi}
            \phi \colon \mathcal{U} = V \longrightarrow \hat{\mathcal{U}} \subseteq \mathbb{R}^n \hspace{1cm} \text{s.t.} \hspace{1cm} \phi(v) = \phi\left(\sum_i \lambda_i v_i\right) = \lambda
        \end{equation}
        Equivalently, for any $x \in \mathcal{A}, \exists! \lambda_1, \dots, \lambda_n \in \R$ such that:
        \begin{equation} \label{eqn:diffeom of x in affine space}
            x = x_0 + v = x_0 + \sum_{j} \lambda_j v_j \stackrel{(\ref{eqn:phi})}{=} x_0 + \phi^{-1}(\lambda) 
        \end{equation} \begin{equation*}\smartqed\end{equation*}
    \item Let $x, \Tilde{x} \in \mathcal{A}$, then by (\ref{eqn:x = xo + phi inv}), we have: 
        \begin{align*}
            x = x_0 + \sum_{i = 1}^n \lambda_i v_i &&  \Tilde{x} = x_0 + \sum_{i = 1}^n \Tilde{\lambda}_i v_i
        \end{align*}
    Since $x - x_0 ~\&~ \Tilde{x} - x_0 \in \U$, we define: 
    \begin{align*}
        \phi(x - x_0) := \lambda = \in \R^n && \phi(\Tilde{x} - x_0) := \Tilde{\lambda} = \in \R^n
    \end{align*}
    Then, by \textit{Pythagorean theorem}, we obtain: 
    \begin{align*}
        \|x - \Tilde{x}\|^2 &= \left\| \left(x_0 + \sum_{i = 1}^n \lambda_i v_i \right) - \left(x_0 + \sum_{i = 1}^n \Tilde{\lambda}_iv_i\right)\right\|^2 \\ 
        &= \left\|\sum_{i = 1}^n \left(\lambda_i - \Tilde{\lambda}_i\right) v_i \right\|^2 \\ 
        &= \sum_{i = 1}^n (\lambda_i - \Tilde{\lambda}_i)^2 \\
        &= \|\lambda - \Tilde{\lambda}\|^2  \smartqed
    \end{align*}
    \item Since $\phi$ is a diffeomorphism, then:  
 $$\phi(x - x_0) = \lambda \iff \phi^{-1}(\lambda) = x - x_0 \stackrel{(\ref{eqn:diffeom of x in affine space})}{=} \sum_{i = 1}^n \lambda_i v_i \Longrightarrow J_{\phi^{-1}} = \begin{bmatrix} v_1 & \dots & v_n \end{bmatrix}$$
    Hence, 
    \begin{align*}
        J_{\phi^{-1}} \left(\phi(x - x_0) - \phi(\Tilde{x} - x_0)\right) &= J_{\phi^{-1}} (\lambda - \Tilde{\lambda}) \\ 
        &=  \begin{bmatrix} v_1 & \dots & v_n \end{bmatrix} \begin{bmatrix} \lambda_1 - \Tilde{\lambda}_1 \\ \vdots \\ \lambda_n - \Tilde{\lambda}_n \end{bmatrix} \\ 
        &= \sum_{i = 1}^n (\lambda_i - \Tilde{\lambda}_i) v_i \\
        &\stackrel{(\ref{eqn:diffeom of x in affine space})}{=} x - \Tilde{x}  \smartqed 
     \end{align*}
\end{enumerate}
\end{proof}
\begin{itemize}[label={$\blacktriangleright$}]
    \item \textbf{Theorem [\ref{thm:PDG-SDG -- manifold}]} \textit{
Given $\beta = (\beta_x, \beta_y) \in (0, +\infty)^2$, $z \in \Z$, and a function $f \in \Gamma_0(\X)$. Then, under the following set of assumptions (we denote it $\mathcal{E}$):
\begin{itemize}
    \item The Fenchel-Conjugate of the objective function, $f^*$, could be written in a separable way:
    \begin{equation} \label{APP:eqn: f* is separable}   f^*(\mu) = f_1^*(\mu_1) + f_2^*(\mu_2), \hspace{1.5cm} \mu \in \X \end{equation}
    \item $f_1^*$ is $L_{f_1^*}-$Lipschitz on its domain, $\dom f_1^*$. 
    \item The domain of $f^*_2$ is a non-empty affine space. 
    \item Let $\mu_0 \in \dom f_2^*$, then $\forall \mu_2 \in\dom f_2^*$, we define 
    \begin{equation} \label{APP:eqn:def of g} g(\lambda) = f_2^*\left(\mu_0 + \phi^{-1}(\lambda)\right) = f_2^*(\mu_2)\end{equation} where $\phi$ is the diffeomorphism defined in Lemma \ref{lem:diffeom of affine space} (eqn:\ref{eqn:phi}). 
    \item The function $g$ is differentiable and has an $L_g-$Lipschitz gradient. 
\end{itemize}
The \acrlong{pdg} and the \acrlong{sdg} defined, respectively, in (\ref{eqn:PDG}) and (\ref{eqn:SDG}) satisfy: 
\begin{equation*}
\begin{split}
\D(z) &\leq \begin{multlined}\left(\left(3 + \beta_x L_g\right)\G(z) + \left(\sqrt{2\beta_x}\left(2\|x\| + L_{f_1^*}\right)+ \sqrt{2\beta_y}\|y\|\right)\sqrt{\G(z)}\right)^2 \\ + 2\beta_{\max} \G(z)\end{multlined} \\ 
    a &= \proj_{\dom f^*} \left(-A^Ty\right)  \hspace{2cm} p= \prox_{\beta_x^{-1}f}\left(x - \frac{1}{\beta_x} A^Ty\right) 
\end{split}
\end{equation*}} 
\end{itemize}
\begin{proof}
We initiate by analyzing and interpreting the assumptions:
\begin{itemize}
    \item $f^*(\mu) = f_1^*(\mu_1) + f^*_2(\mu_2)$ is separable implies that:
    \begin{align}
    \partial f^*(\mu) &= \partial f_1^*(\mu_1) \times \partial f_2^*(\mu_2)  \label{APP:eqn:f* is separable -- sub-grad}\\ 
    \dom f^* &= \dom f_1^* \times \dom f_2^*  \label{APP:eqn: dom is separable}
    \end{align}
    So, $p \in \partial f^*(\Tilde{a}) \stackrel{(\ref{APP:eqn:f* is separable -- sub-grad})}{=}\partial f_1^*(\Tilde{a}_1) \times \partial f_2^*(\Tilde{a}_2)$ if, and only if, 
    \begin{equation} \label{APP:eqn: p1 et p2}
     p = (p_1, p_2) ~ \& ~ \Tilde{a} = (\Tilde{a}_1, \Tilde{a}_2) ~~~\text{s.t.}~~~ p_1 \in \partial f_1^*(\Tilde{a}_1) ~\&~ p_2 \in \partial f_2^*(\Tilde{a}_2)
    \end{equation}
    Also, for 
    \begin{align*}
        a = \proj_{\dom f^*} \left(-A^Ty\right) &\stackrel{(\ref{APP:eqn: dom is separable})}{=} \proj_{\dom f_1^* \times \dom f_2^*} \left(-A^Ty\right) \\
        &=  \left(\proj_{\dom f_1^* } \left(-\left(A^Ty\right)_1\right), \proj_{\dom f_2^*} \left(-\left(A^Ty\right)_2\right)\right)
    \end{align*}
    So, for $A^Ty = \left(\left(A^Ty\right)_1, \left(A^Ty\right)_2\right)$, we let $a = (a_1, a_2)$ be defined as follows: 
    \begin{align} \label{APP:eqn: a1 et a2}
        a_1 = \proj_{\dom f_1^*} \left(-\left(A^Ty\right)_1\right) &&  a_2 = \proj_{\dom f_2^*} \left(-\left(A^Ty\right)_2\right)
    \end{align}
    \item $f_1^*$ is $L_{f_1^*}-$Lipschitz on its domain implies: 
    \begin{equation} \label{APP:eqn:f1* is lipschitz}
        f_1^*(a_1) \leq f_1^*(\Tilde{a}_1) + L_{f_1^*} \|a_1 - \Tilde{a}_1\| 
    \end{equation}
    \item By Lemma \ref{lem:diffeom of affine space}, we have seen that we can rewrite any vector $\mu_2 \in \dom f_2^*$ as: 
    \begin{equation}
        \mu_2 = \mu_0 + \sum_{i = 1}^n \lambda_i v_i = \mu_0 + \phi^{-1}(\lambda) 
    \end{equation}
    where $\mu_0 \in \dom f_2^*, \mathcal{B} = \{v_1, \dots, v_n\}$ is an orthonormal basis for $V = \dom f_2^* - \mu_0$, and $\phi(\mu_2 - \mu_0 \in V):= \lambda$. Thus, for any $\mu_2 \in \dom f_2^*$:
    \begin{equation}
        f_2^*(\mu_2) = f_2^*\left(\mu_0 + \sum_{i = 1}^n \lambda_i v_i\right) = f_2^*\left(\mu_0 + \phi^{-1}(\lambda)\right) 
    \end{equation}
    Hence, by defining $g(\lambda) = f_2^*(\mu_0 + \phi^{-1}(\lambda))$ and assuming that $g$ is differentiable, we mean that the function $f_2^*$ is differentiable on its domain. 
    \item The function $g$ is differentiable, so by the chain rule: $\{\nabla g(\lambda)\} = J_{\phi^{-1}}^T \partial f_2^*(\mu_0 + \phi^{-1}(\lambda))$. Hence, 
    \begin{equation} \label{APP:eqn:grad g = J.w}
         \nabla g(\lambda) = J_{\phi^{-1}}^T . \omega, \hspace{1.5cm} \forall \omega \in \partial f_2^*(\mu_0 + \phi^{-1}(\lambda))
    \end{equation}
    \item The function $g$ has an $L_g-$Lipschitz gradient implies that the Taylor-Lagrange inequality holds. That is: for any $\lambda, \Tilde{\lambda} \in \hat{\U} \subseteq \R^n$, we have: 
    \begin{equation} \label{APP:eqn:nabla g is lipschitz}
        g(\lambda) \leq g(\Tilde{\lambda}) + \left\langle \nabla g(\Tilde{\lambda}), \lambda - \Tilde{\lambda} \right\rangle + \frac{L_g}{2} \|\lambda - \Tilde{\lambda}\|^2
    \end{equation}
\end{itemize}
Now, for the points $a_2, \Tilde{a}_2 \in \dom f_2^*$, we define: 
\begin{align}
    \lambda := \phi(a_2 - \mu_0) && \Tilde{\lambda} := \phi(\Tilde{a}_2 - \mu_0)
\end{align}
Thus, by (\ref{APP:eqn: p1 et p2}, \ref{APP:eqn:grad g = J.w}, and \ref{APP:eqn:nabla g is lipschitz}), we obtain: 
\begin{equation}
    g(\lambda) \leq g(\Tilde{\lambda}) + \left\langle J^T_{\phi^{-1}} . p_2, \lambda - \Tilde{\lambda} \right\rangle + \frac{L_g}{2} \|\lambda - \Tilde{\lambda}\|^2 \\    
\end{equation}
Equivalently: 
\begin{equation} \label{APP:eqn:TLI -- f2*}
    f_2^*(a_2) \leq f_2^*(\Tilde{a}_2) + \left\langle p_2, J_{\phi^{-1}} (\lambda - \Tilde{\lambda} \right\rangle + \frac{L_g}{2} \|\lambda - \Tilde{\lambda}\|^2 
\end{equation}
Everything is ready now to be used to upper-bound the term: 
\begin{align*}
    &f(x) +  f^*(a) + \langle b, y \rangle \\
    & \stackrel{(\ref{APP:eqn: f* is separable})}{=} f(x) + {\color{blue} f_1^*(a_1)} + { \color{red} f_2^*(a_2)} + \langle b, y \rangle \\
    &\stackrel{(\ref{APP:eqn:f1* is lipschitz}, \ref{APP:eqn:TLI -- f2*})}{\leq} f(x) + {\color{blue} f_1^*(\Tilde{a}_1) + L_{f^*_1} \|a_1 - \Tilde{a}_1\|} + {\color{red} f_2^*(\Tilde{a}_2) + \left\langle p_2, J_{\phi^{-1}} (\lambda - \Tilde{\lambda} \right\rangle + \frac{L_g}{2} \|\lambda - \Tilde{\lambda}\|^2} + \langle b, y \rangle \\
    &\stackrel{(\ref{eqn:phi preserves the norm}, \ref{eqn:Jacobian of phi})}{=}  f(x) + f_1^*(\Tilde{a}_1) + L_{f^*_1} \|a_1 - \Tilde{a}_1\| + f_2^*(\Tilde{a}_2) + { \color{red} \left\langle p_2, a_2 - \Tilde{a}_2 \right\rangle} + {\color{red} \frac{L_g}{2} \|a_2 - \Tilde{a}_2\|^2} + \langle b, y \rangle  \\
    &= f(x) + {\color{blue} f_1^*(\Tilde{a}_1) + f_2^*(\Tilde{a}_2)}  + L_{f^*_1} \|a_1 - \Tilde{a}_1\| + \left\langle p_2, a_2 - \Tilde{a}_2 \right\rangle +  \frac{L_g}{2} \|a_2 - \Tilde{a}_2\|^2 + \langle b, y \rangle  \\ 
    &= \begin{multlined}[t] f(x) + {\color{blue} f^*(\Tilde{a})} + L_{f^*_1} \|a_1 - \Tilde{a}_1\| + \left\langle p_2, a_2 - \Tilde{a}_2 \right\rangle  + \frac{L_g}{2} \|a_2 - \Tilde{a}_2\|^2  {\color{blue} - \langle x, \Tilde{a} \rangle - \langle x, a - \Tilde{a} \rangle} \\ {\color{blue} +  \langle x, a + A^Ty \rangle - \langle Ax - b, y \rangle} \end{multlined} \\ 
    &= \begin{multlined}[t] {\color{blue} f(x) +  f^*(\Tilde{a}) - \langle x, \Tilde{a} \rangle } + L_{f^*_1} \|a_1 - \Tilde{a}_1\| + \left\langle p_2, a_2 - \Tilde{a}_2 \right\rangle + \frac{L_g}{2} \|a_2 - \Tilde{a}_2\|^2 {\color{red} -\langle x, a - \Tilde{a} \rangle } \\ + \langle x, a + A^Ty \rangle - \langle Ax - b, y \rangle  \end{multlined} \\ 
    &\stackrel{(\ref{PDG-SDG:Approx DG upper bound})}{\leq} \begin{multlined}[t] {\color{blue} \G(z) } + L_{f^*_1} \|a_1 - \Tilde{a}_1\| + \left\langle p_2 {\color{red} - x_2}, a_2 - \Tilde{a}_2 \right\rangle + \frac{L_g}{2} \|a_2 - \Tilde{a}_2\|^2   {\color{red} -  \langle x_1, a_1 - \Tilde{a}_1 \rangle} \\ + \langle x, a + A^Ty \rangle  - \langle Ax - b, y \rangle \end{multlined} \\    
    &\leq \begin{multlined}[t]  \G(z) + L_{f^*_1} \|a_1 - \Tilde{a}_1\| + \|x_1\|\|a_1 - \Tilde{a}_1 \| + {\color{blue} \left\|p_2 - x_2\right\|} \|a_2 - \Tilde{a}_2 \| + \frac{L_g}{2} \|a_2 - \Tilde{a}_2\|^2\\ + \|x\| \left\|a + A^Ty\right\| + \left\|Ax - b\right\|\|y\| \end{multlined} \\  
    &\stackrel{(\ref{SDG-PDG:Moreau's identity})}{=} \begin{multlined}[t]  \G(z) + \left( L_{f^*_1} + {\color{blue} \|x_1\|} \right) \|a_1 - \Tilde{a}_1\| + {\color{blue} \frac{1}{\beta_x} \|p_2^*\|}\|a_2 - \Tilde{a}_2 \|  + \frac{L_g}{2} \|a_2 - \Tilde{a}_2\|^2 \\ + \|x\|\left\|a + A^Ty\right\|  + \left\|Ax - b\right\|\|y\| \end{multlined} \\ 
    &\leq \begin{multlined}[t]  \G(z) + \left( L_{f^*_1} + {\color{blue} \|x\|} \right) \|a_1 - \Tilde{a}_1\| + {\color{blue} \frac{1}{\beta_x} \|p^*\|}\|a_2 - \Tilde{a}_2 \|  + \frac{L_g}{2} \|a_2 - \Tilde{a}_2\|^2 + \|x\|{\color{red} \left\|a + A^Ty\right\| } \\ + \left\|Ax - b\right\|\|y\| \end{multlined} \\ 
    &\stackrel{(\ref{PDG-SDG:proj -- norms})}{\leq} \begin{multlined}[t]  \G(z) + \left( L_{f^*_1} + \|x\| \right) {\color{blue}  \|a_1 - \Tilde{a}_1\|} + \frac{1}{\beta_x} \|p^*\| {\color{blue} \|a_2 - \Tilde{a}_2 \|}  + \frac{L_g}{2} {\color{blue} \|a_2 - \Tilde{a}_2\|^2 }  + \|x\|{\color{red} \|p^*\|}  \\ + \left\|Ax - b\right\|\|y\| \end{multlined} \\ 
    &\stackrel{(\ref{PDG-SDG:a, a tilde, p*})}{\leq} \begin{multlined}[t]  \G(z) + \left( L_{f^*_1} +  \|x\|\right) {\color{blue} \|p^*\|} + \frac{1}{\beta_x} {\color{blue} \|p^*\|^2} + \frac{L_g}{2} {\color{blue} \|p^*\|^2} + \|x\|  {\color{blue} \|p^*\|} + \left\|Ax - b\right\| \|y\| \end{multlined} \\ 
    &\stackrel{(\ref{eqn:p*<G})}{\leq} \begin{multlined}[t]  \G(z) + \left( L_{f^*_1} + \|x\|\right) {\color{blue} \sqrt{2\beta_x\G(z)}}  + {\color{blue} 2\G(z) + \beta_xL_g \G(z)} + \|x\|  {\color{blue} \sqrt{2\beta_x\G(z)}}  \\  + {\color{red} \left\|Ax - b\right\|} \|y\|\end{multlined} \\ 
    &\stackrel{(\ref{eqn:FG-SDG})}{\leq} \left(3 + \beta_xL_g\right) \G(z) + \left( L_{f^*_1} + 2\|x\|\right) \sqrt{2\beta_x\G(z)} + \|y\|{\color{red} \sqrt{2\beta_y\G(z)}}
\end{align*}
Consequently, we have derived an upper bound for the term $f(x) + f^*(a) + \langle b, y \rangle$. Thanks to Lemma \ref{lem:PDG-SDG:DG lower bound} that provides us with a lower bound of that term as well. By combining these two bounds, we obtain:
\begin{equation}\label{PDG-SDG-manifold:DG bound}
    |f(x) + f^*(a) + \langle b, y\rangle| \leq \left(3 + \beta_x L_g\right)\G(z) + \left(\sqrt{2\beta_x}\left(2\|x\| + L_{f_1^*}\right)+ \sqrt{2\beta_y}\|y\|\right)\sqrt{\G(z)}
\end{equation}
Lastly, Lemma \ref{lem:PDG-SDG:primal-dual feas bound} along with this last bound (\ref{PDG-SDG-manifold:DG bound}) conclude the proof: 
\begin{align*}
    \D(z) &=  {\color{red} |f(x) + f^*(a) + \langle b, y\rangle|^2} + \|Ax - b\|^2 + \left\|a + A^Ty\right\|^2 \\ 
    &\stackrel{(\ref{PDG-SDG-manifold:DG bound})}{\leq} {\color{red} \left(\G(z) + \left(\sqrt{2\beta_x}\left(\|x\| + L_{f^*}\right) + \sqrt{2\beta_y}\|y\|\right)\sqrt{\G(z)}\right)^2 } + {\color{blue} \|Ax - b\|^2 + \left\|a + A^Ty\right\|^2} \\ 
    &\stackrel{(\ref{PDG-SDG:primal-dual feas bound})}{\leq} \left(\G(z) + \left(\sqrt{2\beta_x}\left(\|x\| + L_{f^*}\right) + \sqrt{2\beta_y}\|y\|\right)\sqrt{\G(z)}\right)^2 + {\color{blue} 2\beta_{\max} \G(z)} \smartqed 
\end{align*}
\end{proof}
Upon setting $f_2^* = 0$ in our preceding theorem, we managed to derive a slightly tighter bound. The detailed expression is presented below.
\begin{proposition} \label{prop:PDG-SDG}
    Given $\beta = (\beta_x, \beta_y) \in (0, +\infty)^2$, $z \in \Z$, and a function $f \in \Gamma_0(\X)$ such that its Fenchel-conjugate, $f^*$, is $L_{f^*}-$Lipschitz on its domain. Let $\beta_{\max} = \max(\beta_x, \beta_y)$. Then, for the \acrlong{pdg} and the \acrlong{sdg} defined, respectively, in (\ref{eqn:PDG}) and (\ref{eqn:SDG}) we have: 
    \begin{equation}\label{eqn:PDG-SDG}
    \begin{split}
        \D(z) &\leq \begin{multlined}\left(\G(z) + \left(\sqrt{2\beta_x}\left(\|x\| + L_{f^*}\right) + \sqrt{2\beta_y}\|y\|\right)\sqrt{\G(z)}\right)^2  + 2\beta_{\max} \G(z)\end{multlined} \\ 
        a &= \proj_{\dom f^*} \left(-A^Ty\right)  \hspace{1.85cm} p= \prox_{\beta_x^{-1}f}\left(x - \frac{1}{\beta_x} A^Ty\right) 
    \end{split}
    \end{equation}
\end{proposition}
\begin{proof}
The main contrast between the proof of Theorem \ref{thm:PDG-SDG -- manifold} and this one lies in the upper bound established for the term $f(x) + f^*(a) + \langle b, y \rangle$. However, the rest of the proof remains unchanged. Since $f^*$ is $L_{f^*}-$Lipschitz on its domain, then: 
\begin{equation} \label{PDG-SDG:f* Lipschitz}
    f^*(a) \leq L_{f^*} \|a - \Tilde{a}\| + f^*(\Tilde{a})
\end{equation}
Therefore, 
\begin{align*}
    f(x) + {\color{red} f^*(a)} + \langle b, y \rangle &\stackrel{(\ref{PDG-SDG:f* Lipschitz})}{\leq} f(x) {\color{red} +  f^*(\Tilde{a}) + L_{f^*} \|a - \Tilde{a} \| } +  \langle b, y \rangle \\
    &= {\color{blue} f(x) + f^*(\Tilde{a}) - \langle x, \Tilde{a} \rangle } - \left\langle Ax - b, y \right\rangle + \left\langle x, \Tilde{a} + A^Ty \right\rangle  + L_{f^*} \|a - \Tilde{a} \| \\
    &\stackrel{(\ref{PDG-SDG:Approx DG upper bound})}{\leq} {\color{blue} \G(z)} - \left\langle Ax - b, y \right\rangle + \left\langle x, \Tilde{a} + A^Ty \right\rangle + L_{f^*} \|a - \Tilde{a} \| \\ 
    &\leq \G(z) + \left\|Ax - b\right\| \|y\| + \|x\| {\color{red} \left\|\Tilde{a} + A^Ty\right\| } + L_{f^*} \|a - \Tilde{a}\| \\ 
    &\stackrel{(\ref{PDG-SDG:proj -- norms})}{\leq} \G(z) + \left\|Ax - b\right\| \|y\| + \|x\| {\color{red} \left\|p^*\right\|} + L_{f^*} \|a - \Tilde{a}\| \\ 
    &\leq  \G(z) + \|y\| \sqrt{2\beta_y \G(z)} + \|x\| \sqrt{2\beta_x \G(z)} + L_{f^*} \sqrt{2\beta_x \G(z)} \stepcounter{equation}\tag{\theequation}\label{PDG-SDG:DG upper bound}
\end{align*}
where in the last line we utilized the upper bounds (\ref{eqn:FG-SDG}, \ref{eqn:p*<G}, \ref{PDG-SDG:a, a tilde, p*}), respectively.
Now, from this upper bound (\ref{PDG-SDG:DG upper bound}) and the earlier-proven lower-bound (\ref{PDG-SDG:DG lower bound}), we get: 
\begin{equation} \label{PDG-SDG:DG bound}
    |f(x) + f^*(a) + \langle b, y\rangle| \leq \G(z) + \left(\sqrt{2\beta_x}\left(\|x\| + L_{f^*}\right) + \sqrt{2\beta_y}\|y\|\right)\sqrt{\G(z)}
\end{equation}
Therefore, Lemma \ref{lem:PDG-SDG:primal-dual feas bound} along with this last bound (\ref{PDG-SDG:DG bound}) conclude the proof: 
\begin{align*}
    \D(z) &=  {\color{red} |f(x) + f^*(a) + \langle b, y\rangle|^2} + {\color{blue} \|Ax - b\|^2 + \left\|a + A^Ty\right\|^2} \\ 
    &\leq {\color{red} \left( \left(3 + \beta_x L_g\right)\G(z) + \left(\sqrt{2\beta_x}\left(2\|x\| + L_{f_1^*}\right)+ \sqrt{2\beta_y}\|y\|\right)\sqrt{\G(z)} \right)^2 } + {\color{blue} 2\beta_{\max} \G(z)} \smartqed 
\end{align*}
\end{proof}